\documentclass[11pt]{article}
\usepackage{amsmath, amssymb,  amsthm,color}
\usepackage{graphicx} 
\usepackage{enumitem}                                 
\usepackage[margin=1.0in]{geometry}

\AtEndDocument{\bigskip{\footnotesize%
\textsc{Department of Mathematics, the Pennsylvania State University, University
Park, PA 16802, USA} \par 
  \textit{E-mail address} \texttt{ aka5983@psu.edu
} \par
\textit{E-mail address} \texttt{ hertz@math.psu.edu
} \par 
\textit{E-mail address} \texttt{ zhirenw@psu.edu
}
}}

\newtheorem{theorem}{Theorem}[section]
\newtheorem{Remark} [theorem]{Remark}

\newtheorem{Counter-example}[theorem]{Counter example}
\newtheorem{Claim}[theorem]{Claim}

\newtheorem{Lemma}[theorem]{Lemma}
\newtheorem{Proposition}[theorem]{Proposition}
\newtheorem{Notation}[theorem]{Notation}
\newtheorem{Definition}[theorem]{Definition}
\newtheorem{Corollary}[theorem]{Corollary}

\newtheorem*{theorem*}{Theorem}

\newcommand{\hide}[1]{}

\DeclareMathOperator*{\sign}{sign}

\newcommand{\supp}{\text{supp}}
\newcommand{\diam}{\text{diam}}

\newcommand{\ttau}{{\tilde{\tau}}}
\newcommand{\tS}{{\widetilde{S}}}

\newcommand{\cA}{{\mathcal A}}

\newcommand{\tA}{{\widetilde A}}

\newcommand{\tZ}{{\widetilde Z}}

\newcommand{\bfp}{{\mathbf{p}}}

\newcommand{\bP}{{\mathbb{P}}}

\title{Pointwise  normality and Fourier decay for self-conformal measures}
\author{Amir Algom, Federico Rodriguez Hertz, and Zhiren Wang}
\date{}
\begin{document}
\maketitle
\begin{abstract}
Let $\Phi$ be a  $C^{1+\gamma}$ smooth IFS on  $\mathbb{R}$, where $\gamma>0$. We provide mild conditions on the derivative cocycle  that ensure that every self conformal measure  is supported on points $x$ that are absolutely normal. That is, for every integer $p\geq 2$ the sequence $\lbrace p^k x \rbrace_{k\in \mathbb{N}}$ equidistributes modulo $1$. We thus extend several state of the art results of Hochman and Shmerkin \cite{hochmanshmerkin2015} about the prevalence of normal numbers in fractals.   When $\Phi$ is self-similar we show that  the set of absolutely normal numbers has full Hausdorff dimension in its attractor, unless $\Phi$ has an explicit structure that is associated with some integer $n\geq 2$. These conditions  on the derivative cocycle are also shown to imply that every self conformal measure is a Rajchman measure, that is, its Fourier transform decays to $0$ at infinity.  When $\Phi$ is self similar and satisfies a certain Diophantine condition, we establish a logarithmic rate of decay.
\end{abstract}

\section{Introduction}
\subsection{Background} \label{Section results}
Let $p$ be an integer greater or equal to $2$. Let $T_p$ be the times-$p$ map, 
$$T_p (x) = p\cdot x \mod 1, \, x\in \mathbb{R}.$$
A number $x\in \mathbb{R}$ is called $p$-normal, or normal to base $p$, if its orbit $\lbrace T_p ^k (x) \rbrace_{k\in \mathbb{N}}$  equidistributes for the Lebesgue measure on $[0,1]$. We call $x$ \textit{absolutely normal} if it is $p$-normal  for all integers $p\geq 2$.  In 1909 Borel proved  that Lebesgue almost every $x$ is absolutely normal. It is believed that this phenomenon should continue to hold true for typical elements of well structured sets with respect to appropriate measures, in the absence of obvious obstructions. Thus, we will call a Borel probability measure $\nu$ on $\mathbb{R}$  \textit{pointwise absolutely normal} if $\nu$ almost every $x$ is absolutely normal. One of the main purposes of this paper is to specify a large and natural class of fractal measures that are pointwise absolutely normal. Additionally, we will indicate a large family of fractal sets that are typically Lebesgue null, such that set of absolutely normal numbers intersects them with full Hausdorff dimension. 

There are two known general techniques to study whether a given Borel probability measure $\mu$ on $\mathbb{R}$ is supported on numbers normal to a given base $p$: The  first method involves establishing sufficiently fast decay of the $L^2 (\mu)$ norms of certain trigonometric polynomials as in Weyl's criterion. This method was famously used by  Cassels and Schmidt \cite{Schmidt1960normal, Cassels1960normal} independently to show that if $\mu$ is the Cantor-Lebesgue measure on the middle$-\frac{1}{3}$ Cantor set, then $\mu$ almost every $x$ is $p$-normal whenever $p$ is independent of $3$, that is, $\frac{\log p}{\log 3} \notin \mathbb{Q}$. Henceforth, we will write $a \not \sim b$ to indicate that $a$ and $b$ are indepedent, and $a\sim b$ otherwise. A virtually optimal criteria for a measure  to be supported on numbers that are $p$-normal in terms of these $L^2$ norms was later formulated by Davenport-Erd\H{o}s-LeVeque \cite{Davenport1964Erdos}, and was used by several authors including  Brown, Pearce, Pollington, and Moran \cite{Pollington1988normal,  Brown1985moran, Brown1987moran}. An excellent exposition to this subject is given in Bugeaud's book \cite{Bugeaud2012book}.

The Davenport-Erd\H{o}s-LeVeque criteria is closely related to the decay rate of the Fourier transform of Borel probability measures on $\mathbb{R}$: Let $\nu$ be such a measure.  For every $q\in \mathbb{R}$   the Fourier transform of $\nu$ at $q$ is defined by 
\begin{equation*} 
\mathcal{F}_q (\nu) := \int \exp( 2\pi i q x) d\nu(x).
\end{equation*}
The measure $\nu$ is called a \textit{Rajchman measure} if $\lim_{|q|\rightarrow \infty} \mathcal{F}_q(\nu)=0$. By the Riemann-Lebesgue Lemma, if $\nu$ is absolutely continuous then it is Rajchman. On the other hand, by Wiener's Lemma  if $\nu$ has an atom then it is not  Rajchman. For measures that are both continuous (no atoms) and singular, determining whether or not $\nu$ is a Rajchman measure may be a challenging problem. The Rajchman property and, when available,  further information about   the rate of decay of $\mathcal{F}_q(\nu)$,  have various consequences on the geometry of $\nu$ \cite{Lyons1995survey}. Returning to the Davenport-Erd\H{o}s-LeVeque Theorem, it  ensures that if  e.g. there is some $\alpha=\alpha(\nu)>0$ such that
$$ \left|\mathcal{F}_q (\nu)\right| \leq O\left( \frac{1}{ \left| \log \log |q| \right| ^{1+\alpha}} \right), \text{ as } |q|\rightarrow \infty$$
then $\nu$ is  pointwise absolutely normal. We note, however, that such bounds are usually hard to obtain (if they are true at all) in concrete situations, even for naturally defined measures. The third objective of this paper, which arises in conjuncture with the number theoretic framework discussed above, is to establish the Rajchman property for a  class of dynamically defined measures.

In 2015 Hochman and Shmerkin \cite{hochmanshmerkin2015} introduced a new method, giving a   fractal geometric condition that is sufficient for a measure  to be supported on $p$-normal numbers. This condition applies to a wide class of measures that arise from some dynamical or arithmetic origin. One of the virtues of this method is that it can be used regardless of knowledge on the behaviour of the Fourier transform of the underlying measure. Instead, one needs to understand the so-called scenery of the measure at typical points \cite[Section 1.2]{hochmanshmerkin2015}. In order to compute the scenery in specific examples, one usually works with measures such that their small "pieces"  have mild (or no) overlaps, and this computation can become difficult in the presence of complicated overlaps.  Many of the results of Hochman-Shmerkin still remain the state of the art on the subject, and we will recall them as we compare them to our results.

In this paper we introduce a new dynamical condition for a self-conformal measure (defined below) to be both a Rajchman measure \textit{and} pointwise absolutely normal. A rate of decay is only established in some special cases, so in general we cannot invoke  Davenport-Erd\H{o}s-LeVeque to get absolute normality. Thus, we will prove pointwise absolute normality \textit{directly}, regardless of the decay rate of the Fourier transform.  This provides many new examples of both Rajchman measures  and of pointwise absolutely normal measures, and allows us to extend  results\footnote{Hochman and Shmerkin can work with more general $\beta$ transformation, that is, maps of the form $x\mapsto \beta x \mod 1$ for certain $\beta>1$. We will compare this to our method in Section \ref{Section proof sketch}.} of Hochman-Shmerkin, as detailed below. We then proceed to show that  self similar sets intersect the set of absolutely normal numbers with full Hausdorff dimension, unless an obvious obstruction is present.

Self-conformal measures are defined as follows: Let $\Phi= \lbrace f_1,...,f_n \rbrace$ be a finite set of strict contractions of a compact interval $I\subseteq \mathbb{R}$ (an \textit{IFS}), such that every $f_i$ is differentiable. We say that  $\Phi$ is $C^\alpha$ smooth if every $f_i$ is at least $C^\alpha$ smooth for some $\alpha\geq 1$.  It is well known that there exists a unique compact set $\emptyset \neq K=K_\Phi \subseteq I$ such that
$$ K = \bigcup_{i=1} ^n f_i (K).$$
The set $K$ is called a \textit{self-conformal set}, and the \textit{attractor} of the IFS $\lbrace f_1,...,f_n \rbrace$. In the special case where each $f_i$ is affine, i.e. $f_i(x)=r_i \cdot x+t_i$ and $0<|r_i|<1$, we call $K$ a \textit{self-similar set}. We always assume that there exist $i\neq j$ such that the fixed point of $f_i$ is not equal to the fixed point of $f_j$. This ensures that $K$ is infinite. We call $\Phi$  \textit{uniformly contracting} if 
$$0< \inf \lbrace |f '(x)|:\, f\in \Phi, x\in I \rbrace \leq \sup \lbrace |f '(x)| :\, f\in \Phi, x\in I  \rbrace <1.$$
Finally, following Hochman-Shmerkin \cite{hochmanshmerkin2015}, we say that $\Phi$ is \textit{regular} 
if it is uniformly contracting, and the intervals $f_i(I)$ are disjoint except possibly at their endpoints (so that, in particular, the so-called  \textit{open set condition} is satisfied).

 Next, let $\textbf{p}=(p_1,...,p_n)$ be a strictly positive probability vector, that is, $p_i >0$ for all $i$ and $\sum_i p_i =1$. It is well known that there exists a unique Borel probability  measure $\nu$ such that
$$\nu = \sum_{i=1} ^n p_i\cdot  f_i\nu,\quad \text{ where } f_i \nu \text{ is the push-forward of } \nu \text{ via } f_i.$$
The measure $\nu$ is called a \textit{self-conformal measure}, and is supported on $K$. If every $f_i$ is affine then $\nu$ is called a \textit{self-similar measure}.   Since $K$ is always assumed to infinite, the assumption that $p_i>0$ for every $i$ implies that $\nu$ is non-atomic. In particular, all self-conformal measures in this paper are  non-atomic.

\subsection{Pointwise  normality and Fourier decay for self conformal measures}
\subsubsection{Main technical Theorem} \label{Section sketch}
We first formulate a general condition that ensures a given self conformal measure is both Rajchman and pointwise absolutely normal. Let $\Phi=\lbrace f_1,...,f_n \rbrace$ be an IFS on an interval $I$ such that each $f_i$ is differentiable.  For every $\omega \in \lbrace 1,...,n\rbrace ^\mathbb{N}$ and $m\in \mathbb{N}$ let
$$f_{\omega|_m} = f_{\omega_1} \circ \circ \circ f_{\omega_m}.$$
Fix $x_0 \in I$.  Then we have a surjective coding map $\lbrace 1,...,n\rbrace ^\mathbb{N} \rightarrow K$ defined by
$$\omega \in \lbrace 1,...,n \rbrace^{\mathbb N} \mapsto x_\omega:= \lim_{m\rightarrow \infty}  f_{\omega|_m}  (x_0).$$
Assuming the IFS is uniformly contracting and $C^{1+\gamma}$ smooth, let $\rho:= \left( \sup_{f\in \Phi} ||f'||_\infty \right)^{\gamma} \in (0,1)$, and  define a metric on $\lbrace 1,...,n\rbrace^{\mathbb N}$ via 
$$d_\rho (\omega,\omega'):=\rho^{\min\{n:\ \omega_n\neq\omega'_n\}}.$$

Let $\textbf{p}$ be a strictly positive probability vector on $\lbrace 1,...,n\rbrace$, and let $\nu$ be the corresponding self conformal measure. Let $\mathbb{P}=\textbf{p}^\mathbb{N}$ be the product measure on $\lbrace 1,...,n\rbrace ^\mathbb{N}$. Then $\nu$ is the push-forward of $\mathbb{P}$ via $\omega\mapsto x_\omega$. Also, for every $1\leq a\leq n$ let $\iota_a: \lbrace 1,...,n\rbrace^{\mathbb N}\to \lbrace 1,...,n\rbrace ^{\mathbb N}$ be the map 
$$\iota_a(\omega_1,\omega_2,\cdots)=(a,\omega_1,\omega_2,\cdots).$$
Let  $G$ to be the free semigroup generated by the family $\{\iota_a: 1\leq a\leq n\}$.  We define the derivative cocycle $c:G\times \lbrace 1,...,n\rbrace^\mathbb{N}\rightarrow \mathbb{R}$ via
 $$c(a,\omega)=-\log|f'_a(x_\omega)|.$$
Choose some $\kappa \in (0,1]$ and let $H^{\kappa}$ denote the space of $\kappa$-H\"older continuous maps $\lbrace 1,....,n\rbrace^\mathbb{N}\rightarrow \mathbb{C}$, and define $\Lambda_c \subseteq \mathbb{R}$ via
\begin{eqnarray} \label{Eq Lambda c}
\Lambda_c &=& \lbrace \theta: \text{ There exists } \phi_\theta \in H^{\kappa} \text{ with } |\phi_\theta|=1 \text{ and } u_\theta \in S^1 \text{ such that }  \\
&&  \phi_\theta(\iota_a(\omega)) = u_\theta \cdot  e^{-i \theta \cdot c(a,\omega)}\cdot \phi_\theta (\omega),\quad \text{ for all } (a,\omega)\in \lbrace 1,...,n \rbrace\times \lbrace 1,...,n \rbrace^\mathbb{N} \rbrace. \nonumber
\end{eqnarray}
It is clear that $0\in \Lambda_c$. If $\Lambda_c = \lbrace 0 \rbrace$ then the cocycle $c$ is aperiodic in the sense of Benoist-Quint  \cite[Equation (15.8)]{Benoist2016Quint}, and this will be used in an essential way to prove the following Theorem:
\begin{theorem} \label{Theorem main tech}
Let $\Phi$ be a uniformly contracting $C^{1+\gamma}$ smooth IFS  for some $\gamma>0$, and let $\nu$ be a self conformal measure. If $\Lambda_c = \lbrace 0 \rbrace$ then:
\begin{enumerate}
\item  $\nu$ is a Rajchman measure, that is,  $ \lim_{|q|\rightarrow \infty} \mathcal{F}_q (\nu) =0$.

\item $\nu$ is pointwise absolutely normal.
\end{enumerate}
\end{theorem}

Note that the conditions of Theorem \ref{Theorem main tech} are invariant under conjugation by $C^{1+\gamma}$ maps with non-vanishing derivative, a useful feature in applications.  Before  turning to these applications, we say a few words about what goes into the proof: For part (1), the most important ingredient is Theorem \ref{Theorem equid}, a  conditional local limit Theorem  for  certain random variables resembeling stopping times that are related to a random walk driven by the derivative cocycle. That is, this local equidistrubtion property holds up to conditioning on "good" cells of  suitable partitions of the space $\lbrace 1,...,n\rbrace ^\mathbb{N}$. These good  cells are produced via a central limit theorem for cocycles, and the local equidistribution  follows from a local limit theorem for aperiodic cocycles, both proved by Benoist-Quint \cite[Theorem 12.1 and Theorem 16.15]{Benoist2016Quint}. This is the only part of the proof where the assumption $\Lambda_c = \lbrace 0 \rbrace$ is used. The rest of the proof consists of subtle linerization arguments and an adaptation of a Lemma of Hochman \cite{Hochman2020Host}, regarding oscillatory integrals that arise at the end of the proof. See Section \ref{Section proof 1} for more details.

For part (2), fixing a $\nu$ typical point $x$ and an integer base $p$, we first  employ a martingale argument in the spirit of Hochman-Shmerkin \cite[Theorem 2.1]{hochmanshmerkin2015} and a further linearization step. This reduces the problem to that of proving the Rajchman property  \textit{with the same rate} for a countable family of measure that arise as $T_p ^n$-magnifications of increasingly small pieces of the measure about the point $x$. The treatment of the Fourier transform of these measures relies on a similar scheme as in part (1), but the additional steps further complicate the already delicate analysis involved. This necessitates the introduction of several new ideas. The most important one is Theorem \ref{Theorem lin equid}, another conditional local limit Theorem  that is tailored to this situation.  See Section \ref{Section proof 2} for more details.

\subsubsection{Applications and related results} \label{Section app}
We proceed to describe some applications of Theorem \ref{Theorem main tech}. Before doing so, we introduce two new definitions: Let $\Phi$ be a self similar IFS  with corresponding contraction ratios $\lbrace r_1,...,r_n \rbrace$. We say that $\Phi$ is \textit{periodic} if there exists  some $r \in \mathbb{R}$ such that
$$\lbrace \log |r_1|,...,\log |r_n| \rbrace \subset   r\mathbb{Z}$$
otherwise, we say that $\Phi$ is \textit{aperiodic}. Note that $\Phi$ is aperiodic if and only if there are $i\neq j$ such that $\frac{\log |r_i|}{\log |r_j|} \notin \mathbb{Q}$.  We call $\Phi$  \textit{Diophantine} if there are $l , C>0$ such that 
\begin{equation} \label{Eq Dio condition}
\inf_{y\in \mathbb{R}} \max_{i\in \lbrace 1,...n\rbrace } d(\,\log |r_i|\cdot x+y,\, \mathbb{Z}) \geq \frac{C}{|x|^{l}}, \text{ for all } x\in \mathbb{R} \text{ large enough in absolute value.}
\end{equation}
Notice that if $\Phi$ is Diophantine then it is aperiodic, but the converse is false in general. This Diophantine condition is adopted from Breuillard's work \cite[Section 3.1]{Breuillard2005llt}. It is generic in the sense that it holds if $n\geq 3$ and we draw $\lbrace \log |r_1|,...,\log |r_n| \rbrace$  according to the Lebesgue measure on $\mathbb{R}^n$ \cite[Proposition 2.4]{Weak2017Cramer}, and is met if e.g.  $\log |r_1|,...,\log |r_n|$ are rationally independent algebraic numbers \cite[Proposition 2.7]{Weak2017Cramer}. See also  Section \ref{Section Moser} for a family of examples related to the work of Moser \cite{Moser1990Dio}. We are now ready to  summarize some corollaries of Theorem \ref{Theorem main tech}:

\begin{Corollary} \label{Main Corollary}
Let $\Phi$ be a uniformly contracting $C^{1+\gamma}$ smooth IFS for some $\gamma>0$, and let $\nu$ be a self conformal measure.
\begin{enumerate}
\item Suppose that for every $t,r \in \mathbb{R}$, the set 
$$\left\lbrace \log \left| f ' \left( y \right) \right| : \text{ where } f (y)=y,\quad f\in \Phi   \right\rbrace$$
is not included in the set $t+r \mathbb{Z}$. Then $\nu$ is both pointwise absolutely normal and Rajchman. 

\item Suppose that $\Phi$ is an aperiodic self similar IFS (so $\nu$ is a self similar measure), and let $g:I\rightarrow \mathbb{R}$ be a $C^{1+\gamma}$ smooth map, where $\gamma>0$. If  $g'$ does not vanish then  $g\nu$ is both pointwise absolutely normal and Rajchman.  \newline
If $\Phi$ is Diophantine then there exists some $\alpha=\alpha(\nu)>0$ such that
$$ \left|\mathcal{F}_q (\nu)\right| \leq O\left( \frac{1}{ \left| \log |q| \right| ^\alpha} \right),\, \text{ as } |q|\rightarrow \infty.$$

\item  Suppose that  $\Phi$ is $C^\omega$ smooth, or  that it is $C^2$ and  $K_\Phi$ is an interval. If $\nu$ is not pointwise absolutely normal or not Rajchman  then $\Phi$  is  $C^\omega$ or $C^{2}$ (depending on the smoothness of $\Phi$) conjugate to a periodic self similar IFS.
\end{enumerate}
\end{Corollary}
We emphasize that no separation condition is imposed on $\Phi$. The  deduction of Corollary \ref{Main Corollary} from Theorem \ref{Theorem main tech} relies on the study of the derivative cocycle under these assumptions, and this  analysis is carried out in  Section \ref{Section coro}.

Part (1) of Corollary \ref{Main Corollary} should be compared with a result  \cite[Theorem 1.4]{hochmanshmerkin2015} of Hochman-Shmerkin: For a given $p$, if 
$$\Phi \text{ is regular and some element of } \left\lbrace  f ' \left( y \right) : \text{ where } f (y)=y,\quad f\in \Phi   \right\rbrace \text{ is independent of } p$$
then $\nu$ almost every point is $p$-normal for every self conformal measure $\nu$. Notice that the assumption in part (1) implies this arithmetic condition holds for every $p$. So, in its setting,  part (1) extends the result of Hochman-Shmerkin by removing the  separation assumption from the IFS.  Furthermore, \cite[Theorem 1.4]{hochmanshmerkin2015} remains true when pushing the measure $\nu$ forward via any real diffeomorphism $g$. Corollary \ref{Main Corollary} part (1), in contrast, remains true when pushing the measure forward via a $C^{1+\gamma}$ diffeomorphism, but we don't know if  this is true for $C^1$ diffeomorphisms. Finally, we remark that \cite[Theorem 1.4]{hochmanshmerkin2015} holds for a more general class of transformations and measures, and we will discuss this in Section \ref{Section proof sketch}.

The normality assertion of part (2) should be compared with another result \cite[Theorem 1.7]{hochmanshmerkin2015} of Hochman-Shmerkin, where they prove that for every self similar measure with respect to a regular self similar IFS, its push-forward under a non-affine $C^\omega$ map is pointwise absolutely normal. So, \cite[Theorem 1.7]{hochmanshmerkin2015} does not require the IFS to be aperiodic, but does require a more restrictive  regularity (separation) assumption. In addition, the smoothness assumption on the perturbing map in  part (2) is less restrictive than \cite[Theorem 1.7]{hochmanshmerkin2015}.  For example, let $\nu$ be any self conformal measure with respect to the aperiodic self similar IFS
$$\lbrace \frac{x}{2}, \frac{x+1}{3}, \frac{x+1}{5} \rbrace.$$
Corollary \ref{Main Corollary} part (2) implies that  $f(x)$ is absolutely normal for $\nu$ almost every $x$ and any diffeomorphism $f\in C^{1+\gamma} (\mathbb{R})$. Notice that since the IFS is not regular and the  map $f$ might not be $C^\omega$  \cite[Theorem 1.7]{hochmanshmerkin2015} does not apply in this situation.  On the other hand, consider  the Cantor-Lebesgue measure $\mu$ on the middle$-\frac{1}{3}$ Cantor set. By \cite[Theorem 1.7]{hochmanshmerkin2015}   $x^2$ is  absolutely normal for $\mu$ almost every $x$. Notice that since the underlying regular IFS 
$$\lbrace \frac{x}{3}, \frac{x+2}{3} \rbrace$$
is not aperiodic, Corollary \ref{Main Corollary} part (2) does not apply here.

The Rajchman assertion of part (2)  should be compared with a recent Theorem of Li-Sahlsten \cite[Theorem 1.2]{li2019trigonometric} (see also \cite{bremont2019rajchman}): They  proved that for an orientation preseving aperiodic self similar IFS  any  self-similar measure is a Rajchman measure.  So,  Part (2) extends the Li-Sahlsten Theorem to all $C^{1+\gamma}$ smooth images.   Corollary \ref{Main Corollary} part (2) also complements a classical Theorem of Kaufman \cite{Kaufman1984ber} that was later extended by Mosquera-Shmerkin   \cite{Shmerkin2018mos} (see also \cite{chang2017fourier}) about polynomial Fourier decay  for $C^2$ images of homogeneous (i.e. $r_i=r_j$ for all $i,j$) self similar measures. We thus partially answer a folklore open question (see e.g. \cite{sahlsten2020fourier}) about the existence of a Kaufman Theorem in the non homogeneous setting.

 The quantitative  assertion of part (2) should be compared with another Theorem of Li-Sahlsten \cite[Theorem 1.3]{li2019trigonometric} where a similar logarithmic decay rate was obtained, but under a different Diophantine assumption. In Section \ref{Section Moser} we will give a family of Diophantine IFS's  that do not satisfy the conditions of \cite[Theorem 1.3]{li2019trigonometric}  since $\frac{\log |r_i|}{\log |r_j|}$ is either rational or a Liouville number for all $i,j$. Thus, via Corollary \ref{Main Corollary} part (2) we obtain many new examples of self similar measures with logarithmic Fourier decay.  Finally,  we will discuss the problem of getting an effective decay rate for non-linear IFS's in Section \ref{Section proof sketch}.

The Rajchman question for self similar measures is a classical problem that has received much attention over the years: Consider, for example, the family of Bernoulli convolutions $\lbrace \nu_r \rbrace_{r\in (0,1)} $: For every $0<r<1$ we define the self similar IFS $\lbrace r\cdot x -1, r \cdot x+1 \rbrace$ with the probability vector $\textbf{p}=(\frac{1}{2},\frac{1}{2})$. It is a fundemental problem to determine for which $r\in (\frac{1}{2},1)$ is $\nu_r$ is absolutely continuous.  A celebrated result of Erd\H{o}s \cite{Erdos1939ber} says that if $r^{-1}$ is a \textit{Pisot} number then $\nu_r$ is not a Rajchman measure and consequently is not absolutely continuous. Recall that a Pisot number is a real algebraic integer greater than one whose Galois conjugates all lie  inside the unit disc. Later, Salem \cite{Salem1943uniq}  completed the picture in terms of the Rajchman property, by showing that $\nu_r$ is not Rajchman only if $r^{-1}$ is Pisot (see also the related works of Piatetski-Shapiro \cite{Shapir1952uniq} and Salem and Zygmund \cite{Salem1955zyg}) . We remark that through some ground breaking recent papers (e.g. \cite{hochman2014self}, \cite{shmerkin2016furstenberg}, \cite{Bre2019Varju}, \cite{Varju2010dim} to name a few) the geometric properties of $\nu_r$ are now  far better understood. However, the question of absolute continuity remains open. More general self similar measures were studied by  Strichartz \cite{Stri1990self}, \cite{Stri1993self}: He proved that their Fourier transforms decay on average, with a recent large deviations estimate on this decay given by Tsujii \cite{Tsujii2015self} (see also \cite{Stri1993bself} for a related paper about self conformal measures). However, these papers do not establish the Rajchman property, since they exclude certain frequencies. 

Very recently, as  we have already mentioned, Li-Sahlsten \cite{li2019trigonometric} proved the Rajchman property in the presence of indepedent contractions.   In the complementary case, when all contractions are powers of some $r\in (0,1)$, Br\'{e}mont \cite{bremont2019rajchman} proved that  a  self-similar measure can fail to be Rajchman only if $r^{-1}$ is Pisot. In fact, Br\'{e}mont fully characterised the IFS up to affine conjugation, and this will play a crucial role later in this paper. Another proof of this fact was given by Varj\'{u}-Yu \cite{varju2020fourier}.   Finally, we note that Li and Sahlsten \cite{Li2020Sahl} also generalized their results to self affine measures in higher dimensions.

The problem of quantitative Fourier decay for self similar measures is also a classical one: For Bernoulli convolutions, it follows from the  works of Erd\H{o}s \cite{Erdos1940ber} and Kahane \cite{Kahane1971Ber} that $\nu_r$ has polynomial  decay outside a set of zero Hausdorff dimension (see also \cite{Dai2012ber}, \cite{Buf2014Sol} \cite{Dai2007Feng} for rates in some explicit examples of $r$). In the complementary case to the afformentioned effective result of  Li-Sahlsten \cite[Theorem 1.3]{li2019trigonometric}, when all contractions are powers of some $r\in (0,1)$,  Varj\'{u}-Yu  \cite{varju2020fourier} proved logarithmic decay as long as $r^{-1}$ is not a Pisot or a Salem number.  Finally, Solomyak \cite{solomyak2019fourier} has recently established polynomial decay for all  self-similar measures except for a zero Hausdorff dimensional exceptional set of contraction ratios.

Next, the normality assertion of Corollary \ref{Main Corollary} part (3)   is related to another Theorem of Hochman-Shmerkin \cite[Theorem 1.6]{hochmanshmerkin2015} , where a similar result is proved for regular  $C^\omega$ IFS's that contain non-affine maps. Apart from removing the separation assumption, the $C^{2}$ case and the classification of the conjugated IFS as in part (3) seem to be completely new.  

The Rajchman assertion of part (3)  gives many new examples of self conformal Rajchman measures. It also  provides a unified proof to several pre-existing results regarding the Rajchman property for $C^\omega$ IFS's that are not conjugate to a self similar IFS. These include those of   Sahlsten-Stevens \cite{sahlsten2020fourier} (for a class of regular $C^\omega$ self-conformal measures that are not conjugate to linear), and in some cases those of Li \cite{li2018fourier, Li2018decay} (Furstenberg measures for $SL(2,\mathbb{R})$ cocycles under mild assumption - see \cite{Yoccoz2004some, Avila2010jairo} for conditions  ensuring that such measures satsify the conditions of part (3)). Part (3) is also closley related to the work of Bourgain-Dyatlov \cite{Bour2017dya} (who study Patterson-Sullivan measures for convex cocompact Fuchsian groups, see also \cite{li2019kleinian}).   However, we do not recover the  polynomial decay rate proved in \cite{Bour2017dya, li2018fourier, sahlsten2020fourier}.  Finally, generlizing the work of Kaufman \cite{Kaufman1980normal} and later Queff\'{e}lec-Ramar\'{e} \cite{Queff2003Ramar}, Jordan-Sahlsten  \cite{Sahl2016Jor} and later Sahlsten-Stevens \cite{sahlsten2018fourier} proved polynomial  decay for certain Gibbs measures for the Gauss map $x\mapsto \frac{1}{x} \mod 1$ on the interval. These can be considered as self-conformal measures with respect to an IFS with countably many maps, see also \cite[Theorem 1.12]{hochmanshmerkin2015}.

\subsection{Dimension of absolutely normal numbers inside self similar sets}
Let us now specialize to self similar sets. Let $\lambda$ denote the Lebesgue measure on $\mathbb{R}$, and recall that Borel's normal number Theorem asserts that $\lambda$-a.e. $x$ is absolutely normal. However, Borel's Theorem gives no information about absolutely normal numbers inside sets that are Lebesgue null. The following rigidity result says that  absolutely normal numbers  have full Hausdorff dimension inside a given self similar set, unless the underlying IFS has a very specific structure. We use the standard notation $\dim X$ for the Hausdorff dimension of a set $X$, and 
$$\dim \mu = \inf \lbrace \dim X : \mu(X)>0 \rbrace$$
for the (lower) Hausdorff dimension of a Borel probability measure $\mu$.  Now, given an orientation preserving self similar IFS $\Phi$ let  $\Phi_1 := \Phi$, and for every integer $m\geq 2$,
$$\Phi_m := \lbrace g: g= \phi_1 \circ ... \circ \phi_\ell, \text{ such that } \phi_i \in \Phi, g'(0)<\frac{1}{m} \text{ and }  \left( \phi_1 \circ ... \circ \phi_{\ell-1} \right)' (0) \geq \frac{1}{m} \rbrace.$$
Observe that for every $m\in \mathbb{N}$, $K_{\Phi_m} = K_\Phi$, i.e. all these IFS's have the same attractor as $\Phi$.
\begin{theorem} \label{Theorem 1}
Let $\Phi$ be an orientation preserving self similar IFS with attractor $K$. If
\begin{equation} \label{Eq condition}
\dim K\cap \lbrace x:\, x \text{ is  absolutely normal } \rbrace < \dim K
\end{equation}
then there exists some $m\in \mathbb{N}$ such that the IFS $\Phi_m = \lbrace g_i (x) = r_i x + t_i \rbrace $ has the following structure:
\begin{enumerate} [label=(\roman*)]
\item There is an integer $n\geq 2$ such that every $r_i = n^{-k_i}$ and the $k_i \in \mathbb{N}$ are relatively prime.

\item Every $t_i$ is not $n$-normal. If there exist $r_k \neq r_j$ and $\Phi_m$ is regular then every $t_i \in \mathbb{Q}$. 
\end{enumerate}
If $\Phi$ admits a self similar measure $\mu$ with $\dim \mu = \dim K$ then this holds for $m=1$. That is, the original IFS $\Phi$ already has this structure.
\end{theorem}
For example, if $\Phi$ is regular and admits two different contraction ratios, then  \eqref{Eq condition} implies that all contraction ratios are powers of some integer $n\geq 2$ and all translations are rational. In general, if $\Phi$  does not have exact overlaps (i.e. the semi-group generated by its maps is free), and if the exact overlaps conjecture \cite{Simon1993overlaps} holds true, then one can always find a self similar measure $\mu$ with $\dim \mu = \dim K$. Hochman \cite[Theorem 1.1]{hochman2014self} verified this conjecture under very weak regularity conditions, that hold true if e.g. all parameters of $\Phi$ are algebraic \cite[Theorem 1.5]{hochman2014self} (Rapaport \cite{rapaport2020proof} recently showed that it suffices to assume only the contraction ratios are algebraic). The conjecture was also fully resolved for Bernoulli convolutions  by the combined efforts of Hochman \cite{hochman2014self}, Breuillard-Varj\'{u} \cite{Bre2019Varju}, and finally Varj\'{u} \cite{Varju2010dim} (see also \cite{rapaport2020self}). Thus, in all the cases listed above, if $\Phi$ does not satisfy (i) and (ii) as in Theorem \ref{Theorem 1}, then
$$\dim K\cap \lbrace x:\, x \text{ is  absolutely normal } \rbrace = \dim K .$$
It is   interesting to compare this with a  result \cite{Broderick2010Weiss} of Broderick et al.  (that extends a Theorem of Schmidt \cite{Schmidt1966normal}): The set of real numbers not normal to any
integer base  intersects \textit{any} infinite self-similar in a set of full Hausdorff dimension in the fractal. 

The proof of Theorem \ref{Theorem 1} relies on the following Theorem, which is a bi-product of the argument proving Theorem \ref{Theorem main tech}:
\begin{theorem} \label{Prop key 1}
Let $\mu$ be a self similar measure. If $\mu$ is a  Rajchman measure  then it is pointwise absolutely normal.
\end{theorem}
This is a significant improvement to the  Davenport-Erd\H{o}s-LeVeque criterion, since no decay rate is required. Once Theorem \ref{Prop key 1} is established, we proceed to prove Theorem \ref{Theorem 1} by first showing that, under its assumptions, there is some $m$ and a self similar  measure on $\Phi_m$ that is not Rajchman. We then combine the results of Li-Sahlsten \cite{li2019trigonometric} and those of  Br\'{e}mont \cite{bremont2019rajchman} to prove that $\Phi_m$ is affinely  conjugated to an IFS in so-called Pisot form \cite[Definition 2.2]{bremont2019rajchman}. We then show that the underlying Pisot number is in fact an integer, and  characterize  the affine conjugating map  by appealing to the recent results of Dayan-Ganguly-Weiss \cite{dayan2020random}. See Section \ref{Section proof thm 1} for more details.

\subsection{Some  further remarks} \label{Section proof sketch}
It is natural to ask for a condition similar to \eqref{Eq Dio condition} that would yield a quantitative decay rate for the Fourier transform of self conformal measures with respect to non-linear IFS's. To this end we require effective versions of the central and local limit Theorem for cocycles from \cite{Benoist2016Quint}, that are used to prove the Rajchman property. Now, in the self-similar case the random walk driven by the derivative cocycle is a classical random walk on the line. Thus, here we can substitute the central limit Theorem \cite[Theorem 12.1]{Benoist2016Quint}  for the   Berry-Esseen inequality \cite{Berry1942Esseen}, and the local limit Theorem  \cite[Theorem 16.15]{Benoist2016Quint} for Breuillard's effective local limit Theorem \cite[Th\'{e}or\`eme 4.2]{Breuillard2005llt}, which is why we require the Diophantine condition. Since the arguments of Breuillard \cite{Breuillard2005llt} and of Benoist-Quint \cite[Chapter 16]{Benoist2016Quint} are closely related, it might be possible to prove an effective local limit Theorem for the derivative cocycle under a suitable Diophantine condition. We plan to study this problem in the near future.

Also, the  method of Hochman-Shmerkin \cite{hochmanshmerkin2015} for pointwise normality   works for a broader class of   $\beta$ transformations of the form $T_\beta (x)= \beta x \mod 1$, when $\beta>1$  is a  Pisot number.    In its current form, our method seems less suitable to treat $T_\beta$ for non-integer $\beta$. The main issue is that the identity $T_p ^n = T_{p^n}$, which is used  several times in our proof (e.g. in Claim \ref{First linearization}) and is trivial for integers $p\geq 2$, is not true for $T_\beta$ when $\beta$ is not an integer. We remark that it might be possible to substantially refine our proof to get around this issue, and  leave this to future research.

The method of Hochman-Shmerkin also works for a broader class of measures. This class of measures is what they refer to as quasi-product measures, which is a more general class than Gibbs measures for H{\"o}lder potentials. We  expect that our results  extend to a broader class of measures as well, but  do not pursue this goal in the present paper.

Finally, the referee has suggested that by incorporating ideas from renewal theory into the proof of the conditional local limit Theorem \ref{Theorem equid}, one might be able to relax the aperiodicity assumption made in Theorem \ref{Theorem main tech} into a more arithmetic assumption. We plan to explore this possibility in  a future project.

\noindent{ \textbf{Organization}} After some preliminaries in Section \ref{Section pre},  the  local limit Theorems previously alluded to are formulated and proved in Section \ref{Section LLT}. The subsequent Sections \ref{Section proof 1} and \ref{Section proof 2} contain the proof of Theorem \ref{Theorem main tech} parts (1) and (2), respectively. After that, in Section \ref{Section coro} we derive Corollary \ref{Main Corollary} from Theorem \ref{Theorem main tech} and provide some examples of Diophantine IFS's. The final Section \ref{Section proof thm 1} contains the proof of Theorem \ref{Theorem 1}, and related constructions.

\noindent{ \textbf{Acknowledgements}}  The inspiration for this work came from Hochman's recent proof \cite{Hochman2020Host} of Host's equidistribution Theorem \cite{Host1995normal}.  We  thank Mike Hochman for providing us with a preprint of his work, and for some illuminating discussions about it. We also thank Tuomas Sahlsten and Connor Stevens  for interesting  discussions regarding the Rajchman property. We are grateful to P\'{e}ter Varj\'{u} and the anonymous referee for their helpful remarks, that in particular allowed us to weaken our previous definition of aperiodicity for self similar IFS's to its current form, and thus strengthen Corollary \ref{Main Corollary}. Finally, we thank the referee for suggesting an alternative proof of Theorem \ref{Theorem main tech} part (2) using the tools of this paper that is outlined in Remark \ref{Rmk refree}, and for pointing out a certain possible simplification of our Diophantine condition, discussed in Remark \ref{Remark Dio}.

\section{Preliminaries} \label{Section pre}
Throughout this section we work with an IFS as in Theorem \ref{Theorem main tech}, and follow the notation introduced in Section \ref{Section sketch}. We also use the notation $\mathcal{A}= \lbrace 1,...,n\rbrace$. By uniform contraction  there exists $D,D'\in \mathbb{R}$  such that
\begin{equation} \label{Eq C and C prime}
0<D:= \min \lbrace -\log |f' (x)| : f\in \Phi, x\in I \rbrace, \quad D':= \max \lbrace -\log |f' (x)| : f\in \Phi, x\in I \rbrace < \infty.
\end{equation}
Equivalently, for every $f\in \Phi$ and $x\in I$,
$$0< e^{-D'} \leq |f'(x)| \leq e^{-D} <1.$$
We also define, for $\mathbb{P}$ a.e. $\omega$ and any fixed $x_0 \in I$,
$$\chi:= \lim_{n} \frac{-\log |f_{\omega|_n}' (x_0) |}{n}>0$$
the corresponding Lyapunov exponent. 
\subsection{Basic geometry of self-conformal measures}
Here, we recall two useful and well known results. The first is the bounded distortion property, which holds in our situation since every $f_i$ is at least $C^{1+\gamma}$ smooth and strictly contracting. We refer e.g. to  \cite[The discussion about equation (1.3)]{Solomyak2000Peres} for more details.
\begin{theorem}  \label{Bounded distortiion theorem}
There exists some $L=L(\Phi)>1$ such that: for any $k\in \mathbb{N}$ and for any word $\eta \in \mathcal{A}^k$ 
$$ L^{-1} \leq \frac{\left|f_\eta ' (x) \right| }{ \left| f_\eta ' (y) \right| } \leq L,\quad \text{ for any } x,y\in I.$$
\end{theorem}
The second is another standard result, about the non-atomicity of self-conformal measures. It follows from e.g. \cite[Proposition 2.2]{feng2009Lau}:
\begin{Lemma} \label{Lemma nu is continuous}
Let $\nu$ be a self conformal measure as in Theorem \ref{Theorem main tech}. Then $\nu$ is not atomic. That is, for every $\epsilon>0$ there is a $\delta>0$ such that for any $y\in I$,
$$\nu (B_\delta (y))<\epsilon$$
where $B_\delta (y)$ is the open ball about $y$ of radius $\delta>0$.
\end{Lemma}
Notice that here we using our standing assumptions that $K_\Phi$ is infinite and that $\textbf{p}$ is a strictly positive probability vector on $\mathcal{A}=\lbrace 1,...,n\rbrace$.

\subsection{Linearization lemmas}
Recall that $\Phi$ is a family of differentiable contractions  $ I\to I$ satisfying
\begin{equation*}
 0<  e^{-D'}\leq\|f'\|_{C^0}\leq e^{-D}<1, \quad \|f\|_{C^{1+\gamma}}\leq C,\quad \forall f\in\Phi
\end{equation*} for some $D,D'>0, \gamma\in(0,1)$ and $C>0$. Define $$\Phi^{*n}:=\{\phi_1\circ\cdots\circ\phi_n:\quad\phi_1,\cdots,\phi_n\in\Phi\}.$$ 

We shall require the following $C^0$ linearization lemma:

\begin{Lemma}  \label{Lemma lin} For every $\beta\in(0,\gamma)$ there exists $\epsilon\in(0,1)$ such that for all $n\geq 1$, $g\in\Phi^{*n}$ and $x,y\in I$ satisfying $|x-y|<\epsilon$, $$\big|g(x)-g(y)-g'(y) (x-y)\big|\leq|g'(y)|\cdot|x-y|^{1+\beta}.$$\end{Lemma}
What Lemma \ref{Lemma lin} means is that for every $y\in B_\epsilon (x)$ the function $g$ may be approximated exponentially fast on $B_\epsilon(x)$ by an affine map with similarity ratio $g'(y)$.

\begin{proof}
For the purpose of this proof only, it will be convenient to use the notation 
\begin{equation*} 0< \kappa' :=e^{-D'} \leq\|\phi'\|_{C^0}\leq \kappa:= e^D <1, \quad  \forall \phi\in\Phi.\end{equation*}
Now, for all $x, y\in I$ and $\phi\in\Phi$, there is an intermediate value $z$ between $x$ and $y$ such that
\begin{equation}\label{EqLinearization1}\begin{aligned}\big|\phi(x)-\phi(y)-\phi'(y)  (x-y)\big|
=&\big|\phi'(z) (x-y)-\phi'(y) (x-y)\big|\\
\leq &  |\phi'(z) -\phi'(y) |  |x-y|
\leq C|z-y|^\gamma  |x-y|\\
\leq &C|x-y|^{1+\gamma}.\end{aligned}\end{equation}

Define a sequence $\beta_0>\beta_1>\beta_2>\cdots$ by $\beta_0=\gamma$, $\beta_n=\beta_{n-1}-b\kappa^{(n-1)\gamma}$, where the constant $b:=(1-\kappa^\gamma)(\gamma-\beta)>0$ is chosen such that $\lim\beta_n=\beta$.  Choose $\epsilon\in(0,\frac1e)$ sufficiently small, such that
\begin{equation}\label{EqLinearizationRestraint0}\frac{4C}{\kappa'}\epsilon^{\gamma-\beta_1}<\min(1-\frac1e,\frac \beta2).\end{equation}

We will  prove inductively that:

{\it If $n\geq 0$, $|x-y|<\epsilon$ and $g\in\Phi^n$, then
\begin{equation}\label{EqLinearizationInductive}\big|g(x)-g(y)-g'(y)(x-y)\big|\leq |g'(y)|\cdot|x-y|^{1+\beta_n}.\end{equation}}

For the $n=0$ case, assume $|x-y|<\epsilon$ and $g\in\Phi^{*0}=\{\mathrm{Id}\}$, then  $|g(x)-g(y)-g'(y)(x-y)\big|=|(x-y)-(x-y)|=0$ and \eqref{EqLinearizationInductive} holds.

\hide{
For the $n=1$ case, assume $|x-y|<\epsilon$ and $\phi\in\Phi=\Phi^{*1}$, then by \eqref{EqLinearizationRestraint0} and \eqref{EqLinearization1}, \begin{equation}\begin{aligned}&\big|\phi(x)-\phi(y)-(D_{y}\phi) (x-y)\big|\\
\leq &C|x-y|^{1+\gamma}=C|x-y|^{\gamma-\beta_1}\cdot |x-y|^{1+\gamma_1}\\
\leq& C\epsilon^{\gamma-\beta_1}\cdot |x-y|^{1+\beta_1}
\leq \kappa_1 |x-y|^{1+\beta_1}\\
\leq& |\phi'(y)|\cdot|x-y|^{1+\beta_1}.\end{aligned}\end{equation}}

Assume $n\geq 1$ and the lemma holds for $n-1$. Suppose $|x-y|=\delta<\epsilon$ and $g\in\Phi^{*n}$. Then $g=\phi\circ \tilde g$ where $\phi\in\Phi$ and $\tilde g\in\Phi^{*(n-1)}$, and \begin{equation}\label{EqLinearization2}\big|\tilde g(x)-\tilde g(y)-\tilde g'(y) (x-y)\big|\leq|\tilde g'(y)| \delta^{1+ \beta_{n-1}}.\end{equation}
 In particular, it follows that
\begin{equation}\label{EqLinearization3} |\tilde g(x)-\tilde g(y)|\leq |\tilde g'(y)| \delta(1+\delta^{\beta_{n-1}}).\end{equation}

Combining \eqref{EqLinearization1} and \eqref{EqLinearization3}, we get
\begin{equation}\label{EqLinearization4}\begin{aligned}&\big| g(x)-  g(y)-g'(y) (x-y)\big|\\
=&\big|\phi(\tilde g(x))-\phi(\tilde g(y))-\phi'(\tilde g(y))\tilde g'(y)(x-y)\big|\\
\leq&\big|\phi(\tilde g(x))-\phi(\tilde g(y))-\phi'(\tilde g(y))(\tilde g(x)-\tilde g(y))\big|\\
&\quad + |\phi'(\tilde g(y))|\cdot\big|\tilde g(x)-\tilde g(y)-\tilde g'(y)(x-y)\big|\\
\leq&C|\tilde g(x)-\tilde g(y)|^{1+\gamma}+|\phi'(\tilde g(y))|\cdot |\tilde g'(y)| \delta^{1+ \beta_{n-1}}\\
\leq&C\big(|\tilde g'(y)| \delta(1+\delta^{\beta_{n-1}})\big)^{1+\gamma}+|g'(y)| \delta^{1+ \beta_{n-1}}.\end{aligned}\end{equation}

Because $0<\delta<\epsilon<1$, we know \eqref{EqLinearization4} is bounded by:\begin{equation}\label{EqLinearization5}\begin{aligned}&\big| g(x)-  g(y)-(D_{y}g) (x-y)\big|\\
\leq &4C|\tilde g'(y)|\cdot|\tilde g'(y)|^\gamma \delta^{1+\gamma}+|g'(y)| \delta^{1+ \beta_{n-1}}\\
=&\Big(4C|\phi'(\tilde g(y))|^{-1}|\tilde g'(y)|^\gamma \delta^{\gamma-\beta_n}+\delta^{\beta_{n-1}-\beta_n}\Big)|g'(y)| \delta^{1+ \beta_n}\\
\leq&\Big(\frac{4C}{\kappa'}\kappa^{(n-1)\gamma}\delta^{\gamma-\beta_n}+\delta^{\beta_{n-1}-\beta_n}\Big)|g'(y)| \delta^{1+ \beta_n}.\end{aligned}\end{equation}

To complete the induction, it suffices to prove 
\begin{equation}\label{EqLinearization6}\frac{4C}{\kappa'}\kappa^{(n-1)\gamma}\delta^{\gamma-\beta_n}+\delta^{\beta_{n-1}-\beta_n}\leq 1.\end{equation}

We distinguish between the cases $\log\frac1\delta\geq\frac1{\beta_{n-1}-\beta_n}$ and  $\log\frac1\delta<\frac1{\beta_{n-1}-\beta_n}$. 

If $\log\frac1\delta\geq\frac1{\beta_{n-1}-\beta_n}$, then $\delta^{\beta_{n-1}-\beta_n}\leq\frac1e$.  Moreover, by \eqref{EqLinearizationRestraint0}, $\frac{4C}{\kappa'}\kappa^{(n-1)\gamma}\delta^{\gamma-\beta_n}\leq\frac{4C}{\kappa'}\epsilon^{\gamma-\beta_1}<1-\frac1e$. Thus \eqref{EqLinearization6} holds in this case.

Assume now $\log\frac1\delta<\frac1{\beta_{n-1}-\beta_n}$. Since $e^{-t}\leq 1-\frac12t$ on $[0,1]$, we know $$\begin{aligned}\delta^{\beta_{n-1}-\beta_n}=&e^{-(\log\frac1\delta) (\beta_{n-1}-\beta_n)}\\
\leq&1-\frac12(\log\frac1\delta) (\beta_{n-1}-\beta_n)\leq 1-\frac12(\beta_{n-1}-\beta_n)
=1-\frac 12b\kappa^{(n-1)\gamma}\\
\leq & 1-\frac{4C}{\kappa'}\kappa^{(n-1)\gamma}\epsilon^{\gamma-\beta_1}
\leq 1-\frac{4C}{\kappa'}\kappa^{(n-1)\gamma}\delta^{\gamma-\beta_n}.  \end{aligned}$$
Here we used the facts that $0<\delta<\epsilon<\frac1e$, $\beta_n\leq \beta_1<\gamma$ and assumption \eqref{EqLinearizationRestraint0}. Hence \eqref{EqLinearization6} holds in this case as well. 

We have established the inductive statement \eqref{EqLinearizationInductive}. As $\beta<\beta_n$ and $|x-y|<\epsilon<1$, the lemma then follows.\end{proof}

An important ingredient in the proof of the second local limit Theorem \ref{Theorem lin equid}, discussed in Section \ref{Section LLT}, is the following (much easier) $C^1$ counterpart:

\begin{Lemma} \label{Lemma C1 lin} For all $n\geq 1$, $g\in\Phi^{*n}$ and $x,y\in I$, $$\big|\log|g'(x)|-\log|g'(y)|\big|\lesssim_\Phi |x-y|^ \gamma.$$\end{Lemma}
\begin{Remark}
By the notation $A \lesssim_\Phi B$ we mean that the number $A$ is smaller than $C\cdot B$, where the multiplicative constant $C=C(\Phi)$ depends only on $\Phi$. Similar notation is used throughout the paper.
\end{Remark}
\begin{proof}
Fix $n\in \mathbb{N}$, and let $g\in \Phi^{*n}$. Then there is some $\eta \in \mathcal{A}^*$ such that $g=f_\eta$. By our assumptions on the maps in $\Phi$, we have
\begin{eqnarray*}
\big|\log|g'(x)|-\log|g'(y)|\big| &\leq  & \sum_{m=1} ^n \left| \log|f_{\eta_m} ' (f_{\eta|_{m-1}} (x))| - \log|f_{\eta_m} ' (f_{\eta|_{m-1}} (y))| \right| \\
& \leq & \sum_{m=1} ^n \frac{1}{\min_{z\in I} |f_{\eta_m} ' (z)|} \left| |f_{\eta_m} ' (f_{\eta|_{m-1}} (x))| - |f_{\eta_m} ' (f_{\eta|_{m-1}} (y))| \right| \\
& \leq & \sum_{m=1} ^n e^{D'} \left| f_{\eta_m} ' (f_{\eta|_{m-1}} (x)) - f_{\eta_m} ' (f_{\eta|_{m-1}} (y)) \right| \\
& \leq & \sum_{m=1} ^n C\cdot e^{D'} \left| f_{\eta|_{m-1}} (x) - f_{\eta|_{m-1}} (y) \right|^\gamma \\
&  \leq& C\cdot e^{D'} \sum_{m=1} ^n e^{-D\cdot m \cdot \gamma}  \left| x - y \right|^\gamma \leq C_0 \left| x - y \right|^\gamma  \\
\end{eqnarray*}
for some global constant $C_0$ that is independent of $n$, completing the proof.
\end{proof}

\subsection{The Fourier transform of scaled measures} \label{Section Hochman Lemma}
The following Lemma is adapted from a recent paper  of Hochman \cite{Hochman2020Host}. For any $s,x\in \mathbb{R}$ let $M_s(x)=s\cdot x$ denote the multiplication map, and for any metric space $X$ let $\mathcal{P}(X)$ denote the space of Borel probability measures on $X$.
\begin{Lemma} \cite[Lemma 3.2]{Hochman2020Host} \label{Lemma 3.2 }
Let $\theta \in \mathcal{P}(\mathbb{R})$, $k >0$ and $\chi, D$ as in the previous Sections. Then for any $r>0$ and $q\neq 0$,

$$\int_{k\cdot \chi} ^{k \cdot  \chi +D'} |\mathcal{F}_q( M_{e^{-t}} \theta)|^2 dt \leq D' \cdot \left( \frac{e^2}{r \cdot |q|} +  \int \theta(B_{e^{\chi k} \cdot r} (y)) d\theta (y) \right).$$
\end{Lemma} 

In fact, Hochman's Lemma states that for any $\theta \in \mathcal{P}(\mathbb{R})$, any $r>0$, an any $m\neq 0$,

$$ \int_0 ^1 |\mathcal{F}_m( M_{p^t} \theta)|^2 dt \leq  \frac{2}{r \cdot |m|\cdot \log p} + \int \theta(B_r (y)) d\theta (y) $$
where here $p>1$.

In the context of Lemma \ref{Lemma 3.2 }, we  apply this result for $p=e^{-1}$ and the measure $M_{e^{-k\chi}} \theta$ between the scales $0$ and $D'$. Then the same proof yields
\begin{eqnarray*}
\int_{k\cdot \chi} ^{k \cdot  \chi +D'} |\mathcal{F}_m( M_{e^{-t}} \theta)|^2 dt &=& \int_{0} ^{D'} |\mathcal{F}_m( M_{e^{-t}} \left( M_{e^{-k\chi}} \theta \right) )|^2 dt \\
&\leq &D' \cdot \left( \frac{e^2}{r \cdot |m|} +  \int M_{e^{-k\chi}} \theta \left( B_r (y) \right) dM_{e^{-k\chi}} \theta(y) \right) \\
&= & D' \cdot \left(  \frac{e^2}{r \cdot |m|} +  \int \theta(B_{e^{\chi k} \cdot r} (y)) d\theta (y) \right) \\
\end{eqnarray*}
which is Lemma \ref{Lemma 3.2 }.

\section{Two local limit theorems} \label{Section LLT}
We continue to work with an IFS $\Phi$ as in Theorem \ref{Theorem main tech}, and follow the notation introduced in Sections \ref{Section sketch} and \ref{Section pre}. The purpose of this Section is to establish two local limit Theorems: Theorem \ref{Theorem equid}, which will be used to prove the Rajchman part of Theorem \ref{Theorem main tech}, and Theorem \ref{Theorem lin equid}  which will be used to prove the normality part of Theorem \ref{Theorem main tech}. Our analysis relies on the central and local limit Theorems for cocycles with target proved by Benoist-Quint \cite{Benoist2016Quint}.

\subsection{Statements of the local limit Theorems} \label{Section statement}
We define the following functions and stopping times on $\mathcal{A}^\mathbb{N}$:
$$S_n(\omega)=-\log|f'_{\omega|_n}(x_{\sigma^n(\omega)})|;\ \tau_k(\omega)=\min\{n: S_n(\omega)\geq k\chi\};$$
$$\tS_n(\omega)=-\log\max_{x\in I}|f'_{\omega|_n}(x)|;\ \ttau_k(\omega)=\min\{n: \tS_n(\omega)\geq k\chi\}.$$
In the definitions above $n$ is a positive integer but $k$ is allowed to take positive non-integer values.  $S_n$ will be shown to arise from a random walk driven by the derivative cocycle. Also, note that $\ttau_k$ is a stopping time. Both $S_n$ and $\tS_n$ are strictly increasing in $n$.

\begin{Definition} \label{Def R.V.} Let $X_1: \mathcal{A}^\mathbb{N}\rightarrow \mathbb{R}$ be the random variable
$$X_1(\omega):= c(\omega_1,\sigma(\omega))=-\log |f'_{\omega_1}(x_{\sigma(\omega)})|.$$
For every integer $n>1$ we define 
$$X_n(\omega) =  - \log |f_{\omega_n} ' \left( x_{\sigma^n (x_\omega)} \right) | = X_1 \circ \sigma^{n-1}.$$
\end{Definition}

 Let $\theta$ be the law of the random variable $X_1$. Recall the definition of $D$ and $D'$ from \eqref{Eq C and C prime}.  Then for every $n$, $X_n \sim \theta$. Moreover, $\theta \in \mathcal{P}([D,D'])$. In particular, the support of $\theta$ is bounded away from $0$.
These are immediate from Definition \ref{Def R.V.}  and equation \eqref{Eq C and C prime}. The following Lemma is a direct consequence of the chain rule:
\begin{Lemma} \label{Lemma relation between R.V and stopping time}
For every $k>0$ and $\omega\in A^\mathbb{N}$ we have
$$S_n(\omega) = \sum_{i=1} ^n X_i (\omega)$$
and  
$$-\log |f_{\omega|_{\tau_k(\omega)}}'(x_{\sigma^{\tau_k(\omega)}(\omega)})|= S_{\tau_k(\omega)} (\omega) \in [k\chi, k\chi+D']. $$
\end{Lemma}

Next, we introduce some partitions of the space $\mathcal{A}^\mathbb{N}$:

\begin{Definition} \label{Def R.V. D} 
 Given a finite word $\eta=(\eta_1,\cdots, \eta_\ell) \in \cA^\ell$:
\begin{enumerate}
\item Denote by $A_\eta\subseteq \cA^{\mathbb N}$ the set of infinite words that begin with $\eta$,
$$A_\eta:= \lbrace \omega: \quad (\omega_1,...,\omega_\ell)= \eta\}.$$
\item    For $h\in[0,\infty)$, let $\cA^h$ be the partition of  $\cA^{\mathbb N}$ according to the function 
$$\iota^h:\omega\to (\omega_1,\cdots,\omega_{\ttau_h(\omega)}).$$
This is a finite partition, and each set in it is a cylinder set of the form $A_\eta$.

\item Let  $\eta'$ be another finite word.  We define the event
$$A_{k,\eta,\eta'}:= \{\omega\in A_\eta:\quad \sigma^{\tau_k(\omega)-1}(\omega)\in A_{\eta'}\}.$$


\item Given $k,h,h'\geq 0$ we denote by $\cA_k^{h,h'}$ the finite partition of $\cA^{\mathbb N}$ according to the map 
$$\iota_k^{h,h'}=(\iota^h(\omega),\iota^{h'}(\sigma^{\tau_k(\omega)-1}(\omega)).$$
 Note that every cell of the partition $\cA_k^{h,h'}$ has the form $A_{k,\eta,\eta'}$.
\end{enumerate} 
\end{Definition}

Given $k,h,h'\geq 0$ and $\omega \in \cA^{\mathbb{N}}$ we write $\cA_k^{h,h'} (\omega)$ for the unique $\cA_k^{h,h'}$ cell that contains $\omega$. For $\mathbb P$-a.e. $\omega$, we denote the conditional measure of $\mathbb{P}$ on the corresponding  cell by $\mathbb{P}_{\cA_k^{h,h'}(\omega)}$. Similarly, $\mathbb{P}_{\cA^h(\omega)}$ denotes the conditional measure of $\mathbb{P}$ on the cylinder in the partition $\cA^h$ that contains $\omega$. Recall that $\lambda$ is the Lebesgue measure on $\mathbb{R}$.

\begin{Definition} \label{Def Gammak}
Let $k\in \mathbb{N}$ and let $\eta$, $\eta'$ be finite words. Assuming $\mathbb{P}(A_{k,\eta,\eta'})>0$, we define a  probability measure $\Gamma_{A_{k,\eta,\eta'}} $ on $[k\chi, k\chi+D']$  by
$$ \Gamma_{A_{k,\eta,\eta'}} := \frac{\int_{A_{\eta'}}{\lambda|_{[k\chi,k\chi+X_1(\omega')]}} d \mathbb P(\omega')}{\int_{A_{\eta'}}{ X_1(\omega')} d \mathbb P(\omega')}.$$.
\end{Definition}
Note that there is actually no dependence on $\eta$; We use this notation for later convenience. 

\begin{Lemma} \label{Lemma Abs. contin of gamma}
If $\mathbb{P}(A_{k,\eta,\eta'})>0$ then  $\Gamma_{A_{k,\eta,\eta'}}  \ll \lambda_{[k \chi,k\chi+D']}$ with a density  that depends only on $\mathbf{p}$, such that its norm is bounded above by $\frac{1}{D}$ independently of all  parameters, including $k, \eta,\eta'$.
\end{Lemma}
\begin{proof}
We write $\Gamma$ instead of $\Gamma_{A_{k,\eta,\eta'}}$.  It is clear that  $\Gamma \ll \lambda_{[k \chi,k\chi+D']}$.
Next, assuming $x\in \supp (\Gamma)$, we need to bound
$$\lim_{r\rightarrow 0} \frac{\Gamma(B(x,r))}{\lambda_{[k \chi,k\chi+D']}(B(x,r)} = \lim_{r\rightarrow 0} \frac{\Gamma(B(x,r))}{2r}$$ 
assuming $x$ is not an endpoint of $[k \chi,k\chi+D']$ (which we may assume). Then for every $\omega'\in  A_{k,\eta,\eta'}$
$$\lambda \left( [k\chi,k\chi+X_1(\omega')] \bigcap B(x,r) \right) \leq  2r$$
so
$$ \frac{\Gamma(B(x,r))}{2r} = \frac{\int_{A_{\eta'}} \lambda \left( [k\chi,k\chi+X_1(\omega')] \bigcap B(x,r) \right) d\mathbb{P}(\omega)  }{2r \cdot \mathbb{E}_{A_{k,\eta,\eta'}} (X_1)} \leq \frac{1}{\mathbb{E}_{A_{k,\eta,\eta'}} (X_1)}.$$
Finally, by equation \eqref{Eq C and C prime} we know that $X_1(\omega) \geq D$ for every $\omega$. We conclude that the density of $\Gamma$ is bounded by $\frac{1}{D}$, independently of all parameters. 
\end{proof}

 \begin{Notation} We will use superscripts such as $o^{k\to \infty}(\cdot)$ in $O(\cdot)$ and $o(\cdot)$ to such bounds take place as which variables are being varied. The variables on which the implied constants depend on will be written in subscripts. The implied constant is absolute when no subscript is present.\end{Notation}
 
 The following local limit Theorem is one of the main keys to the proof of Theorem \ref{Theorem main tech}. It is the only place  where the assumption $\Lambda_c = \lbrace 0 \rbrace$ from Theorem \ref{Theorem main tech} is used.

\begin{theorem}  \label{Theorem equid}
Fix $h_0\geq 0$. For all $k, h'>0$, $0\leq  h\leq h_0$, and $A_\eta \in \cA^h$, there exists a subset $\tA_{k,\eta}^{h,h'}\subseteq A_\eta$ such that:\begin{enumerate} [label=(\roman*)]
\item $\mathbb P(\tA_{k,\eta}^{h,h'})\geq \mathbb P(A_\eta)\cdot (1-o_{h_0,\bfp}^{k\to\infty}(1))$.
\item for all $\xi\in \tA_{k,\eta}^{h,h'}$, $\mathbb{P}(\cA_k^{h,h'}(\xi))>0$.
\item  for all $\xi\in \tA_{k,\eta}^{h,h'}$ and for any sub-interval $J\subseteq  [k\chi, k\chi+D']$, 
$$
 \mathbb{P}_{\cA_k^{h,h'}(\xi)}(S_{\tau_k}\in J)= \Gamma_{\cA_k^{h,h'}(\xi)} (J)  +o^{k\to\infty}_{h_0,\bfp}(1).
$$
\end{enumerate}
\end{theorem}

We emphasize that all bounds  are uniform in $h'$. We remark that to prove Theorem  \ref{Theorem main tech} part (1) we only need the case\footnote{The referee has pointed out to us that this special case is  related to the work of Kesten \cite[Theorem 1]{Kesten1974renewal}. Indeed, when $h_0=0$ one may deduce a weaker version of Part (iii) of Theorem \ref{Theorem equid} from Kesten's result.} $h_0 =0$ which means that $A_\eta$ is the full symbolic space.  However, this more general  version is needed to obtain the following upgraded version of Theorem \ref{Theorem equid}, which is what we require for Theorem \ref{Theorem main tech} part (2):

\begin{theorem} \label{Theorem lin equid} 
For all $k, h', h>0$ and $A_\eta \in \cA^h$, there exists a subset $\overline A_{k,\eta}^{h,h'}\subseteq A_\eta$ such that:\begin{enumerate} [label=(\roman*)]
\item\label{lin equid 1} $\mathbb P(\overline A_{k,\eta}^{h,h'})\geq \mathbb P(A_\eta)\cdot (1-o_{\bfp}^{\min(h,k-h)\to\infty}(1))$.
\item\label{lin equid 2} for all $\xi\in\overline A_{k,\eta}^{h,h'}$, $\mathbb{P}(\cA_k^{h,h'}(\xi))>0$.
\item\label{lin equid 3}  for all $\xi\in\overline A_{k,\eta}^{h,h'}$ and for any sub-interval $J\subseteq  [k\chi, k\chi+D']$, 
$$
 \mathbb{P}_{\cA_k^{h,h'}(\xi)}(S_{\tau_k}\in J)= \Gamma_{\cA_k^{h,h'}(\xi)} (J)  +o^{\min(h,k-h)\to\infty}_{\bfp}(1).
$$
\end{enumerate}
\end{theorem}
The difference between Theorem \ref{Theorem lin equid} and Theorem \ref{Theorem equid} lies in the role of $h$: In Theorem \ref{Theorem equid} it is assumed to be bounded by some uniform $h_0$. There is no such restriction in Theorem \ref{Theorem lin equid}, but the "price" is that the error  is now $o_{\bfp}^{\min(h,k-h)\to\infty}(1)$ instead of $o^{k\to\infty}_{h_0,\bfp}(1)$. So, to make  this error small, we need both $h$ and $k-h$ to go to $\infty$ simultaneously.

We proceed to prove Theorem \ref{Theorem equid}. Afterwards, we prove Theorem \ref{Theorem lin equid} using  the result of Theorem \ref{Theorem equid} as a black box and the $C^1$ linerization Lemma \ref{Lemma C1 lin}.

\subsection{Proof of Theorem \ref{Theorem equid}} \label{Section proof of LLT}
\subsubsection{Benoist and Quint's central and local limit Theorems for cocycles with target}
The proof of Theorem \ref{Theorem equid} relies on two limit  Theorems due to Benoist-Quint \cite{Benoist2016Quint}. Before stating them, we need some preliminaries. First, notice that  the indicator function $\mathbf 1_{A_{\eta}}$ is a locally constant function on $\mathcal{A}^\mathbb{N}$. For the following Claim, recall the definition of the maps $\iota_a$ from Section \ref{Section sketch} and the choice of our metric $d:=d_\rho$ on $\mathcal{A}^\mathbb{N}$:
\begin{Claim} \label{Properties of cocycle}
For every $a\in \lbrace 1,...,n\rbrace$ the following statements hold true:
\begin{enumerate}
\item The map $\iota_a$ is uniformly contracting: 
$$d(\iota_a(\omega),\iota_a(\eta))=\rho d(\omega,\eta).$$

\item The cocycle $c(a,\omega)$ is  uniformly bounded, Lipschitz in $\omega$, with a uniformly bounded Lipschitz constant as $a \in \lbrace 1,...,n\rbrace$ varies.
\end{enumerate}
\end{Claim} 
This is standard, since all the maps in $\Phi$ are $C^{1+\gamma}$ smooth and by equation \eqref{Eq C and C prime}.  We are now ready to state a consequence of the central limit Theorem for cocycles with target, proved by Benoist-Quint:
\begin{theorem} \cite[Theorem 12.1. part (i)]{Benoist2016Quint}  \label{B-Q CLT}
Let $A_\eta$ be  as in Theorem \ref{Theorem equid}. There exists a variance $r=r(\textbf{p})>0$ such that:

For every $R\in \mathbb{R}$ the function $\psi = 1_{A_\eta} \times 1_{[R,\infty)}$ satisfies that
$$\mathbb{P} \left(\omega\in A_\eta: \, \frac{\left|  S_n (\omega)  -  n\cdot \chi \right|}{\sqrt{n}}\geq R \right) 
= \mathbb P(A_\eta) ( N(0, r^2) > R)(1+ o_{\psi}^{n\to\infty}(1))$$
where $( N(0, r^2) > R)$ stands for the probability that a Gaussian random variable with mean $0$ and variance $r^2$ is larger than $R$.
\end{theorem}
We remark that \cite[Theorem 12.1. part (i)]{Benoist2016Quint} applies here since by  Claim \ref{Properties of cocycle}, the cocycle $c(\cdot ,\cdot)$ satisfies the bounded moment conditions  \cite[(11.14),(11.15)]{Benoist2016Quint} and is not constant. We also remark that for Theorem \ref{B-Q CLT} we do not need the assumption that $\Lambda_c=\lbrace 0 \rbrace$ made in Theorem \ref{Theorem main tech}. However, this assumption is crucial for the local limit Theorem for cocycles with target, also proved by Benoist-Quint:

\begin{theorem} \label{B-Q LLT general} \cite[Theorem 16.15]{Benoist2016Quint} 
Let $A_\eta$ be  as in Theorem \ref{Theorem equid}. Then for every $\omega' \in \mathcal{A}^\mathbb{N}$, $\epsilon>0$, $m\in \mathbb{N}$ and $w\in \mathbb{R}_+$ such that $|w-m\chi| \leq 4 \chi \sqrt{m \log m}$,
$$ \mathbb{P}_{\sigma^{-m}(\{\omega'\})}\big(\omega\in A_\eta, \, \, S_m (\omega) \in [w, w+ \epsilon\chi]\big)=G_{\sqrt m r}(w-m\chi)\cdot \mathbb{P}(A_\eta) \cdot \epsilon\chi\cdot (1+o^{m\to\infty}_{\epsilon,h_0,\textbf{p}} (1))$$
where $G_s(\cdot )$ stands for the density of the Gaussian law $N(0,s^2)$, and  $r=r(\mathbf{p})$ is as in Theorem \ref{B-Q CLT}. The decay rate in $o^{m\to\infty}_{\epsilon,h_0,\textbf{p}} (1)$  depends only on $\epsilon\chi$, $h_0$ and $\textbf{p}$, and is uniform in $\omega',w$.
\end{theorem}

To be precise, a-priori the decay rate   depends only on $\epsilon\chi$, $A_\eta$ and $\textbf{p}$. Hence it only depends on $\epsilon$, $h_0$, and $\textbf{p}$  as there are only finitely many possible choices for $A_\eta$. In this case, as in Theorem \ref{B-Q CLT}, Claim \ref{Properties of cocycle} implies that our cocycle $c(\cdot ,\cdot)$ satisfies the bounded moment conditions \cite[(11.14),(11.15)]{Benoist2016Quint}. It is aperiodic in the sense of \cite[Equation (15.8)]{Benoist2016Quint}  by the assumption that $\Lambda_c = \lbrace 0 \rbrace$. By aperiodicity, the cocycle $\tilde c$ defined by \cite[Equation (16.9)]{Benoist2016Quint}  is equal to $c$. So \cite[Theorem 16.15]{Benoist2016Quint} applies  with, in the notations therein, $X= \mathcal{A} ^{\mathbb N}$, $\varphi=\mathbf 1_{A_\eta}$, the convex set $C$ being $[0,\epsilon\chi]$, the translation $v_m = w - m\chi$ and $v_\mu =0$, and finally $x=\omega'$.

\subsubsection{Proof of Theorem \ref{Theorem equid}}
For every $r\in \mathbb{R}$ let $U_r (x)=x+r$ be the translation map.  In addition, for every $k$ we define the interval 
\begin{equation} \label{Eq Def of Ik}
I_k = [k- \sqrt{k\log k}, k+\sqrt{k\log k}].
\end{equation}
To begin the proof of Theorem \ref{Theorem equid}, we decompose the left hand side in (iii) as 
\begin{equation}\label{Eq Theorem equid}
\mathbb{P}_{\cA_k^{h,h'}(\xi)}(S_{\tau_k}\in J)=\sum_{m \not \in I_k} \mathbb{P}_{\cA_k^{h,h'}(\xi)} (\tau_k = m+1, S_{\tau_k}\in J)+\sum_{m \in I_k} \mathbb{P}_{\cA_k^{h,h'}(\xi)} (\tau_k = m+1, S_{\tau_k}\in J).\end{equation}

The two terms are respectively treated by Proposition \ref{Prop equid local} and Proposition \ref{Prop equid central} below, and the Theorem follows.

\begin{Proposition} \label{Prop equid central} In the setting of Theorem \ref{Theorem equid},  there exists a set $\widetilde A_\eta$ such that claims (i) and (ii) hold and for all $\xi\in\widetilde A_\eta$, 
$$\mathbb{P}_{\cA_k^{h,h'}(\xi)}( \tau_k -1 \notin I_k )= o^{k\to\infty}_{h_0,\textbf{p}}(1).$$
\end{Proposition}
Notice that we are using the abbreviated notation $\widetilde A_\eta$ instead of $\tA_{k,\eta}^{h,h'}$.

\begin{proof} We first prove the following claim: \begin{equation}\label{EqTotalCLT}\mathbb{P}_{A_\eta}( \tau_k-1\notin I_k)= o^{k\to\infty}_{h_0,\textbf{p}}(1), \text{ for every }\eta.\end{equation}


For the function $b=b(k)= \sqrt{ k \log k}-1$,  suppose that $|\tau_k (\omega) - k| >b=b(k)$. We also fix a small $\epsilon>0$. Without   loss of generality,  suppose first that $\tau_k (\omega) - k > b$. This implies that 
$$S_{\lfloor  k+b \rfloor} (\omega) < k \chi $$
and therefore
\begin{equation*}
| S_{\lfloor k+b\rfloor} (\omega)  - \chi\cdot\lfloor k+b\rfloor| \geq |S_{\lfloor k+b\rfloor} (\omega) -  k \chi - b\chi   - \chi | \geq b\chi.
\end{equation*}
Let $r>0$ be as in Theorem \ref{B-Q CLT}, and let $R=R(r,\epsilon)>0$ be such that $(N(0, r^2) > R)=\epsilon$. Then by Theorem \ref{B-Q CLT} applied for the corresponding $\psi$, we get
$$ \begin{aligned}&\mathbb{P} \left(\omega\in A_\eta:  \left|  S_{\lfloor  k+b \rfloor} (\omega)  -  \lfloor  k+b \rfloor \chi \right| \geq b\chi \right) \\
=&\mathbb{P} \left( \omega\in A_\eta:  \frac{\left|  S_{\lfloor  k+b \rfloor} (\omega)  -  \lfloor  k+b \rfloor \chi \right|}{\sqrt{\lfloor  k+b \rfloor}}\geq \frac{b\chi}{\sqrt {\lfloor  k+b \rfloor}}  \right) \\
\leq &\mathbb{P} \left(\omega\in A_\eta:  \frac{\left|  S_{\lfloor  k+b \rfloor} (\omega)  -  \lfloor  k+b \rfloor \chi \right|}{\sqrt{\lfloor  k+b \rfloor}}\geq R \right) \\
= &\mathbb P(A_\eta) ( N(0, r^2) > R)(1+ o_{\psi}^{k\to\infty}(1))\\=&\epsilon+o_{\epsilon,h_0,\textbf{p}}^{k\to\infty}(1).\end{aligned}$$  Here we used that $ \frac{b\chi}{\sqrt {\lfloor  k+b \rfloor}}\to \infty$ as $k\to\infty$.  Since $\epsilon$ is arbitrary, it follows that
\begin{equation*}\mathbb{P}(\tau_k (\omega) - k > b)\leq \mathbb{P} \left(  \left|  S_{\lfloor  k+b \rfloor} (\omega)  -  \lfloor  k+b \rfloor \chi \right| \geq b\chi \right) =o_{h_0,\textbf{p}}^{k\to\infty}(1).\end{equation*}
A similar argument shows \begin{equation*}\mathbb{P}\left(\tau_k (\omega) - k <-b\right)\leq \mathbb{P} \left(  \left|  S_{\lceil  k-b \rceil} (\omega)  -  \lceil  k-b \rceil\chi \right| \geq b\chi \right) = o_{h_0,\textbf{p}} ^{k\to\infty}(1).\end{equation*} The claim \eqref{EqTotalCLT} then follows by combining the two inequalities above.

The deduction of the proposition from  \eqref{EqTotalCLT} is standard. Indeed, it suffices to set $$\widetilde A_\eta=\left\{\xi\in A_\eta:\quad  \mathbb P({\cA_k^{h,h'}(\xi)})>0, \quad \mathbb{P}_{\cA_k^{h,h'}(\xi)}(\tau_k-1\notin I_k)\leq \sqrt{\mathbb{P}_{A_\eta}( \tau_k-1\notin I_k)}\right\}.$$ Then $\mathbb P_{A_\eta}(A_\eta\backslash \widetilde A_\eta)\leq\sqrt{\mathbb{P}_{A_\eta}( \tau_k-1\notin I_k)}=o_{h_0,\textbf{p}} ^{k\to\infty}(1)$.
\end{proof}

\hide{
To pass from Lemma \ref{Lem equid central} to Proposition \ref{Prop equid central} we need the bounded distortion property for non-linear IFS's. 

\begin{Lemma}\label{LemBoundedDistortion}\cite{Peres2000Solomyak}*{inequality (1.3)}There exists $L>1$, such that for all $m\in\mathbb N$ and $\eta\in\mathcal A^m$,  and $x,y\in[0,1]$, 
$$\frac{|f'_\eta(x)|}{|f'_\eta(y|}\in[\frac1L,L].$$  \end{Lemma}

\begin{Corollary}\label{CorCLTDistortion}If two finite words $\eta'$, $\zeta'$ satisfy  $\mathbb{P}(A_{k,\eta,\eta'}^{k'})>0$ and $\mathbb{P}(A_{k,\eta,\zeta'}^{k'})>0$. Then for $a=\frac{2\log L}{D}$,  $$\mathbb{P}_{A_{k,\eta,\eta'}^{k'}}( \tau_k -1 \notin I_k )\leq\mathbb{P}_{A_{k,\eta,\zeta'}^{k'}}( \tau_k -1  \notin  [k- \sqrt{k\log k}+ a, k+\sqrt{k\log k}- a]).$$\end{Corollary}
\begin{proof}For $\omega'\in\mathcal A^{\mathbb N}$ and define the event $$A_{k,\eta,\omega'}=\{\omega\in A_\eta, \sigma^{\tau_k(\omega)-1}(\omega)=\omega'\}.$$ Since $A_{k,\eta,\eta'}^{k'}=\bigcup_{\omega'\in A_{\eta'}^{k'}}A_{k,\eta,\omega'}$, it suffices to show for $\mathbb P$-a.e.  $\omega'\in A_{\eta'}^{k'}$, \begin{equation}\label{EqProp equi central 1}\mathbb{P}_{A_{k,\eta,\omega}^{k'}}( \tau_k -1 \notin I_k )= o^{k\to\infty}_{p}(1).\end{equation}

Suppose $\omega',\xi'\in\mathcal A^{\mathbb N}$ and $\mathbb{P}_{A_{k,\eta,\omega'}^{k'}}( \tau_k -1 \notin I_k )=p>0$.  Define a map $\pi_{\xi'}:A_{k,\eta,\omega'}\to\mathcal A^{\mathbb N}$ by $$\pi_{\xi'}(\omega)=(\omega_1,\cdots,\omega_{\tau_k(\omega)-1},\xi').$$\end{proof}
}

We now take care of the second term in \eqref{Eq Theorem equid}:

\begin{Proposition}\label{Prop equid local} In the setting of Theorem \ref{Theorem equid}, for all $\xi$ in the set $\widetilde A_\eta$ from Proposition \ref{Prop equid central}, $$\sum_{m \in I_k} \mathbb{P}_{\cA_k^{h,h'}(\xi)}(\tau_k = m+1, S_{\tau_k}\in J)=\Gamma_{\cA_k^{h,h'}(\xi)}(J)  +o^{k\to\infty}_{h_0,\textbf{p}}(1).$$\end{Proposition}

\begin{proof} Let $\eta'$ be the finite word such that ${\cA_k^{h,h'}(\xi)}=A_{k,\eta,\eta'}$. We first notice that 
\begin{equation}\label{Eq equid local 1}\begin{aligned}&\sum_{m \in I_k} \mathbb{P}_{A_{k,\eta,\eta'}}  (\tau_k = m+1, S_{\tau_k}\in J)\\
=&\frac{\sum_{m \in I_k}\mathbb{P}\big(\omega\in A_{k,\eta,\eta'}, \, \tau_k =m+1, \, S_{m+1} \in J\big)}{\mathbb{P}(A_{k,\eta,\eta'})}.\\
\end{aligned}\end{equation}

Each summand in the numerator can be written as 
\begin{equation}\label{Eq equid local 2}\begin{aligned}&\mathbb{P}\big(\omega\in A_{k,\eta,\eta'},\quad \tau_k =m+1,\quad S_{m+1} \in J\big)\\
=&\mathbb{P}(\omega\in A_\eta\cap\sigma^{-m}(A_{\eta'}),\quad S_m<k\chi,\quad S_{m+1}\in J)\\
=&\int_{A_{\eta'}}\mathbb{P}_{\sigma^{-m}(\{\omega'\})}\big(\omega\in A_\eta,\quad S_m<k\chi,\quad  S_{m+1} \in J\big)d\mathbb P (\sigma^{-m}(\{\omega'\}))\\
=&\int_{A_{\eta'}^{k'}}\mathbb{P}_{\sigma^{-m}(\{\omega'\})}\big(\omega\in A_\eta,\quad S_m<k\chi,\quad S_m+X_1(\omega') \in J\big)d\mathbb P (\omega')\\
=&\int_{A_{\eta'}^{k'}}\mathbb{P}_{\sigma^{-m}(\{\omega'\})}\big(\omega\in A_\eta,\quad S_m \in J^{\omega'}\big)d\mathbb P (\omega')\\\end{aligned}\end{equation}
where the interval $J^{\omega'}$ is defined by 
\begin{equation}\label{EqIntervalShift}J^{\omega'}:=[k\chi-X_1(\omega'),k\chi)\cap U_{-X_1(\omega')}J.\end{equation}

Fix an arbitrarily small $\epsilon>0$.  Since $m\in I_k$ then by Theorem \ref{B-Q LLT general} for all translates $W\subseteq J^{\omega'}$ of the form $[w,w+\epsilon\chi)$, 
\begin{equation}\label{EqLLT}\mathbb{P}_{\sigma^{-m}(\{\omega'\})}\big(\omega\in A_\eta, \, S_m \in W\big)=G_{\sqrt m r}(w-m\chi)\cdot \mathbb{P}(A_\eta)\cdot   \epsilon\chi\cdot (1+o^{m\to\infty}_{\epsilon,h_0,\textbf{p}} (1))\end{equation} 
where we recall that $G_s(\cdot )$ stands for the density of $N(0,s^2)$, $r>0$ is as in Theorem \ref{B-Q CLT}, and $o^{k\to\infty}_{\epsilon,h_0,\textbf{p}} (1)$ in \eqref{EqLLT} depends only on $\epsilon$, $h_0$, and $\textbf{p}$, and is uniform in $\omega'$. Because $m\in I_k$, $m\to\infty$ if and only if $k\to\infty$, so we know by Lemma \ref{LemGaussian} below that 
$$\mathbb{P}_{\sigma^{-m}(\omega')}(\omega\in A_\eta,\quad S_m \in W)=G_{\sqrt k r}((m-k+\beta)\chi) \cdot  \mathbb{P}(A_\eta) \cdot  \epsilon\chi\cdot (1+o^{k\to\infty}_{\epsilon,h_0,\textbf{p}}  (1))$$ for all $\omega'\in A_{\eta'}$, $m\in I_k$ and $\beta\in[0,1)$ as $k\to\infty$.

Now, since the interval $J^{\omega'}$ contains $\lfloor\frac{\lambda(J^{\omega'})}{\epsilon\chi}\rfloor$ disjoint intervals of the form $[w,w+\epsilon\chi)$ and is covered by $\lceil\frac{\lambda(J^{\omega'})}{\epsilon\chi}\rceil$ such intervals, we know that for all  $m\in I_k$ and $\beta\in[0,1)$,
\begin{equation}\label{EqLLTComponent}\begin{aligned}&\mathbb{P}_{\sigma^{-m}(\omega')}(\omega\in A_\eta,\quad S_m \in  J^{\omega'} )\\
=&G_{\sqrt k r}((m-k+\beta)\chi) \cdot  \mathbb{P}(A_\eta)\cdot\Big(\frac{\lambda(J^{\omega'})}{\epsilon\chi}+O(1)\Big)\cdot\epsilon\chi\cdot (1+o^{k\to\infty}_{\epsilon,h_0,\textbf{p}}  (1))\\
=&G_{\sqrt k r}((m-k+\beta)\chi) \cdot  \mathbb{P}(A_\eta)\cdot\big(\lambda( J^{\omega'})+O(\epsilon\chi)+o^{k\to\infty}_{\epsilon,h_0,\textbf{p}} (1)\big),\end{aligned}\end{equation} where the implied constant in the $O(\epsilon\chi)$ is $1$: the term represented by $O(\epsilon\chi)$ is of absolute value bounded by $\epsilon\chi$. The error term $o^{k\to\infty}_{\epsilon,h_0,p} (1)$ is uniform in $J$, $\omega'$ and $\beta$.

Integrating \eqref{EqLLTComponent} inside \eqref{Eq equid local 2} leads to
$$\begin{aligned}&\mathbb{P}(\omega\in A_{k,\eta,\eta'},\quad \tau_k=m+1, \quad S_{m+1}\in J)\\
=&\int_{A_{\eta'}}G_{\sqrt k r}((m-k+\beta)\chi) \cdot  \mathbb{P}(A_\eta) \cdot  \big(\lambda( J^{\omega'})+O(\epsilon\chi)+o^{k\to\infty}_{\epsilon,h_0,\textbf{p}} (1)\big)d\mathbb{P}(\omega')\\
=&G_{\sqrt k r}((m-k+\beta)\chi) \cdot  \mathbb{P}(A_\eta) \cdot  \mathbb{P}(A_{\eta'})\cdot\big(\mathbb E_{\omega'\in A_\eta'}(\lambda( J^{\omega'}))+O(\epsilon\chi)+o^{k\to\infty}_{\epsilon,h_0,\textbf{p}} (1)\big).\end{aligned}$$

By summing over $m\in I_k$ and integrating over $\beta\in[0,1)$, we obtain
\begin{equation}\label{EqStopLLT1}\begin{aligned}&\sum_{m  \in I_k}\mathbb{P}(\omega\in A_{k,\eta,\eta'},\quad \tau_k=m+1,\quad S_{m+1}\in J)\\
=&\left(\int_{k-\sqrt {k\log k}}^{k+\sqrt {k\log k}+1}G_{\sqrt k r}((t-k)\chi) d t\right)\\
&\quad\cdot\mathbb{P}(A_\eta)\mathbb{P}(A_{\eta'})\cdot\big(\mathbb E_{\omega'\in A_\eta'}(\lambda( J^{\omega'}))+O(\epsilon\chi)+o^{k\to\infty}_{\epsilon,h_0,\textbf{p}} (1)\big)\\
=&\left(\int_{-\sqrt {k\log k}}^{\sqrt {k\log k}+1}G_{\sqrt k r}(t\chi) d t\right)\cdot\mathbb{P}(A_\eta)\mathbb{P}(A_{\eta'})\cdot\big(\mathbb E_{\omega'\in A_\eta'}(\lambda( J^{\omega'}))+O(\epsilon\chi)+o^{k\to\infty}_{\epsilon,h_0,\textbf{p}} (1)\big)\\
=&\left(\int_{-\sqrt {\log k}}^{\sqrt {\log k}+\frac 1{\sqrt k}}G_{r}(t\chi) d t\right)\cdot\mathbb{P}(A_\eta)\mathbb{P}(A_{\eta'})\cdot\big(\mathbb E_{\omega'\in A_\eta'}(\lambda( J^{\omega'}))+O(\epsilon\chi)+o^{k\to\infty}_{\epsilon,h_0,\textbf{p}} (1)\big)\\
=&(\frac1\chi-o^{k\to\infty}_p (1))\cdot\mathbb{P}(A_\eta)\mathbb{P}(A_{\eta'})\cdot\big(\mathbb E_{\omega'\in A_\eta'}(\lambda( J^{\omega'}))+O(\epsilon\chi)+o^{k\to\infty}_{\epsilon,h_0,\textbf{p}} (1)\big)\\
=&\mathbb{P}(A_\eta)\mathbb{P}(A_{\eta'})\cdot\left(\frac1\chi\mathbb E_{\omega'\in A_\eta'}(\lambda( J^{\omega'}))+O(\epsilon)+o^{k\to\infty}_{\epsilon,h_0,\textbf{p}} (1)\right)\end{aligned}\end{equation} as $k\to\infty$. The term represented by $O(\epsilon)$ is uniformly  bounded by $\epsilon$ in absolute value.

As $\epsilon>0$ is arbitrary, this implies that for all intervals $J\subset[k\chi,k\chi+D')$, 
\begin{equation}\label{EqLLTNumerator}\begin{aligned}&\sum_{m  \in I_k}\mathbb{P}(\omega\in A_{k,\eta,\eta'},\quad \tau_k=m+1,\quad S_{m+1}\in J)\\
=&\mathbb{P}(A_\eta)\mathbb{P}(A_{\eta'})\cdot\left(\frac1\chi\mathbb E_{\omega'\in A_{\eta'}}(\lambda( J^{\omega'}))+o^{k\to\infty}_{h_0,\textbf{p}} (1)\right),\end{aligned}\end{equation} where $J^{\omega'}$ is defined by \eqref{EqIntervalShift} and the error term $o^{k\to\infty}_{h_0,\textbf{p}} (1)$ is uniform in $J$, $\omega'$ and $\beta$.

Consider the special case of $J=[k\chi,k\chi+D')$, where $J^{\omega'}=[k\chi-X_1(\omega'),k\chi)$.  Because the event $\{\tau_k=m+1, S_{\tau_k}\in [k\chi,k\chi+D')\}$ coincides with $\{\tau_k=m+1\}$, we obtain 
\begin{equation}\label{EqLLTDenominator}\begin{aligned}&\sum_{m  \in I_k}\mathbb{P}(\omega\in A_{k,\eta,\eta'}, \tau_k=m+1)\\
=&\mathbb{P}(A_\eta)\mathbb{P}(A_{\eta'})\cdot\left(\frac1\chi\mathbb E_{\omega'\in A_{\eta'}}\big(\lambda([k\chi-X_1(\omega'),k\chi))\big)+o^{k\to\infty}_{h_0,\textbf{p}} (1)\right)\\
=&\mathbb{P}(A_\eta)\mathbb{P}(A_{\eta'})\cdot\left(\frac1\chi\mathbb E_{\omega'\in A_{\eta'}} X_1(\omega')+o^{k\to\infty}_{h_0,\textbf{p}} (1)\right).\end{aligned}\end{equation}

Therefore, by \eqref{Eq equid local 1}, \eqref{EqLLTNumerator} and \eqref{EqLLTDenominator}, \begin{equation}\label{Eq equid local 3}\begin{aligned}&\sum_{m \in I_k} \mathbb{P}_{A_{k,\eta,\eta'}}  (\tau_k = m+1, S_{\tau_k}\in J)\\
=&\frac{\sum_{m \in I_k}\mathbb{P}\big(\omega\in A_{k,\eta,\eta'}, \tau_k =m+1, S_{m+1} \in J\big)}{\sum_{m \in I_k}\mathbb{P}\big(\omega\in A_{k,\eta,\eta'}, \tau_k =m+1\big)}\cdot \frac{\sum_{m \in I_k}\mathbb{P}\big(\omega\in A_{k,\eta,\eta'}, \tau_k =m+1\big)}{\mathbb{P}(A_{k,\eta,\eta'})}\\
=&\frac{\frac1\chi\mathbb E_{\omega'\in A_{\eta'}}(\lambda( J^{\omega'}))+o^{k\to\infty}_{h_0,p} (1)}{\frac1\chi\mathbb E_{\omega'\in A_{\eta'}} X_1(\omega')+o^{k\to\infty}_{h_0,\textbf{p}} (1)}\cdot \mathbb P_{A_{k,\eta,\eta'}}(\tau_k-1\in I_k)\\
=&\frac{\frac1\chi\mathbb E_{\omega'\in A_{\eta'}}(\lambda( J^{\omega'}))+o^{k\to\infty}_{h_0,p} (1)}{\frac1\chi\mathbb E_{\omega'\in A_{\eta'}} X_1(\omega')+o^{k\to\infty}_{h_0,\textbf{p}} (1)}\cdot\big(1-o^{k\to\infty}_{h_0,p} (1)\big)\hskip8mm\text{(since $\xi\in\widetilde A_\eta$ and $A_{k,\eta,\eta'}=\cA_k^{h,h'}(\xi)$\ )}\\
=&\left(\frac{\mathbb E_{\omega'\in A_{\eta'}}\lambda( J^{\omega'})}{\mathbb E_{\omega'\in A_{\eta'}} X_1(\omega')}+o_{h_0,p}^{k\to\infty}(1)\right)\cdot\big(1-o^{k\to\infty}_{h_0,\textbf{p}} (1)\big)\hskip8mm\text{(since $0<D\leq\mathbb E_{\omega'\in A_{\eta'}} X_1(\omega')\leq D'$)}\\
=&\frac{\mathbb E_{\omega'\in A_\eta'}\lambda( J^{\omega'})}{\mathbb E_{\omega'\in A_{\eta'}} X_1(\omega')}+o_{h_0,\textbf{p}}^{k\to\infty}(1).
\end{aligned}\end{equation}
To conclude, it suffices to notice that $\frac{\mathbb E_{\omega'\in A_{\eta'}}\lambda( J^{\omega'})}{\mathbb E_{\omega'\in A_{\eta'}} X_1(\omega')}$ is exactly $\Gamma_{A_{k,\eta,\eta'}}(J)$.
\end{proof}

\begin{Lemma}\label{LemGaussian}For $m\in I_k$, $w\in [k\chi-D',k\chi]$ and $\beta\in[0,1)$, $$\left|\frac{G_{\sqrt m r}(w-m\chi)}{G_{\sqrt kr}((m-k+\beta)\chi)}-1\right|\lesssim_{\mathbf{p}} k^{-\frac12}(\log k)^{\frac 32}\text{ as }k\to\infty,$$
\end{Lemma}

\begin{proof}Recall $G_s(x)=\frac 1{\sqrt{2\pi}s}\exp(-\frac{x^2}{2s^2})$ and $\log G_s(x)=-\log\sqrt {2\pi}-\log s-\frac{x^2}{2s^2}$. 

So as $k\to\infty$, $$\begin{aligned}&\big|\log G_{\sqrt m r}(w-m\chi)-\log G_{\sqrt k r}(w-m\chi)\big|\\
\leq &\big|\log(\sqrt m r)-\log(\sqrt k r)|+\frac{(w-m\chi)^2}{2r^2}\big|\frac 1m-\frac1k\big|
\\
\leq& \frac12\big|\log\frac mk\big|+\frac{(|k-m|\chi+D')^2}{2mr^2}\big|\frac mk-1\big|\\
\lesssim&\big|\frac mk-1\big|+\frac{(\sqrt{k\log k}+1)^2}{k}\big|\frac mk-1\big|\lesssim(\log k)\big|\frac mk-1\big|\\
\lesssim&(\log k)\frac{\sqrt{\log k}}{\sqrt k}=k^{-\frac12}(\log k)^{\frac 32},\end{aligned}$$
where the implied constant depends only on $\chi$, $D'$ and $r$, and hence only on $p$.

Moreover, $$\begin{aligned}&\big|\log G_{\sqrt k r}(w-m\chi)-\log G_{\sqrt kr}((m-k+\beta)\chi)\big|\\
= &\frac1{2kr^2}\big|(w-m\chi)^2-(k\chi- m\chi+\beta \chi)^2\big|\\
= &\frac1{2kr^2}\cdot |w-k\chi-\beta\chi|\cdot |2(k-m-\beta)\chi+(w-k\chi-\beta\chi)|\\
\leq &\frac1{2kr^2}\cdot (D'+\beta\chi)\cdot\big(2(\sqrt{k\log k}+1)\chi+(D'+\beta\chi)\big)\\
\lesssim&k^{-\frac12}(\log k)^{\frac 12},\end{aligned}$$
where the implied constant similarly depends only on $p$.

Combining the two inequalities above shows  $$\big|\log G_{\sqrt m r}(w-m\chi)-\log G_{\sqrt kr}((m-k+\beta)\chi)\big|\lesssim k^{-\frac12}(\log k)^{\frac 32},$$ which in turn implies the lemma. \end{proof}

\subsection{Proof of Theorem \ref{Theorem lin equid}}
\subsubsection{Fixing parameters and preliminary steps} \label{Section para}

Fix $\epsilon>0$ and choose $\ell=\ell(\epsilon,\bfp)$ such that $e^{-\ell\chi}<\epsilon^{\frac 2\gamma}$. In this proof, we will view $\epsilon$, $\bfp$ and $\ell$ as fixed inputs, while $k$ and $h$ are varying.

 Suppose $$\min(h,k-h)>\ell$$ and fix $A_\eta\in\cA^h$. Decompose $\eta=\eta^\#\eta^*$ where $A_{\eta^\#}\in\cA^{h-\ell}$. Then 
 $$-\log\max_{x\in I}|f'_{\eta^*}(x)|\in[\ell\chi-O_\bfp(1),\ell\chi+O_\bfp(1)]$$
 by bounded distortion (Theorem \ref{Bounded distortiion theorem}). Define real values $h^*=h^*(\epsilon,\bfp,\eta,h)$ and $k^*=k^*(\epsilon,\bfp,\eta,h,k)$ by $$h^*=\frac{-\log\max_{x\in I}|f'_{\eta^*}(x)|}\chi\in[\ell-O_\bfp(1),\ell+O_\bfp(1)],$$ 
$$k^*=k+\frac{\log\max_{x\in I}|f'_{\eta^\#}(f_{\eta^*}(x))|-C \epsilon^2}\chi\in [k-h+\ell-O_\bfp(1),k-h+\ell+O_\bfp(1)],$$ where $C=C(\bfp)$ stands for the implied constant in Lemma \ref{Lemma Defect Lin} below.

It is important that $h^*$ is uniformly bounded even though it depends on $\eta$ and $h$. It follows that $k^*-(k-h)=O_{\epsilon,\bfp}(1)$, and that $k^*<k$ if $h>O_{\epsilon,\bfp}(1)$.

We will keep the value $h'$.

Note that $A_\eta=\eta^\#\cdot A_{\eta^*}$, where we use the notation $$\theta \cdot Y=\{\theta\theta^*:\theta^*\in Y\}.$$ Moreover, by the choice of $h^*$, $A_{\eta^*}\in \cA^{h^*}$.

\begin{Lemma}\label{Lemma Defect Lin}For all $\xi^*\in A_{\eta^*}$ 
$$0\leq\log\max_{x\in I}|f'_{\eta^\#}(f_{\eta^*}(x))|-\log|f'_{\eta^\#}(x_{\xi^*})|\lesssim_\bfp\epsilon^2.$$
We denote the implied constant by $C=C(\bfp)$.\end{Lemma}
\begin{proof} The inequality $0 \leq$ is obvious because of the $\max$. For the second inequality, since $x_{\xi^*}=f_{\eta^*}(x_{\sigma^{|\eta^*|}\xi^*})$, by bounded distortion (Theorem \ref{Bounded distortiion theorem}), 
$$|x_{\xi^*}-f_{\eta^*}(x)|\leq  |f_{\eta^*}(I)| \lesssim_\bfp e^{-\ell\chi}<\epsilon^{\frac 2\gamma}.$$
The  Lemma follows by an application of Lemma \ref{Lemma C1 lin}.\end{proof}

\subsubsection{Specifying the exceptional set $\overline A_{k,\eta}^{h,h'}$}
Let $\tA_{k^*,\eta^*}^{h^*,h'}$ be given by Theorem \ref{Theorem equid}, for the corresponding parameters as in Section \ref{Section para} (here it is important that $h^*$ is uniformly bounded). For clarity, we repeat the properties satisfied by this set:
\begin{enumerate} [label=(\roman*$^*$)]
\item\label{lin equid 1*} $\bP(\tA_{k^*,\eta^*}^{h^*,h'})\geq \bP(A_{\eta^*})\cdot (1-o_{\ell+O_\bfp(1),\bfp}^{k^*\to\infty}(1))=\bP(A_{\eta^*})\cdot (1-o_{\epsilon,\bfp}^{k-h\to\infty}(1))$.
\item\label{lin equid 2*} For all $\omega^*\in \tA_{k^*,\eta^*}^{h^*,h'}$, $\mathbb{P}(\cA_{k^*}^{h^*,h'}(\omega^*))>0$.
\item\label{lin equid 3*}  For all $\omega^*\in \tA_{k^*,\eta^*}^{h^*,h'}$ and for any sub-interval $J\subseteq  [k\chi, k\chi+D']$,
$$
 \mathbb{P}_{\cA_{k^*}^{h^*, h' }(\omega^*)}(S_{\tau_{k^*}}\in J)= \Gamma_{\cA_{k^*}^{h^*,h'}(\omega^*)} (J)  + o_{\epsilon,\bfp}^{k-h\to\infty}(1).
$$
\end{enumerate}
Define an exceptional subset by 

$$Z_{k,\eta}^{h,h'}=\{\xi^*\in A_{\eta^*}:\quad \tau_k(\eta^\#\xi^*)\neq |\eta^\#|+\tau_{k^*}(\xi^*)\}.$$

\begin{Lemma} $\bP(Z_{k,\eta}^{h,h'})\leq \bP(A_{\eta^*})\cdot (O_\bfp(\epsilon^2)+o_{\epsilon,\bfp}^{k-h\to\infty}(1))$.\end{Lemma}
\begin{proof} We need to show $\bP_{A_{\eta^*}}(Z_{k,\eta}^{h,h'})\leq O_\bfp(\epsilon^2)+o_{\epsilon,\bfp}^{k-h\to\infty}(1)$. Since $\bP_{A_{\eta^*}}(\tA_{k^*,\eta^*}^{h^*,h'})\geq 1-o_{\epsilon,\bfp}^{k-h\to\infty}(1)$, it suffices to show $\bP_{A_{\eta^*}}(\tA_{k^*,\eta^*}^{h^*,h'}\cap Z_{k,\eta}^{h,h'})\leq O_\bfp(\epsilon^2)+o_{\epsilon,\bfp}^{k-h\to\infty}(1).$

Now, as $A_{\eta^*}$ is $\cA_{k^*}^{h^*,h'}$-measurable, it suffices to show that for every $\omega^*\in \tA_{k^*,\eta^*}^{h^*,h'}$, 
$$\bP_{\cA_{k^*}^{h^*,h'}(\omega^*)}(\tA_{k^*,\eta^*}^{h^*,h'}\cap Z_{k,\eta}^{h,h'})\leq O_\bfp(\epsilon^2)+o_{\epsilon,\bfp}^{k-h\to\infty}(1).$$

Suppose $\xi^*\in \tA_{k^*,\eta^*}^{h^*,h'}\cap Z_{k,\eta}^{h,h'}$. For simplicity write $a=|\eta^\#|$ and $b=\tau_{k^*}(\xi^*)$, then $S_{b-1}(\xi^*)<k^*\leq S_b(\xi^*)$. Since $\tau_k(\eta^\#\xi^*)\neq a+b$
then one of the following holds: 

\noindent{\it Case 1:} $S_{a+b-1}(\eta^\#\xi^*)\geq k$, which implies $$\begin{aligned}\frac{-\log|f'_{\eta^\#}(x_{\xi^*})|}\chi=&S_a(\eta^\#\xi^*)=S_{a+b-1}(\eta^\#\xi^*)-S_{b-1}(\xi^*)\\
>&k-k^*=\frac{-\log\max_{x\in I}|f'_{\eta^\#}(f_{\eta^*}(x))|+C\epsilon^2}\chi.\end{aligned}$$ Comparing with Lemma \ref{Lemma Defect Lin}, we know this case cannot happen.

\noindent{\it Case 2:} $S_{a+b}(\eta^\#\xi^*)<k$, which implies  $$\begin{aligned}\frac{-\log|f'_{\eta^\#}(x_{\xi^*})|}\chi=&S_a(\eta^\#\xi^*)=S_{a+b}(\eta^\#\xi^*)-S_b(\xi^*)\\
<& k-k^*
=\frac{-\log\max_{x\in I}|f'_{\eta^\#}(f_{\eta^*}(x))|+C\epsilon^2}\chi.\end{aligned}$$
 Comparing with Lemma \ref{Lemma Defect Lin}, we know $S_b(\xi^*)\in [k^*,k^*+C\epsilon^2)$. 
 Since $\Gamma_{A_{k^*,\eta^*,\eta'}}$ is absolutely continuous with a uniformly bounded density (Lemma \ref{Lemma Abs. contin of gamma}),
 $$\Gamma_{A_{k^*,\eta^*,\eta'}}([k^*,k^*+C\epsilon^2))\lesssim_\bfp\epsilon^2$$
uniformly. The lemma follows from property \ref{lin equid 3*} of the set $\tA_{k^*,\eta^*}^{h^*,h'}$.
\end{proof}

Next, set 
$$\begin{aligned}\tZ_{k,\eta}^{h,h'}=&\{\omega\in A_ \eta:\quad \bP_{\cA_k^{h,h'}(\omega)}( \eta^\# \cdot Z_{k,\eta}^{h,h'})\geq \bP_{\cA_{\eta^*}}(Z_{k,\eta}^{h,h'})^{\frac12}\}\\
&\quad \cup\eta^\#\cdot\{\omega^*\in A_{\eta^*}:\quad \bP_{\cA_{k^*}^{h^*,h'}(\omega^*)}(Z_{k,\eta}^{h,h'})\geq \bP_{\cA_{\eta^*}}(Z_{k,\eta}^{h,h'})^{\frac12}\}.\end{aligned}$$
Because $  \bP_{ A_ \eta}(\eta^\# \cdot Z_{k,\eta}^{h,h'})= \bP_{ A_{\eta^*}}(Z_{k,\eta}^{h,h'})$ and $\eta^\#\cdot$ sends $\bP_{\eta^*}$ to $\bP_\eta$, we have 
$$\bP_{A_\eta}(\tZ_{k,\eta}^{h,h'})\leq \bP_{A_{\eta^*}}(Z_{k,\eta}^{h,h'})^{\frac12}= O_\bfp(\epsilon)+o_{\epsilon,\bfp}^{k-h\to\infty}(1),$$ since this bound holds for both components in $\tZ_{k,\eta}^{h,h'}$.

The set $\overline A_{k,\eta}^{h,h'}$ can be now defined as
$$\overline A_{k,\eta}^{h,h'}=\Big(\eta^\#\cdot(\tA_{k^*,\eta^*}^{h^*,h'}\backslash  Z_{k,\eta}^{h,h'})\Big)\backslash \tZ_{k,\eta}^{h,h'}.$$
\subsubsection{Proof of Theorem \ref{Theorem lin equid}}

For property (i), it suffices to prove that if $\min(h,k-h)$ is sufficiently large (depending only on $\epsilon$ and $\mathbf p$), then $\frac{\bP(\overline A_{k,\eta}^{h,h'})}{\bP(A_\eta)}>1-O_\bfp(\epsilon)$. Indeed,
\begin{eqnarray*}
\frac{\bP(\overline A_{k,\eta}^{h,h'})}{\bP(A_\eta)} &=& \frac{\bP(\tA_{k^*,\eta^*}^{h^*,h'}\backslash  Z_{k,\eta}^{h,h'})}{\bP(A_{\eta^*})}-\frac{\bP(\tZ_{k,\eta}^{h,h'} )}{\bP(A_\eta)} \\
&\geq & \Big((1-o_{\bfp}^{k-h\to\infty}(1))-(O_\bfp(\epsilon^2)+o_{\epsilon,\bfp}^{k-h\to\infty}(1))\Big) - (O_\bfp(\epsilon)+o_{\epsilon,\bfp}^{k-h\to\infty}(1))\\
&=& 1- (O_\bfp(\epsilon)+o_{\epsilon,\bfp}^{k-h\to\infty}(1)).
\end{eqnarray*}
This implies the claim.

For property (ii): This property is a formality and always holds after omitting a null set  (in fact, we don't need to omit a null set based on the construction here).

The proof of property (iii) will requires more effort. By adjusting the threshold on $\min(h,k-h)$ based on $\epsilon$, it suffices to prove that $$
 \bP_{\cA_k^{h,h'}(\omega)}(S_{\tau_k}\in J)= \Gamma_{\cA_k^{h,h'}(\omega)} (J)  +O_p(\epsilon)+o^{\min(h,k-h)\to\infty}_{\epsilon,\bfp}(1)
$$ for all $\omega\in \overline A_{k,\eta}^{h,h'}$. In fact, proving $\geq$ instead of $=$ is enough, by applying to both $J$ and $J^c$.

In particular, it suffices to show \begin{equation}\label{lin equid 3 eq 1}\bP_{\cA_k^{h,h'}(\omega)}(\{\xi\in\eta^\#\cdot\{A_{\eta^*}\backslash Z_{k,\eta}^{h,h'}\}:\quad  S_{\tau_k(\xi)}  \in J \})\geq  \Gamma_{\cA_k^{h,h'}(\omega)} (J)-O_p(\epsilon)-o^{k-h\to\infty}_{\epsilon,\bfp}(1).\end{equation}

In the following Lemmas we take a closer look at $\eta^\#\cdot\{A_{\eta^*}\backslash Z_{k,\eta}^{h,h'}\}$:
\begin{Lemma}\label{lem atom approx} If  $\omega\in \overline A_{k,\eta}^{h,h'}$, then $$\cA_k^{h,h'}(\omega)\backslash(\eta^\#\cdot Z_{k,\eta}^{h,h'})=\eta^\#\cdot(\cA_{k^*}^{h^*,h'}(\omega^*)\backslash Z_{k,\eta}^{h,h'}).$$\end{Lemma}
\begin{proof}{\bf The $\subseteq$ direction:} Suppose $\xi\in \cA_k^{h,h'}(\omega)\backslash(\eta^\#\cdot Z_{k,\eta}^{h,h'})$, then both $\xi, \omega \in \eta^\#\cdot(A_{\eta^*}\backslash Z_{k,\eta}^{h,h'})$. Write them respectively as $\eta^\#\xi^*$ and $\eta^\#\omega^*$, then \begin{equation}\label{eq lem atom approx 1}\tau_k(\xi)-\tau_{k^*}(\xi^*)= |\eta^\#| =\tau_k(\omega)-\tau_{k^*}(\omega^*).\end{equation} Hence, $$\sigma^{\tau_{k^*}(\xi^*)}\xi^*=\sigma^{\tau_k(\xi)}\xi\equiv_{\cA^{h'}}\sigma^{\tau_k(\omega)}\omega=\sigma^{\tau_{k^*}(\omega^*)}\omega^*.$$
This shows that $\xi^*\in\cA_{k^*}^{h^*,h'}(\omega^*)$, meaning $\cA_k^{h,h'}(\omega)\backslash(\eta^\#\cdot Z_{k,\eta}^{h,h'})\subseteq \eta^\#\cdot\cA_{k^*}^{h^*,h'}(\omega^*)$. In fact: $$\cA_k^{h,h'}(\omega)\backslash(\eta^\#\cdot Z_{k,\eta}^{h,h'})\subseteq \eta^\#\cdot(\cA_{k^*}^{h^*,h'}(\omega^*)\backslash Z_{k,\eta}^{h,h'}).$$

{\bf The $\supseteq$ direction:} Keep the notations and assume $\xi^*\in \cA_{k^*}^{h^*,h'}(\omega^*)\backslash Z_{k,\eta}^{h,h'}$. The equality \eqref{eq lem atom approx 1} still holds, showing that  $$\sigma^{\tau_k(\xi)}\xi=\sigma^{\tau_{k^*}(\xi^*)}\xi^*\equiv_{\cA^{h'}}\sigma^{\tau_{k^*}(\omega^*)}\omega^*=\sigma^{\tau_k(\omega)}\omega.$$ Hence $\eta^\#\cdot(\cA_{k^*}^{h^*,h'}(\omega^*)\backslash Z_{k,\eta}^{h,h'})\subseteq \cA_k^{h,h'}(\omega)$, or more precisely, 
$$\eta^\#\cdot(\cA_{k^*}^{h^*,h'}(\omega^*)\backslash Z_{k,\eta}^{h,h'})\subseteq \cA_k^{h,h'}(\omega)\backslash(\eta^\#\cdot Z_{k,\eta}^{h,h'}).$$
\end{proof}

\begin{Lemma}\label{lem prefix approx} If $\omega=\eta^\#\omega^*\in \overline A_{k,\eta}^{h,h'}$ and $\xi=\eta^\#\xi^*\in\cA_k^{h,h'}(\omega)\backslash(\eta^\#\cdot Z_{k,\eta}^{h,h'})$, then $$S_{\tau_k}(\xi)-S_{\tau_{k^*}}(\xi^*)\in [k-k^*-\frac{C\epsilon^2}\chi,k-k^*]$$
where $C$ is as in Lemma \ref{Lemma Defect Lin}.  \end{Lemma}

\begin{proof} By the proof of the previous lemma,
$$\begin{aligned}&S_{\tau_k}(\xi)-S_{\tau_{k^*}}(\xi^*)\\
=&S_{|\eta^\#|}(\xi)=\frac{-\log|f'_{\eta^\#}(x_{\xi^*})|}\chi\\
\in&[\frac{-\log\max_{x\in I}|f'_{\eta^\#}(f_{\eta^*}(x))|}\chi,\frac{-\log\max_{x\in I}|f'_{\eta^\#}(f_{\eta^*}(x))|+C\epsilon^2}\chi]\\
=&[k-k^*-\frac{C\epsilon^2}\chi,k-k^*].\end{aligned}$$\end{proof}

We proceed to compute a few auxiliary bounds:

Since $\omega\notin \tZ_{k,\eta}^{h,h'}$, 
\begin{equation}\label{lin equid 3 eq 2}\bP_{\cA_k^{h,h'}(\omega)}\big(\cA_k^{h,h'}(\omega)\backslash(\eta^\#\cdot Z_{k,\eta}^{h,h'})\big) \geq  1-\bP_{A_{\eta^*}}(Z_{k,\eta}^{h,h'})^{\frac12}\geq 1- O_\bfp(\epsilon)-o_{\epsilon,\bfp}^{k\to\infty}(1). \end{equation}

Moreover, by Lemma \ref{lem atom approx},  Lemma \ref{lem prefix approx} and property \ref{lin equid 3*},

\begin{equation}\label{lin equid 3 eq 3}\begin{aligned} 
 & \bP_{\cA_k^{h,h'}(\omega)\backslash(\eta^\#\cdot Z_{k,\eta}^{h,h'})}\big(\{\xi\in\cA_k^{h,h'}(\omega)\backslash(\eta^\#\cdot Z_{k,\eta}^{h,h'}): S_{\tau_k(\xi)}\in J \}\big)\\
\geq &  \bP_{\eta^\#\cdot\cA_{k^*}^{h^*,h'}(\omega^*)}\big(\{\xi\in\cA_k^{h,h'}(\omega)\backslash(\eta^\#\cdot Z_{k,\eta}^{h,h'}): S_{\tau_k(\xi)}\in J \}\big) \\
\geq &  \bP_{\eta^\#\cdot\cA_{k^*}^{h^*,h'}(\omega^*)}\big(\{\xi\in\eta^\#\cdot(\cA_{k^*}^{h^*,h'}(\omega^*)\backslash Z_{k,\eta}^{h,h'}):\\
&\quad S_{\tau_{k^*}(\xi^*)}\in \big(J-(k-k^*-\frac{C\epsilon^2}\chi)\big)\cap \big(J-(k-k^*)\big)\}\big)\\ 
\geq & \bP_{\cA_{k^*}^{h^*,h'}(\omega^*)}(S_{\tau_{k^*}}\in J^*)-\bP_{\cA_{k^*}^{h^*,h'}(\omega^*)}(Z_{k,\eta}^{h,h'})\\
\geq &\big(\Gamma_{\cA_{k^*}^{h^*,h'}(\omega^*)} (J^*)  - o_{\epsilon,\bfp}^{k-h\to\infty}(1)\big) -\bP_{A_{\eta^*}}(Z_{k,\eta}^{h,h'})^{\frac12}\\
\geq &\Gamma_{\cA_{k^*}^{h^*,h'}(\omega^*)} (J^*) -O_\bfp(\epsilon)-o_{\epsilon,\bfp}^{k-h\to\infty}(1),
 \end{aligned}\end{equation}
Here  $J^*$ denotes, provided that $J$ is a subinterval of $[k\chi,k\chi+D']$, the interval
$$ J^*=\big(J-(k-k^*-\frac{C\epsilon^2}\chi)\big)\cap \big(J-(k-k^*)\big)\}\big)\subseteq [k^*\chi,k^*\chi+D'].$$ Since $\big| \big(J-(k-k^*)\big)\backslash J^*\big|\leq\frac{C\epsilon^2}\chi$, 
$$\Gamma_{\cA_{k^*}^{h^*,h'}(\omega^*)} (J^*)\geq \Gamma_{\cA_{k^*}^{h^*,h'}(\omega^*)}\big (J-(k-k^*)\big)-O_\bfp(\epsilon^2)$$ by the uniform absolute continuity of the probability distribution $\Gamma_{\cA_{k^*}^{h^*,h'}(\omega^*)}$ (Lemma \ref{Lemma Abs. contin of gamma}).

Finally, since $\sigma^{\tau_{k^*}(\omega^*)}\omega^*=\sigma^{\tau_k(\omega)}\omega$ (i.e. the $h'$ components are the same $\eta'$), one can check by construction that $\Gamma_{\cA_{k^*}^{h^*,h'}(\omega^*)}= U_{-(k-k^*)\chi} \Gamma_{\cA_k^{h,h'}(\omega)}$, where for $x\in \mathbb{R}$ , $U_x:\mathbb{R}\rightarrow \mathbb{R}$ is the translation by $x$.
So
\begin{equation}\label{lin equid 3 eq 4}\Gamma_{\cA_{k^*}^{h^*,h'}(\omega^*)} (J^*)\geq  \Gamma_{\cA_k^{k,h'}(\omega)} (J)-O_\bfp(\epsilon^2).\end{equation}

To obtain \eqref{lin equid 3 eq 1}, plug \eqref{lin equid 3 eq 4} into \eqref{lin equid 3 eq 3}, then multiply by \eqref{lin equid 3 eq 2}. This completes the proof of Property \ref{lin equid 3}.

Theorem \ref{Theorem lin equid} is established.

\section{Proof of Theorem \ref{Theorem main tech} part (1)} \label{Section proof 1}
In this Section we prove Theorem \ref{Theorem main tech} part (1). We  require a preliminary step, which is an adaptation of Theorem \ref{Theorem equid} for Fourier modes. This is the content of the next subsection:
\subsection{Application of Theorem \ref{Theorem equid} to Fourier modes} \label{Section Fourier modes}
Fix a Borel probability measure $\rho\in \mathcal{P}(\mathbb{R})$.  For every $q\in \mathbb{R}$ we define a function $g_{q,\rho}:\mathbb{R}\rightarrow \mathbb{R}$ via
$$g_{q,\rho} (t) = \left| \mathcal{F}_q \left( M_{e^{-t}} \rho \right) \right|^2$$
where we recall that $M_s(x)=s\cdot x$ for any $s,x\in \mathbb{R}$.

Next, fixing $h_0=0$ in Theorem \ref{Theorem equid} and assuming $k\in \mathbb{N}$, we define the sequence $o_k :=o_{0,\textbf{p}} ^{k\rightarrow \infty}$. Notice that the assumption $h_0=0$ means that $A_\eta$ is the entire symbolic space $\mathcal{A}^\mathbb{N}$ (since the only word of length $0$ is the empty word). The following Theorem is needed in conjunction with Theorem \ref{Theorem equid}, since in practice we will need a version of Theorem \ref{Theorem equid} for functions rather than intervals.

\begin{theorem} \label{Theorem equid 2}
Let $q$ be large, let $C>1$, and let $k=k(q)$ be defined implicitly as an integer satisfying
\begin{equation} \label{Eq relation between q and k}
|q|= \Theta_{C} \left( o_k ^{-\frac{1}{4}} e^{(k+h') \chi} \right)
\end{equation}
where  $h'= \sqrt{k}$. Let $\rho\in \mathcal{P}(\mathbb{R})$ be a measure such that 
$$\diam \left( \supp \left( \rho \right) \right) = O(e^{-h'\chi}).$$
Then for every $\xi \in \tA_{k,\eta}^{h,h'} \subseteq \mathcal{A}^\mathbb{N}$ as in Theorem \ref{Theorem equid}, recalling that here $h=h_0=0$,  we have
\begin{equation*}
\left| \mathbb{E}_{\cA_k^{h,h'}(\xi)} \left[ g_{q,\rho} (S_{\tau_k(\omega)}) \right] - \int_{k\chi} ^{k\chi +D'} g_{q,\rho}(x) d \Gamma_{\cA_k^{h,h'}(\xi)} (x) \right|\leq O(o_k ^{\frac{1}{4}}).
\end{equation*}
\end{theorem}
\begin{proof}
We first claim that the function $g_{q,\rho} (t)$ is $4\pi qe^{-\chi k} \cdot \diam \left( \supp \left( \rho \right) \right)$ Lipschitz, whenever $t\in [k\chi, k\chi +D']$. Indeed, since the complex exponential is a $1$-Lipschitz function, for any $x,y\in \supp(\rho)$ and $t,s\in [k\chi, k\chi +D']$  we have
$$|\exp(2\pi i q e^{-t} (x-y)) - \exp(2\pi i q e^{-s} (x-y))| \leq |2 \pi q (x-y)|\cdot |e^{-t}-e^{-s}|\leq 4 \pi q \diam \left( \supp \left( \rho \right) \right) \cdot e^{-k \chi}.$$
Since the $L^1$ norm is always bounded by the $L^\infty$ norm, and since
$$g_{q,\rho} (t) = \left| \mathcal{F}_q \left( M_{e^{-t}} \rho \right) \right|^2 = \int \int \exp(2\pi i q e^{-t} (x-y))d\rho(x) d\rho(y) $$
the Claim follows.

Recalling that  $k=k(q)$ satisfies
$$q=\Theta_{C} \left( o_k ^{-\frac{1}{4}} e^{(k+h') \chi} \right)$$
and that
$$\diam \left( \supp \left( \rho \right) \right) = O(e^{-h'\chi})$$
it follows that the function
$$t\in [k\chi, k\chi +D'] \mapsto g_{q,\rho} (t)$$
is $o_k ^{-\frac{1}{4}}$ Lipschitz (up to a constant universal multiplicative factor). Therefore, there exists a step function $\psi: [k\chi, k\chi +D'] \rightarrow \mathbb{R}$ such that:
\begin{enumerate}
\item $\psi$ consists of $o_k ^{-\frac{2}{4}}$ steps (indicators of intervals).

\item $|| \psi - g_{q,\rho}||_{\infty} \leq o_k ^{\frac{1}{4}}$ on the interval $[k\chi, k\chi +D']$.
\end{enumerate}
Thus,
 $$\left| \mathbb{E}_{\cA_k^{h,h'}(\xi)} \left[ g_{q,\rho} (S_{\tau_k(\omega)}) \right] - \int_{k\chi} ^{k\chi +D'} g_{q,\rho}(x) d \Gamma_{\cA_k^{h,h'}(\xi)} (x) \right|$$
\begin{equation} \label{Term 1}
\leq \left| \mathbb{E}_{\cA_k^{h,h'}(\xi)} \left[ g_{q,\rho} (S_{\tau_k(\omega)}) \right]  - \mathbb{E}_{\cA_k^{h,h'}(\xi)} \left[ \psi (S_{\tau_k(\omega)}) \right] \right|
\end{equation}
\begin{equation} \label{Term 2}
+ \left| \mathbb{E}_{\cA_k^{h,h'}(\xi)} \left[ \psi (S_{\tau_k(\omega)}) \right]  - \int_{k\chi} ^{k\chi +D'} \psi(x) d \Gamma_{\cA_k^{h,h'}(\xi)} (x)  \right|
\end{equation}
\begin{equation} \label{Term 3}
+ \left| \int_{k\chi} ^{k\chi +D'} \psi(x) d \Gamma_{\cA_k^{h,h'}(\xi)} (x) - \int_{k\chi} ^{k\chi +D'} g_{q,\rho}(x) d \Gamma_{\cA_k^{h,h'}(\xi)} (x)  \right|.
\end{equation}
Now, the terms in \eqref{Term 1} and \eqref{Term 3} are bounded by $o_k ^{\frac{1}{4}}$ by point 2 above. \newline
Finally, the term in \eqref{Term 2} is bounded by $o_k ^{-\frac{2}{4}} \cdot o_k$ since by point 1 above there are at most $o_k ^{-\frac{2}{4}}$ steps in $\psi$, and since by Theorem \ref{Theorem equid} each such step introduces an error of at most $o_k$.   
\end{proof}

In the context of Theorem \ref{Theorem equid 2}, it is natural to ask about the existence of integers $k$ that satisfy \eqref{Eq relation between q and k} with respect to some uniform $C>1$. This is the content of the following Lemma:
\begin{Lemma} \label{Lemma choice}
By potentially making $o_k$ go to zero slower, we may assume that there exists a constant $C>1$ such that for every $q$ large there exists  $k=k(q) \in \mathbb{N}$ such that
\begin{equation*} 
|q|= \Theta_{C} \left( o_k ^{-\frac{1}{4}} e^{(k+\sqrt{k}) \chi} \right).
\end{equation*}
\end{Lemma}
\begin{proof}
We first make the following assumptions on the sequence $o_k$:
\begin{enumerate}
\item $o_k \geq \frac{1}{k}$. Otherwise, we move to the sequence $a_k = \max \lbrace o_k, \frac{1}{k} \rbrace$.  Then $o_k \leq a_k$ and still $a_k\rightarrow 0$.

\item It is monotonic decreasing. Otherwise, for every $k$ define $v_k = \sup \lbrace o_n : n\geq k \rbrace$. Then $o_k \leq v_k$, $v_k$ is decreasing,  and it is clear that $v_k \rightarrow 0$.

\item For every $k$ we have $\frac{1}{4}\cdot o_{k} \leq o_{k+1} \leq o_{k}$. Otherwise, we move to the recursively defined sequence $b_k$, where
$$b_1= o_1,\quad b_k = \max \lbrace o_k, \frac{b_{k-1}}{4} \rbrace.$$
Then $b_k \geq o_k \geq \frac{1}{k}$ and 
$$\frac{b_k}{4} \leq b_{k+1} \leq b_k, \quad b_k  \rightarrow 0.$$
\end{enumerate}

Now, let $g:\mathbb{R}_+ \rightarrow \mathbb{R}_+$ be a smooth monotonic decreasing function such that $g(k)=o_k$. Let $q$ be large. Find $x\in \mathbb{R}_+$ such that
$$|q|= g(x)^{-\frac{1}{4}} \cdot e^{(x+ \sqrt{x}) \chi}.$$
Notice that
$$|\frac{1}{4}\log \frac{1}{g([x])}+(x+\sqrt{x})\chi -\frac{1}{4}\log \frac{1}{g(x)} - ([x]+ \sqrt{[x]} )\chi|$$
$$ \leq |\frac{1}{4} \log \frac{g(x)}{g([x])}|  + |(x + \sqrt{x}-[x] -  \sqrt{[x]} )\cdot \chi|$$
$$\leq \frac{1}{4} |\log \frac{g([x]+1)}{g([x])}| + 3\chi$$
$$\leq \frac{1}{4} |\log 4| + 3\chi.$$
It follows that
$$\frac{g(x)^{-\frac{1}{4}} \cdot e^{(x+ \sqrt{x}) \chi}}{g([x])^{-\frac{1}{4}} \cdot e^{([x]+   \sqrt{[x]}  \chi}} = \exp\left( \frac{1}{4}\log \frac{1}{g([x])}+(x+\sqrt{x})\chi -\frac{1}{4}\log \frac{1}{g(x)} - ([x]+ \sqrt{[x]} )\chi \right)=O(1).$$
We thus choose, for every $q$ large, our $k$ as $[x]$. It follows that there is some uniform $C>1$ such that
$$C^{-1} \leq \frac{|q|}{o_k ^{-\frac{1}{4}} e^{(k+\sqrt{k}) \chi}} \leq C$$
which is what we claimed.
\end{proof}

\subsection{Proof of Theorem \ref{Theorem main tech} part (1)}

 Let $\nu$ be as in Theorem \ref{Theorem main tech}. Our goal is to show that 
$$\lim_{|q|\rightarrow \infty} \mathcal{F}_q (\nu) =0.$$
So, let $\epsilon>0$,  let $|q|$ be large, and choose $k=k(q)\in \mathbb{N}$ as in Lemma \ref{Lemma choice}. Recall that this means that for some $C>1$,
$$q= \Theta_C \left( o_k ^{-\frac{1}{4}} \cdot e^{(k+h')\chi} \right)$$
where our standing assumption is that 
$$h'=\sqrt{k}.$$
By Lemma \ref{Lemma choice} (and its proof),  any requirement that $k$ be large translates to a requirement on $q$ being large. In the notation of Theorem \ref{Theorem equid}, we let $k$ be large, and fix $h_0=0$. Recall that $h_0=0$ means that $h=0$ and so $A_\eta = \mathcal{A}^\mathbb{N}$.  We also define an auxiliary stopping time $\tilde{\beta}_k : \mathcal{A}^\mathbb{N} \rightarrow \mathbb{N}$ by
\begin{equation} \label{Eq for tilde beta}
\tilde{\beta}_k (\omega) = \min \lbrace m: |f_{\omega|_m}'(x_0)|<e^{-(k+h')\chi} \rbrace
\end{equation}
where we recall that $x_0\in I$ is our prefixed point as in Section \ref{Section sketch}.
\begin{Lemma} \label{Claim 3'} 
For every $k\in \mathbb{N}$,
$$\nu = \mathbb{E}(f_{\omega|_{\tilde{\beta}_k (\omega)}} \nu).   $$ 
\end{Lemma}
\begin{proof}
This is standard, and follows since $\nu$ is self-conformal. See e.g.  \cite[Lemma 2.2.4]{bishop2013fractal}.
\end{proof}

So,  by Lemma \ref{Claim 3'} and  Jensen's inequality we obtain
\begin{eqnarray*}
|\mathcal{F}_q (\nu)|^2 &=& \left|\mathcal{F}_q (\mathbb{E}(f_{\omega|_{\tilde{\beta}_k (\omega)}} \nu)) \right|^2 \\
&=& \left|\mathbb{E} \left( \mathcal{F}_q \left( f_{\omega|_{\tilde{\beta}_k (\omega)}} \nu \right) \right) \right|^2 \\
&\leq & \mathbb{E} \left( \left| \mathcal{F}_q \left( f_{\omega|_{\tilde{\beta}_k (\omega)}} \nu \right) \right|^2  \right). \\
\end{eqnarray*}
Next, appealing to Theorem \ref{Theorem equid} with our choice of $k,h_0,h'$,  there is a subset that we denote by $\tilde{A_\eta} \subseteq \mathcal{A}^\mathbb{N}$ with $\mathbb{P} (\tilde{A_\eta}) \geq 1-o_k(1)$ such that

$$ \mathbb{E} \left( \left| \mathcal{F}_q \left( f_{\omega|_{\tilde{\beta}_k (\omega)}} \nu \right) \right|^2 \right) =  $$
$$ = \int_{\xi \in \mathcal{A}^\mathbb{N}\setminus \tilde{A_\eta}}  \mathbb{E}_{\cA_k^{h,h'}(\xi)}  \left| \mathcal{F}_q \left( f_{\omega|_{\tilde{\beta}_k (\omega)}} \nu \right) \right|^2 d \mathbb{P}( \xi) +  \int_{\xi \in \tilde{A_\eta}}  \mathbb{E}_{\cA_k^{h,h'}(\xi)}  \left| \mathcal{F}_q \left( f_{\omega|_{\tilde{\beta}_k (\omega)}} \nu \right) \right|^2 d \mathbb{P}( \xi).$$
Combining this with the previous equation array, and using that $\left| \mathcal{F}_q \left( f_{\omega|_{\tilde{\beta}_k (\omega)}} \nu \right) \right|^2\leq 1$ uniformly in all parameters, we conclude that
\begin{equation} \label{Eq first bound}
|\mathcal{F}_q (\nu)|^2 \leq  \int_{\xi \in \tilde{A_\eta}}  \mathbb{E}_{\cA_k^{h,h'}(\xi)}  \left| \mathcal{F}_q \left( f_{\omega|_{\tilde{\beta}_k (\omega)}} \nu \right) \right|^2 d \mathbb{P}( \xi) + o_k(1).
\end{equation}
Next, we take a closer look at the maps $f_{\omega|_{\tilde{\beta}_k (\omega)}}$:

\begin{Lemma} \label{Lemma big P}
There exists some integer $P>1$ such that:

For all $k$ large enough, and for every $\omega$, letting $\eta'$ be such that $A_{k,\eta,\eta'} = \cA_k^{h,h'}(\omega) $, we have
$$\left|\tilde{\beta}_k(\omega)-\tau_k(\omega)- |\eta'| \right| \leq P.$$
\end{Lemma}
\begin{proof}
We first observe that, by the definition of $\tilde{\beta}_k$ from \eqref{Eq for tilde beta},  
\begin{equation*} 
f_{\omega|_{\tilde{\beta}_k (\omega)}} = f_{\omega|_{\tau_k (\omega)}}\circ f_{\omega|_{\tau_k(\omega)} ^{\tilde{\beta}_k(\omega)}},\quad \text{ and } f_{\omega|_{\tau_k(\omega)+|\eta'|}} = f_{\omega|_{\tau_k (\omega)}}\circ f_{\eta'} \text{ since } \omega\in A_{k,\eta,\eta'}.  
\end{equation*}
So, either $\omega|_{\tau_k(\omega)} ^{\tilde{\beta}_k(\omega)}$ is a prefix of $\eta'$, or vice versa. By the last displayed equation, for any $x\in I$,
$$\frac{ \left| f_{\omega|_{\tilde{\beta}_k (\omega)}} ' (x) \right| }{\left| f_{\omega|_{\tau_k (\omega)}}' \left( f_{\omega|_{\tau_k(\omega)} ^{\tilde{\beta}_k(\omega)}} (x) \right) \right|} = |f_{\omega|_{\tau_k(\omega)} ^{\tilde{\beta}_k(\omega)}}'(x)|.$$
Now, it is a consequence of Theorem \ref{Bounded distortiion theorem} (bounded distortion), the definition of $\tilde{\beta}_k$, and of the definition  of $\tau_k$ (Section \ref{Section LLT}), that for some $L>1$ and  all $y\in I$,
\begin{equation} \label{Use of bounded dist}
L^{-1} \cdot e^{-(k+h')\chi-D'} \leq  \left| f_{\omega|_{\tilde{\beta}_k (\omega)}} ' (y) \right| \leq L \cdot e^{-(k+h')\chi},\quad L^{-1}\cdot e^{-k\chi-D'} \leq  \left| f_{\omega|_{\tau_k (\omega)}} ' (y) \right| \leq L \cdot e^{-k\chi}.
\end{equation}
combining the last two displayed equations, we see that there some constant $C'>1$ such that
\begin{equation} \label{Eq der of rest of digits}
|f_{\omega|_{\tau_k(\omega)} ^{\tilde{\beta}_k(\omega)}}'(x)| = \Theta_{C'} \left( e^{-h'\chi} \right), \quad \forall x\in I.
\end{equation}

On the other hand, by the definition of the event $ A_{k,\eta,\eta'}$ and by Theorem \ref{Bounded distortiion theorem}
\begin{equation} \label{Eq der of eta prime}
|f_{\eta'} (x)| = \Theta_{L} \left( e^{-h'\chi} \right), \quad \forall x\in I.
\end{equation}

Therefore, combining equation \eqref{Eq der of eta prime} with \eqref{Eq der of rest of digits} (and noting that the constants $C',L$ are uniform), that  either $\omega|_{\tau_k(\omega)} ^{\tilde{\beta}_k(\omega)}$ is a prefix of $\eta'$ or vice versa, and equation \eqref{Eq C and C prime}, the Lemma follows.
\end{proof}

Let $P$ be as in Lemma \ref{Lemma big P}. For every word $\eta'\in \lbrace1,...,n\rbrace^*$ of length $|\eta'|>P$ we define
$$\bar{\eta'}:=\eta'|_{|\eta'|-P}$$
That is, $\bar{\eta'}$ is the prefix of $\eta'$ of length $|\eta'|-P$. It is now a corollary of Lemma \ref{Lemma big P} that for any $\omega$, if $\eta'=\eta'(\omega)$ is as in Lemma \ref{Lemma big P}, then there is a   word $\rho_{\omega,k}$ such that
$$ f_{\omega|_{\tilde{\beta}_k (\omega)}} = f_{\omega|_{\tau_k (\omega)}}\circ f_{\bar{\eta'}} \circ f_{\rho_{\omega,k}}$$
and $|\rho_{\omega,k}|\leq 2P$.

With this information, we revisit equation \eqref{Eq first bound}. Recall that $M_s(t)=s\cdot t$.

\begin{Claim} \label{Claim linear} Fix $\beta\in(0,\gamma)$. Then for all $k$ large enough,
\begin{eqnarray*}
|\mathcal{F}_q (\nu)|^2 &\leq& \int_{\xi \in \tilde{A_\eta}}  \mathbb{E}_{\cA_k^{h,h'}(\xi)}  \left| \mathcal{F}_q \left(   M_{e^{-S_{\tau_k(\omega)} (\omega)}} \circ M_{ \sign\left( f'_{\omega|_{\tau_k(x)}} (x_{\sigma^{\tau_k(\omega)}(\omega)})\right)}  \circ f_{\bar{\eta'}} \circ f_{\rho_{\omega,k}}  \nu \right) \right|^2 \, d \mathbb{P}( \xi)\\
&+& O(q\cdot e^{-(k+h') \chi-\beta\cdot h' \chi})+o_k(1)
\end{eqnarray*}
where for every $\xi \in \tilde{A_\eta}$, recalling that ${\cA_k^{h,h'}(\xi)}= A_{k,\eta,\eta'}$, $\eta'$ is defined as  $\eta'=\eta'(\xi)$. 
\end{Claim}
\begin{proof}
Fix $\xi \in \tilde{A_\eta}$. Assuming ${\cA_k^{h,h'}(\xi)} = A_{k,\eta,\eta'}$, let $\eta'$ be this  $\eta'=\eta'(\xi)$. Assume $\omega\in A_{k,\eta,\eta'}$. Then we have seen that there is a word $\rho_{\omega,k}$ such that
\begin{equation}  \label{Eq unwinding of beta tilde}
f_{\omega|_{\tilde{\beta}_k (\omega)}} = f_{\omega|_{\tau_k (\omega)}}\circ f_{\bar{\eta'}} \circ f_{\rho_{\omega,k}}
\end{equation}
and $|\rho_{\omega,k}|\leq 2P$. It follows from the proof of Lemma \ref{Lemma big P} (specifically, equation \eqref{Eq der of rest of digits}) that there some constant $C'>1$ such that
\begin{equation} \label{Eq der of eta bar}
|\left( f_{\bar{\eta'}} \circ f_{\rho_{\omega,k}} \right) ' (x)| = \Theta_{C'} \left( e^{-h'\chi} \right), \quad \forall x\in I.
\end{equation}
Also, it is a consequence of \eqref{Eq unwinding of beta tilde} that there exists some $z\in I$ with $f_{\bar{\eta'}} \circ f_{\rho_{\omega,k}}(z)=  x_{\sigma^{\tau_k(\omega)}(\omega)}$.

Now,  plug into the $C^0$ linearization Lemma \ref{Lemma lin} the parameters $g = f_{\omega|_{\tau_k (\omega)}}$,  $y=f_{\bar{\eta'}} \circ f_{\rho_{\omega,k}} (z)$ and for $x\in  I$ we plug in $f_{\bar{\eta'}} \circ f_{\rho_{\omega,k}} (x)$. Then, by \eqref{Eq unwinding of beta tilde} and assuming $k$ (and therefore $h'$) are large enough,
$$\left| f_{\omega|_{\tilde{\beta}_k (\omega)}} (x) - f_{\omega|_{\tau_k(\omega)}} \left( f_{\bar{\eta'}} \circ f_{\rho_{\omega,k}}(z) \right) - f_{\omega|_{\tau_k(\omega)}} ' \left( f_{\bar{\eta'}} \circ f_{\rho_{\omega,k}} (z)\right) \left( f_{\bar{\eta'}} \circ f_{\rho_{\omega,k}}(x)- f_{\bar{\eta'}} \circ f_{\rho_{\omega,k}}(z) \right) \right| $$
$$ \leq |f_{\omega|_{\tau_k(\omega)}} ' (f_{\bar{\eta'}} \circ f_{\rho_{\omega,k}}(z))| \cdot  \left| f_{\bar{\eta'}} \circ f_{\rho_{\omega,k}}(x)- f_{\bar{\eta'}} \circ f_{\rho_{\omega,k}}(z) \right|^{1+\beta}. $$
And, by the definition of  $\tau_k$, since $f_{\bar{\eta'}} \circ f_{\rho_{\omega,k}}(z)=  x_{\sigma^{\tau_k(\omega)}(\omega)}$, and by  \eqref{Eq der of eta bar}
$$|f_{\omega|_{\tau_k(\omega)}} ' \left( f_{\bar{\eta'}} \circ f_{\rho_{\omega,k}}(z) \right) | \leq e^{-k \chi},\quad \left| f_{\bar{\eta'}} \circ f_{\rho_{\omega,k}}(x)- f_{\bar{\eta'}} \circ f_{\rho_{\omega,k}}(z) \right| \leq  O(e^{-h' \chi}).$$

Now, for every $\omega\in A_{k,\eta,\eta'}$  we define a smooth map $S_{\omega,k,\eta'}:I\rightarrow \mathbb{R}$ via
\begin{eqnarray} \label{Eq for linear}
S_{\omega,k,\eta'} (x)&=&  \left| f'_{\omega|_{\tau_k(\omega)}} (x_{\sigma^{\tau_k(\omega)}(\omega)}) \right| \cdot \sign\left( f'_{\omega|_{\tau_k(\omega)}} (x_{\sigma^{\tau_k(\omega)}(\omega)})\right)\cdot f_{\bar{\eta'}} \circ f_{\rho_{\omega,k}} (x)\\
&& -f'_{\omega|_{\tau_k(\omega)}} (x_{\sigma^{\tau_k(\omega)}(\omega)})\cdot x_{\sigma^{\tau_k(\omega)}(\omega)}+ f'_{\omega|_{\tau_k(\omega)}} (x_{\sigma^{\tau_k(\omega)}(\omega)}). \nonumber
\end{eqnarray}
This map is affine in  $\sign\left( f'_{\omega|_{\tau_k(\omega)}} (x_{\sigma^{\tau_k(\omega)}(\omega)})\right) \cdot f_{\bar{\eta'}} \circ f_{\rho_{\omega,k}} (x)$. Then we have just shown that
$$||f_{\omega|_{\tilde{\beta}_k (\omega)}} - S_{\omega,k,\eta'}||_{C^0 (I)} \leq O(e^{-(k+h') \chi -\beta h' \chi}).$$
So, since $\mathcal{F}_q(\cdot )$ is a $2  \pi q$-Lipschitz function, 
$$\left|\mathcal{F}_q \left( f_{\omega|_{\tilde{\beta}_k (\omega)}} \nu \right) -  \mathcal{F}_q \left( S_{\omega,k,\eta'} \nu \right)\right| \leq O(q\cdot e^{-(k+k') \chi -\beta k' \chi}).$$
Therefore
$$\left| |\mathcal{F}_q \left( f_{\omega|_{\tilde{\beta}_k (\omega)}} \nu \right)|^2 -  |\mathcal{F}_q \left( S_{\omega,k,\eta'} \nu \right)|^2  \right| \leq O(q\cdot e^{-(k+h') \chi -\beta h' \chi}).$$
and so for every event $A_{k,\eta,\eta'}$ we have
\begin{equation} \label{Eq for every event}
\left| \mathbb{E}_{A_{k,\eta,\eta'}} \left| \mathcal{F}_q \left( f_{\omega|_{\tilde{\beta}_k (\omega)}} \nu \right) \right|^2  - \mathbb{E}_{A_{k,\eta,\eta'}}  \left| \mathcal{F}_q \left( S_{\omega,k,\eta'}  \nu \right) \right|^2  \right| \leq O(q\cdot e^{-(k+h') \chi -\beta h' \chi}).
\end{equation}

Finally, recall equation \eqref{Eq first bound}, and recall the definition of the maps $S_{\omega,k,\eta'}$ from \eqref{Eq for linear}. Note that
$$\log |f_{\omega|_{\tau_k(\omega)}}'(x_{\sigma^{\tau_k(\omega)}(\omega)})|=- S_{\tau_k(\omega)} (\omega)$$
by Lemma \ref{Lemma relation between R.V and stopping time}. The Claim follows from \eqref{Eq first bound} and \eqref{Eq for every event}, since the translation of $S_{\omega,k,\eta'}$  does not effect the absolute value of $\mathcal{F}_q(\cdot )$, by integrating over all $\xi  \in \tilde{A_\eta}$ (using that the bounds we got are uniform in $\xi$). 
\end{proof}

\begin{Corollary} \label{Coro finite sum}
There is some $K_1=K_1(\epsilon)$ such that for all $k>K_1$,
\begin{equation} \label{equation before fun}
\left| \mathcal{F}_q (\nu) \right|^2 \leq \sum_{|\rho|\leq 2P} \int_{\xi \in \tilde{A_\eta}}  \mathbb{E}_{\cA_k^{h,h'}(\xi)}  \left| \mathcal{F}_q \left( M_{e^{-S_{\tau_k(\omega)} (\omega)}}\circ  f_{\bar{\eta'}} \circ f_\rho  \nu \right) \right|^2 \, d \mathbb{P}( \xi) + \epsilon
\end{equation}
where $P$ is the constant from Lemma \ref{Lemma big P}. Furthermore, there is some global constant $C'>1$ such that for all $\bar{\eta'}$ and $\rho$ as above
$$|\left( f_{\bar{\eta'}} \circ f_{\rho} \right) ' (x)| = \Theta_{C'} \left( e^{-h'\chi} \right), \quad \forall x\in I.$$
\end{Corollary}
\begin{Remark} \label{Remark}
 For notational convenience, in this Corollary and the subsequent argument, we make the  assumption that we always have $ f'_{\omega|_{\tau_k(x)}} (x_{\sigma^{\tau_k(\omega)}(\omega)})>0$. Otherwise, we simply make the sum on the right hand side of \eqref{equation before fun} larger, by including the possibility that it is negative. Since there are uniformly finitely many such options, still the sum above is over uniformly finitely many terms, and the proof follows through.
\end{Remark}
\begin{proof} 
For every $k$ large enough, for every $\xi \in \tilde{A_\eta}$ and every $\omega \in {\cA_k^{h,h'}(\xi)}$, as $|\rho_{\omega,k}|\leq 2P$ 
$$\left| \mathcal{F}_q \left( M_{e^{-S_{\tau_k(\omega)} (\omega)}}\circ f_{\bar{\eta'}} \circ f_{\rho_{\omega,k}}  \nu \right) \right|^2 \leq \sum_{|\rho|\leq 2P} \left| \mathcal{F}_q \left( M_{e^{-S_{\tau_k(\omega)} (\omega)}}\circ f_{\bar{\eta'}} \circ f_\rho  \nu \right) \right|^2$$
and so, by Claim \ref{Claim linear}, assuming  $f'_{\omega|_{\tau_k(x)}} (x_{\sigma^{\tau_k(\omega)}(\omega)})>0$ is always true,
$$|\mathcal{F}_q (\nu)|^2 \leq \int_{\xi \in \tilde{A_\eta}}  \mathbb{E}_{\cA_k^{h,h'}(\xi)}  \left| \mathcal{F}_q \left( M_{e^{-S_{\tau_k(\omega)} (\omega)}}\circ f_{\bar{\eta'}} \circ f_{\rho_{\omega,k}}  \nu \right) \right|^2 \, d \mathbb{P}( \xi) + O(q\cdot e^{-(k+h') \chi-\beta\cdot h' \chi})+o_k(1)$$
$$\leq \sum_{|\rho|\leq 2P} \int_{\xi \in \tilde{A_\eta}}  \mathbb{E}_{\cA_k^{h,h'}(\xi)}  \left| \mathcal{F}_q \left( M_{e^{-S_{\tau_k(\omega)} (\omega)}}\circ f_{\bar{\eta'}} \circ f_\rho  \nu \right) \right|^2 \, d \mathbb{P}( \xi) + O(q\cdot e^{-(k+h') \chi-\beta\cdot h' \chi})+o_k(1).$$
Recalling the choice of $k=k(q)$ (which is as in Lemma \ref{Lemma choice}),
$$O(q\cdot e^{-(k+h') \chi -\beta h' \chi})= O( o_k^{-\frac{1}{4}} \cdot e^{-\beta h' \chi}).$$
So, since $o_k$ decays in at most a polynomial rate (by e.g. Lemma \ref{Lemma choice}), there is some $K_1=K_1(\epsilon)$ as we claimed. The last assertion is a consequence of equation \eqref{Eq der of eta prime}, Theorem \ref{Bounded distortiion theorem} (bounded distortion), that $|\rho|\leq 2P$, and of equation \eqref{Eq C and C prime}.
\end{proof}

Now, fix some $\rho$ with $|\rho|\leq 2P$ and consider the corresponding term in \eqref{equation before fun}
$$ \int_{\xi \in \tilde{A_\eta}}  \mathbb{E}_{\cA_k^{h,h'}(\xi)}  \left| \mathcal{F}_q \left( M_{e^{-S_{\tau_k(\omega)} (\omega)}}\circ  f_{\bar{\eta'}} \circ f_\rho  \nu \right) \right|^2 \, d \mathbb{P}( \xi).$$
We next appeal to Theorem \ref{Theorem equid 2} for every event ${\cA_k^{h,h'}(\xi)}$   separately. To do this, we notice that by Corollary  \ref{Coro finite sum}, for every $f_{\bar{\eta'}} \circ f_\rho$ involved 
$$\diam \left( \supp \left( f_{\bar{\eta'}} \circ f_\rho    \nu \right) \right) =O(e^{-h' \chi}).$$
Notice that the the error term in Theorem \ref{Theorem equid 2} is $O(o_k^{\frac{1}{4}})$ independently of the event ${\cA_k^{h,h'}(\xi)}$. So, 
$$ \int_{\xi \in \tilde{A_\eta}}  \mathbb{E}_{\cA_k^{h,h'}(\xi)} \left| \mathcal{F}_q \left( M_{e^{-S_{\tau_k(\omega)} (\omega)}}\circ f_{\bar{\eta'}} \circ f_\rho   \nu \right) \right|^2 \, d \mathbb{P}( \xi) $$
$$\leq \int_{\xi \in \tilde{A_\eta}}   \int_{k\chi} ^{k\chi +D'} \left| \mathcal{F}_q \left(M_{e^{-x}} \circ f_{\bar{\eta'}} \circ f_\rho    \nu \right) \right|^2 \, d \Gamma_{\cA_k^{h,h'}(\xi)} (x) \,  d \mathbb{P}( \xi) +O(o_k^{\frac{1}{4}}).$$
Let $K_2=K_2(\epsilon)$ be large enough to ensure that if $k\geq K_2$ then $O(o_k^{\frac{1}{4}}) \leq \frac{\epsilon}{|\lbrace \rho: |\rho|\leq 2P\rbrace| }$.  Then, since this is true for every $\rho$ with $|\rho|\leq 2P$, we see that for $k\geq \max \lbrace K_2,K_1 \rbrace$, putting this into \eqref{equation before fun} we get
$$\left| \mathcal{F}_q (\nu) \right|^2 \leq \sum_{|\rho|\leq 2P} \int_{\xi \in \tilde{A_\eta}}   \int_{k\chi} ^{k\chi +D'} \left| \mathcal{F}_q \left(M_{e^{-x}}\circ f_{\bar{\eta'}} \circ f_\rho   \nu \right) \right|^2 d \Gamma_{\cA_k^{h,h'}(\xi)} (x)   d \mathbb{P}( \xi) + 2\epsilon. $$

Recall that by Lemma \ref{Lemma Abs. contin of gamma}, the probability measure $\Gamma_{\cA_k^{h,h'}(\xi)}$ is absolutely continuous with respect to the Lebesgue measure on $[k \chi, k \chi +D']$, such that  the norm of its density function is  uniformly bounded by $\frac{1}{D}>0$ independently of all parameters. Using this fact,  as long as $k \geq \max \lbrace K_1,K_2\rbrace$,
$$\left| \mathcal{F}_q (\nu) \right|^2 \leq \sum_{|\rho|\leq 2P} \int_{\xi \in \tilde{A_\eta}} \left(   \int_{k\chi} ^{k\chi +D'} \left| \mathcal{F}_q \left(M_{e^{-z}} \circ f_{\bar{\eta'}} \circ f_\rho \nu \right) \right|^2 \cdot \frac{1}{D} dz   \right) d \mathbb{P}( \xi) + 2\epsilon. $$

We now invoke  Lemma \ref{Lemma 3.2 }. Taking the measure(s) to be $f_{\bar{\eta'}} \circ f_\rho \nu $,  for any $r>0$, we get the inequality (as long as $k$ is large enough)
$$\left| \mathcal{F}_q (\nu) \right|^2 \leq \sum_{|\rho|\leq 2P} \int_{\xi \in \tilde{A_\eta}} \frac{D'}{D}  \cdot \left(  \frac{e^2}{r\cdot |q|}+ \int f_{\bar{\eta'}} \circ f_\rho  \nu \left( B_{e^{\chi k}\cdot r} (y) \right) d \left( f_{\bar{\eta'}} \circ f_\rho \nu (y) \right)  \right) d\mathbb{P}(\xi) +2 \epsilon.$$
By Corollary \ref{Coro finite sum}, all the maps $\left( f_{\bar{\eta'}} \circ f_\rho \right)  ^{-1}$ as above are $O(e^{h'\chi})$ Lipschitz (with the implied constant in the $O(\cdot )$ being uniform).  Therefore,  there is some $T>1$ such that for every $\eta'$ and $y\in I$, 
$$\left( f_{\bar{\eta'}} \circ f_\rho \right)  ^{-1} \left( B_{e^{\chi k}\cdot r} ( f_{\bar{\eta'}} \circ f_\rho (y)  \right) \subseteq B_{T\cdot e^{\chi (k+h')}\cdot   r} (y)$$
so, for a fixed $r>0$ we can relax the dependence on both $\rho$ and $\xi$, and get
\begin{equation} \label{Eq a}
|\mathcal{F}_q (\nu)|^2 \leq |\lbrace \rho: |\rho|\leq 2P\rbrace| \cdot \left( \frac{e^2}{r\cdot |q|}+  \int \nu(B_{T\cdot e^{\chi (k+h')}\cdot   r} (y)) d  \nu (y) \right) \cdot    \frac{D'}{D}  + 2\epsilon.
\end{equation}

By Lemma \ref{Lemma nu is continuous} there exists some $\delta=\delta(\epsilon)>0$ such that
\begin{equation} \label{Eq b}
\nu(B_\delta(y))< \frac{\epsilon \cdot D}{2 \cdot D'\cdot |\lbrace \rho: |\rho|\leq 2P\rbrace| },\quad \forall y\in \mathbb{R}.
\end{equation}
Now,  we choose $r$ so that $T\cdot e^{\chi (k+k')}\cdot r =\delta$. This implies that $\nu(B_{T\cdot e^{\chi (k+h')}\cdot   r} (y))\leq \frac{\epsilon \cdot D}{2 \cdot D'\cdot |\lbrace \rho: |\rho|\leq 2P\rbrace| }$ for every $y$. Therefore, $\frac{1}{r} = \frac{T\cdot e^{\chi (k+h')}}{\delta}$. So,  as $|q|= \Theta_C \left( o_k ^{-\frac{1}{4}} \cdot e^{(k+h')\chi} \right)$,
$$\frac{e^2}{r\cdot |q|} = e^2 \cdot T \cdot  \frac{e^{\chi (k+h')}}{\delta \cdot |q|} \leq C\cdot  e^2 \cdot T \cdot o_k ^{\frac{1}{4}} \cdot  \frac{e^{\chi (k+h')} }{\delta\cdot e^{\chi (k+h')}} = o_k ^{\frac{1}{4}} \cdot \frac{e^2 \cdot  T\cdot  C}{\delta}.$$
So, as long as $k\geq K_3=K_3 (\epsilon)$,
\begin{equation} \label{Eq c}
\frac{e^2}{r\cdot |q|} \leq  \frac{\epsilon \cdot D}{2 \cdot D'\cdot |\lbrace \rho: |\rho|\leq 2P\rbrace| }.
\end{equation}
Finally, if $k\geq \max \lbrace K_1,K_2,K_3\rbrace$, plugging \eqref{Eq b} and \eqref{Eq c} into \eqref{Eq a},
$$|\mathcal{F}_q (\nu)|^2 \leq 3 \epsilon$$
which implies the Theorem.

\section{Proof of Theorem \ref{Theorem main tech} part (2)} \label{Section proof 2}
\subsection{Some reductions} \label{Section reductions}
We continue to assume the condition of Theorem \ref{Theorem main tech}, and use the notation introduced in Sections \ref{Section sketch}, \ref{Section pre}, \ref{Section LLT}.  Fix an integer $p\geq 2$. We aim to prove the following Theorem:
\begin{theorem} \label{Main theorem}
Let $\nu$ be a measure as in Theorem \ref{Theorem main tech}. Then for $\nu$ almost every $x$,
$$ \lim_{N} \frac1N\sum_{n=1}^N \delta_{T_p^n (x)} = \lambda_{[0,1]}$$
where $\lambda_{[0,1]}$ is the Lebesgue  measure on $[0,1]$.
\end{theorem}
Since $p$ is arbitrary,  Theorem \ref{Main theorem} implies Theorem \ref{Theorem main tech} part (2). We will need the following definitions: let $\pi: \mathbb{R} \rightarrow \mathbb{T}=\mathbb{R}/ \mathbb{Z}\simeq [0,1)$ be the projection
$$\pi(x)= x \mod 1.$$
Let $\bar{T}_p :\mathbb{T} \rightarrow \mathbb{T}$ be the continuous map
\begin{equation*} 
\bar{T}_p (y) = p\cdot y \mod 1.
\end{equation*}
Notice that for any $x\in \mathbb{R}$ and any $n\in \mathbb{N}$ we have
\begin{equation} \label{Eq relation times p}
\bar{T}_p ^n \circ \pi (x) = \pi \circ T_p ^n (x) =  p^n \cdot x   \mod 1.
\end{equation}

 Our first step is to reduce to the following statement, where we make use of the fact that $\nu$ is the push-forward of $\mathbb{P}$ under $\omega \mapsto x_\omega$.

\begin{theorem} \label{Lemma sufficient} For every $\epsilon>0$ there exists $q^*=q^*(\epsilon) \in \mathbb{N}$ such that for all integers $q$ with $|q|\geq q^*$ and for $\mathbb{P}$ almost every $\omega$, 
$$\limsup_{N} \left| \mathcal{F}_q \left( \frac1N\sum_{n=1}^N\delta_{T_p^n (x_\omega)} \right) \right| <\epsilon.$$
 \end{theorem}

\noindent{\bf Proof that Theorem \ref{Lemma sufficient} implies Theorem \ref{Main theorem}}  Let $\omega$ be a $\mathbb{P}$ typical point, and let $\nu_\infty$ be a weak-* limit of the sequence 
$$\frac1N\sum_{n=1}^N\delta_{T_p^n (x_\omega)}.$$
We will show that $\nu_\infty$ is the Lebesgue measure on $[0,1]$. It suffices to show that $\mathcal{F}_q (\nu_\infty)=0$ for every integer $q\neq 0$. Consider the measure $\pi \nu_\infty$ on $\mathbb{T}$: It is a consequence of \eqref{Eq relation times p} that $\pi \nu_\infty$ arises from the $\bar{T}_p$ orbit of $\pi(x_\omega)$, and so $\pi \nu_\infty$ is $\bar{T}_p$ invariant. Now, let $\epsilon>0$. Assuming Theorem \ref{Lemma sufficient} holds true, let $n\in \mathbb{N}$ be large enough so that $|q\cdot p^n|\geq q^* (\epsilon)$. Then
\begin{eqnarray*}
|\mathcal F_q(\nu_\infty)| &=& |\mathcal F_q(\pi \nu_\infty)| \\
&=& |\mathcal F_q(\bar{T}_p  ^n \pi \nu_\infty)| \\
&=&  |\mathcal F_{qp^n}( \pi \nu_\infty)| \\
&\leq&  \limsup_{N} \left| \mathcal{F}_{qp^n} \left( \pi \left( \frac1N\sum_{k=1}^N\delta_{T_p^k (x_\omega)} \right) \right) \right| \\
& =& \limsup_{N} \left| \mathcal{F}_{qp^n} \left(  \frac1N\sum_{k=1}^N\delta_{T_p^k (x_\omega)} \right) \right|\\
&<& \epsilon
\end{eqnarray*}
where we have made use several times of the fact that $q$ is an integer, that $\pi \nu_\infty$ is $\bar{T}_p$ invariant, and in the last line of our choice of $q$. Since $\epsilon$ was arbitrary we obtain $\mathcal{F}_q (\nu_\infty)=0$, and we are done. \hfill{$\Box$}
\newline

Theorem \ref{Lemma sufficient}, in turn, reduces to the following statement. Recall the definition of the stopping time $\ttau$ and the random variable $\tau$ as in the beginning of Section \ref{Section statement}.

\begin{theorem} \label{Claim sufficient} For all $\epsilon>0$ and $\gamma'\in (0,\gamma)$ there is $q^*(\epsilon,\bfp) \in \mathbb{N}$ such that for every  $q\in \mathbb{Z}$ with $|q|\geq q^*(\epsilon,\bfp)$:

There are values $h(q,\epsilon,\bfp)$, $k(q,\epsilon,\bfp), h'(q,\epsilon,\bfp)>0$ such that, for all  $n\in\mathbb N$ and  every $\omega \in \mathcal{A}^{\mathbb{N}}$,
$$ \mathbb{E}_{\xi \in \mathcal{A}^h (\omega')}\big|\mathcal F_{qs}(f_{\xi|_{ \tau_k (\xi)+\ttau_{h'} ( \sigma^{\tau_k (\xi)} (\xi))}} \circ f_{\rho_\xi} \nu)\big|^2 + |q|e^{(1+\gamma')h\chi}<\epsilon.$$
Where the partition $\mathcal{A}^h$ is as in Definition  \ref{Def R.V. D} part (2), and:
\begin{enumerate}
\item $\omega ' =  \omega'(n,\omega) = \sigma^{\ttau_{\frac{n\log p}{\chi}}(\omega)} (\omega)$.

\item $s=s(\omega,n,h)= p^n \cdot f_{\omega|_{\ttau_{\frac{n\log p}{\chi}}(\omega)}} ' \left( f_{\omega' |_{\ttau_ h (\omega')}} (x_0) \right)$, where $x_0\in I$ is our prefixed point. In particular, $s,s^{-1}=O(1)$   where this $O(1)$ depends only on the IFS.

\item The word $\rho_\xi$ for $\xi \in \mathcal{A}^\mathbb{N}$ is defined as the unique word satisfying
$$\xi|_{\tau_k (\xi)+\ttau_{h'} ( \sigma^{\tau_k (\xi)} (\xi))}*\rho_\xi= \xi|_{\ttau_h(\xi) + \ttau_{k+h'-h+Q } \left( \sigma^{ {\ttau_h(\xi)} } \xi\right)}$$
for some global $Q>0$ that only depends on the IFS. Furthermore, there exists some global $P>0$ such that $|\rho_\xi| \leq P$ for all $\xi$.
\end{enumerate}
\end{theorem}  

We remark that the  word $\rho_\xi$ also depends on the parameters $k,h,h',Q$, but we suppress this in our notation. Both the proof of Theorem \ref{Claim sufficient} and the proof that it implies Theorem \ref{Lemma sufficient} are not trivial. Thus, we dedicate the next Section to the proof that Theorem \ref{Claim sufficient} implies Theorem \ref{Lemma sufficient}. The subsequent Section contains the proof of Theorem \ref{Claim sufficient}. So, all in all, once these two assertions are established, Theorem \ref{Main theorem} is proved and we are done.

\subsection{Proof that Theorem \ref{Claim sufficient} implies Theorem \ref{Lemma sufficient}} 
\subsubsection{The martingale argument}
In this Section we employ a deep observation that was  originally made by Hochman and Shmerkin \cite[Theorem 2.1]{hochmanshmerkin2015}, and was recently further refined by Hochman \cite{Hochman2020Host}. Recall that if $\mathcal{C}$ is a partition of a space $X$ and $\mu$ is a probability measure on $X$, then for any $x \in \supp(\mu)$ we denote by $\mathcal{C}(x)$ the atom of $\mathcal{C}$ containing $x$, and by $\mu_{\mathcal{C}(x)}$ the conditional measure of $\mu$ on this atom. 

\begin{theorem} \cite[Theorem 2.2]{Hochman2020Host} \label{Theorem 2.2}
Let $T:X\rightarrow X$ be a continuous  map of a compact metric space, and let $\mu$ be a Borel probability measure on $X$. Let $\lbrace \mathcal{C}_n \rbrace_{n\in \mathbb{N}}$ be a refining sequence of Borel partitions. Suppose that
\begin{equation} \label{Eq diameter condition}
\lim_{k\rightarrow \infty} \, \sup_{n\in \mathbb{N}} \lbrace \diam (T^n A):\, A\in \mathcal{C}_{n+k},\mu(A)>0 \rbrace = 0.
\end{equation}
Then for $\mu$ almost every $x$,
$$\lim_{N\rightarrow \infty} \left( \frac{1}{N} \sum_{n=1} ^N \delta_{T^n (x)} - \frac{1}{N} \sum_{n=1} ^N T^n \mu_{\mathcal{C}_n (x)} \right)=0$$
in the weak-* sense. 
\end{theorem}
That is, as long as the partitions $\lbrace \mathcal{C}_n \rbrace_{n\in \mathbb{N}}$ are compatible with the dynamics of $T$ in the sense of \eqref{Eq diameter condition},  the orbits of $\mu$ typical points are Ces\`aro equivalent to the $T^n$-magnifications of the conditionals of $\mu$ on their $\mathcal{A}^n$ atoms.

Now, let $0<h \ll 1$ be a fixed parameter, and consider the stopping time $\beta_{n,h}$ defined by
$$\beta_{n,h} (\omega) = \ttau_{\frac{n\log p}{\chi}}(\omega) + \ttau_ h(\sigma^{\ttau_{\frac{n\log p}{\chi}}(\omega)}\omega).$$
Note that in the  stopping time  $\beta_{n,h}$ we let $n$ vary but keep $h$ fixed. 

\begin{theorem}  \label{Theorem martingale}
For $\mathbb{P}$ almost every $\omega$ and for every integer $q\in \mathbb{Z}$,
\begin{equation*}
\lim_{N\rightarrow\infty} \mathcal{F}_q \left( \frac{1}{N}  \sum_{n=0} ^{N-1} \delta_{T_p ^n (x_\omega)} \right) - \mathcal{F}_q \left( \frac{1}{N} \sum_{n=0} ^{N-1} T_p ^n \circ f_{\omega|_{\beta_{n,h} (\omega)}} \nu \right)  =0.
\end{equation*}
\end{theorem}
\begin{proof}
Equip the compact space $X = \mathcal{A}^\mathbb{N} \times \mathbb{T}$ with the usual metric on each coordinate and the $\sup$ metric on the product space. For every $n$ let $\mathcal{C}_{n,h}$ be the partition of $X$ given by
$$(\omega, \pi(x_\omega)) \sim_{\mathcal{C}_{n,h}} (\eta, \pi(x_\eta))  \iff (\omega_1,...,\omega_{\beta_{n,h} (\omega)})  = (\eta_1,...,\eta_{\beta_{n,h} (\eta)}),$$
$$ \text{ and if } x\neq \pi(x_\omega), y \neq \pi(x_\eta) \text{ then }  (\omega, x) \sim_{\mathcal{C}_{n,h}} (\eta,y).$$
Notice that we are grouping all the elements of $X$ that are not of the form $(\omega,\pi(x_\omega))$ into a single partition cell, which we denote by $B$.

Let $\mu$ be the probability measure on $X$ defined by
$$\mu(A)= \mathbb{P} \left( \lbrace \omega:\quad (\omega,\pi(x_\omega))\in A \rbrace \right).$$
Notice that the projection of $\mu$ to $\mathbb{T}$ is $\pi \nu$, and that $\mu(B)=0$.
To complete the setup,  let $T:X\rightarrow X$ be the continuous map
$$T(\omega,x) = (\omega, \bar{T}_p (x)).$$

We now verify that \eqref{Eq diameter condition} holds true: Let $n,k\in \mathbb{N}$ and let $A \in \mathcal{C}_{n+k,h}$ be such that $\mu(A)>0$. Then for any $\omega$ with $(\omega,x_\omega)\in A$, recalling the metric on $\mathcal{A}^\mathbb{N}$ defined in Section \ref{Section sketch},
\begin{eqnarray*}
\diam\left( T^n A \right) &=& \max \lbrace \diam \left( A_{\omega|_{\beta_{n+k,h}(\omega)}} \right) , \diam \left( T_p ^n \circ \pi \circ f_{\omega|_{\beta_{n+k,h}(\omega)}} (K) \right) \rbrace\\
&\leq & \max \lbrace  \rho^{ \ttau_{\frac{k\log p}{\chi}}(\omega)}, p^{-k} \rbrace \rightarrow 0 \text{ as } k\rightarrow \infty \text{ uniformly in } n \text{ and } \omega.
\end{eqnarray*}
Thus,  we may apply Theorem \ref{Theorem 2.2}: For $\mu$ almost every $(\omega, \pi(x_\omega))$ letting $\mu_{\mathcal{C}_{k,h} (\omega,\pi(x_\omega))}$ be the conditional measure of $\mu$ on the atom $\mathcal{C}_{k,h} (\omega,x_\omega)$, we get
\begin{equation} \label{Eq before proj}
\lim_{N\rightarrow\infty}  \left( \frac{1}{N}\sum_{k=0} ^{N-1} \delta_{T ^k (\omega,\pi(x_\omega))} -\frac{1}{N} \sum_{k=0} ^{N-1} T ^k \mu_{\mathcal{C}_{k,h} (\omega,\pi(x_\omega))} \right) =0.
\end{equation}
Notice that the projection of $\mu_{\mathcal{C}_{k,h} (\omega,\pi(x_\omega))}$ to $\mathbb{T}$ is $\pi\circ f_{\omega|_{\beta_{k,h} (\omega)}} \nu$. So, projecting equation \eqref{Eq before proj} to $\mathbb{T}$ and using this observation, for $\mathbb{P}$ almost every $\omega$
\begin{equation*} 
\lim_{N\rightarrow\infty}  \left( \frac{1}{N}\sum_{k=0} ^{N-1} \delta_{\bar{T}_p ^k \circ \pi(x_\omega)} -\frac{1}{N} \sum_{k=0} ^{N-1} \bar{T}_p ^k \circ \pi\circ f_{\omega|_{\beta_{k,h} (\omega)}} \nu \right) =0.
\end{equation*}

Finally, let $q\in \mathbb{Z}$. Invoking \eqref{Eq relation times p}, for every $k\geq 1$,
$$\mathcal{F}_q \left(\delta_{\bar{T}_p ^k \circ \pi(x_\omega)} \right) = \mathcal{F}_q \left( \delta_{T_p ^k (x_\omega)} \right), \, \text{ and } \mathcal{F}_q \left( \bar{T}_p ^k \circ \pi\circ f_{\omega|_{\beta_{k,h} (\omega)}} \nu \right) = \mathcal{F}_q \left( T_p ^k \circ f_{\omega|_{\beta_{k,h} (\omega)}} \nu \right).$$
So, combining the last two displayed  equations, the Theorem is proved.
\end{proof}

\begin{Remark} \label{Rmk refree}
As pointed out to us by the anonymous referee, once Theorem \ref{Theorem martingale} is established there is another way to prove Theorem \ref{Theorem main tech} part (2): Via a  slight modification of Theorem \ref{Theorem martingale}, it is enough to show that for $\nu$ as in Theorem \ref{Theorem main tech},
$$\text{ For every } C,C_0>0,\quad \lim_q \sup \lbrace |\mathcal{F}_q (g\nu)|: {g\in C^{1+\gamma}:\, ||g||_{C^{1+\gamma}} <C_0,\, \inf |g'|\geq C_1} \rbrace =0.$$
That is, the rate of decay of $|F_q (g\nu)|$ is uniform in $g\in C^{1+\gamma}$, as long as its $C^{1+\gamma}$ norm and $\inf |g'|$  are uniformly bounded.  We believe this Claim to be true, and it should follow by verifying that:
\begin{enumerate}
\item The rates in Theorems \ref{B-Q CLT} and \ref{B-Q LLT general}  are uniform in $g$. 

\item The other estimates as in Section \ref{Section proof 1} are uniform in $g$.
\end{enumerate}
Taking this approach allows one to circumvent the use of Theorem \ref{Theorem lin equid}, thus possibly shortening the proof. However, since we hope to have other applications for our method that do make use of Theorem \ref{Theorem lin equid} (e.g. for higher dimensions or dimension theory), and since it is one of the goals of this paper to show how local limit Theorems may be adapted to study the geometry of self conformal measures, we present the proof of Theorem \ref{Theorem main tech}  in its original form.
\end{Remark}

\subsubsection{First linearization and stopping time argument}
Our next step is to linearize the maps appearing in Theorem \ref{Theorem martingale}, in the following sense: Fix an integer frequency $q \neq 0\in \mathbb{Z}$. We want to estimate $\mathcal{F}_q (\cdot )$ for the push-forward of $\nu$ via the map
$$T_p ^n \circ f_{\omega|_{\beta_{n,h} (\omega)}} = T_p ^n \circ f_{\omega|_{\ttau_{\frac{n\log p}{\chi}}(\omega) + \ttau_ h(\sigma^{\ttau_{\frac{n\log p}{\chi}}(\omega)}\omega)}}.$$
The idea is use the first $\ttau_{\frac{n\log p}{\chi}}(\omega)$ digits to cancel out the $T_p ^n$ factor. The price is a uniformly bounded defect $s$ in the frequency, and a controllable error term that relies on $h$ and $q$:
\begin{Claim} (First linearization) \label{First linearization}
For every $\omega \in \mathcal{A}^\mathbb{N}$, $q\neq 0 \in \mathbb{Z}$, $n \in \mathbb{N}$, and $h>0$ that is large enough in manner dependent only on $\gamma' \in (0,\gamma)$
$$\left| \left|\mathcal{F}_q \left( T_p ^n \circ f_{\omega|_{\beta_{n,h} (\omega)}} \nu \right) \right| - \left|\mathcal{F}_{qs} \left( f_{\omega' |_{\ttau_ h (\omega')}} \nu \right) \right|  \right| \leq  |q|e^{-(1+\gamma')h\chi} $$
where we recall that $\omega ' = \sigma^{\ttau_{\frac{n\log p}{\chi}}(\omega)} (\omega)$ and for our prefixed $x_0 \in I$ $$s=s(\omega,n,h)= p^n \cdot f_{\omega|_{\ttau_{\frac{n\log p}{\chi}}(\omega)}} ' \left( f_{\omega' |_{\ttau_ h (\omega')}} (x_0) \right) =O(1),\quad \text{ and also } s^{-1} = O(1)$$ 
where this $O(1)$ only depends on the IFS.
\end{Claim}
\begin{proof}
We use the notation $\omega'$ as in the statement of the Claim.  Plugging in $g=f_{\omega|_{\ttau_{\frac{n\log p}{\chi}}(\omega)}}$ into Lemma \ref{Lemma lin}, as long as $h=h(\gamma')$ is large enough, for any $x$ and our prefixed $x_0\in I$
\begin{eqnarray} \label{substitue}
|f_{\omega|_{\ttau_{\frac{n\log p}{\chi}}(\omega)}} \left( f_{\omega' |_{\ttau_ h (\omega')}} (x) \right) &-& f_{\omega|_{\ttau_{\frac{n\log p}{\chi}}(\omega)}} \left( f_{\omega' |_{\ttau_ h (\omega')}} (x_0) \right) \\
&-& f_{\omega|_{\ttau_{\frac{n\log p}{\chi}}(\omega)}} ' \left( f_{\omega' |_{\ttau_ h (\omega')}} (x_0) \right) \cdot \left( f_{\omega' |_{\ttau_ h (\omega')}} (x) - f_{\omega' |_{\ttau_ h (\omega')}} (x_0) \right) | \nonumber \\
&\leq& \left| f_{\omega|_{\ttau_{\frac{n\log p}{\chi}}(\omega)}} ' \left( f_{\omega' |_{\ttau_ h (\omega')}} (x_0) \right) \right| \cdot  \left| f_{\omega' |_{\ttau_ h (\omega')}} (x) - f_{\omega' |_{\ttau_ h (\omega')}} (x_0) \right|^{1+\gamma'} \nonumber \\
&\leq& e^{ - \frac{n\log p}{\chi} \cdot \chi} \cdot e^{- h\chi(1+\gamma')} = p^{-n} \cdot e^{- h\chi(1+\gamma')}. \nonumber
\end{eqnarray}
Denote $t_0 = f_{\omega|_{\ttau_{\frac{n\log p}{\chi}}(\omega)}} ' \left( f_{\omega' |_{\ttau_ h (\omega')}} (x_0) \right)$ and $t_1 = (1-t_0) \cdot f_{\omega' |_{\ttau_ h (\omega')}} (x_0)$.  Using that $\mathcal{F}_q$  is $|q|$-Lipschitz and \eqref{substitue},
$$\left| \left|\mathcal{F}_q \left( T_p^n \circ f_{\omega|_{\beta_{n,h} (\omega)}} \nu \right) \right| -  \left|\mathcal{F}_{q } \left( T_p^n \left( t_0\cdot f_{\omega' |_{\ttau_ h (\omega')}} \nu \right) \right) \right|  \right|$$
$$= \left| \left|\mathcal{F}_q \left( T_p^n \circ f_{\omega|_{\ttau_{\frac{n\log p}{\chi}}(\omega)+\ttau_ h(\sigma^{\ttau_{\frac{n\log p}{\chi}}(\omega)}\omega)}} \nu \right) \right| -  \left|\mathcal{F}_{q } \left( T_p^n \circ \left( t_0\cdot f_{\omega' |_{\ttau_ h (\omega')}}+t_1 \right) \nu  \right) \right|  \right|$$
$$ \leq \left| \mathcal{F}_{qp^n } \left(  f_{\omega|_{\ttau_{\frac{n\log p}{\chi}}(\omega)+\ttau_ h(\sigma^{\ttau_{\frac{n\log p}{\chi}}(\omega)}\omega)}} \nu \right)  -  \mathcal{F}_{qp^n} \left(  \left( t_0 \cdot f_{\omega' |_{\tau_ h (\omega')}} +t_1 \right) \nu \right)  \right|$$
$$\leq |q|p^n \cdot ||f_{\omega|_{\ttau_{\frac{n\log p}{\chi}}(\omega)+\ttau_ h(\sigma^{\ttau_{\frac{n\log p}{\chi}}(\omega)}\omega)}} - \left( t_0 \cdot f_{\omega' |_{\tau_ h (\omega')}} +t_1 \right)||_{\infty}  $$
$$\leq |q|p^n \cdot  p^{-n} \cdot e^{- h\chi(1+\gamma')} = |q|e^{-(1+\gamma')h\chi}.$$

Finally, using bounded distortion (Theorem \ref{Bounded distortiion theorem}), set $s=p^n t_0$ and note that
$$|s| = \left|  p^n \cdot f_{\omega|_{\tau_{\frac{n\log p}{\chi}}(\omega)}} ' \left( f_{\omega' |_{\tau_ h (\omega')}} (x_0) \right) \right| \in [C_0,1],\quad \text{ where } 0<C_0<1 \text{ is a global constant.}$$
Then the result follows since, as $q$ is an integer,
$$\left|\mathcal{F}_{q } \left( T_p^n \left( t_0\cdot f_{\omega' |_{\ttau_ h (\omega')}} \nu \right) \right) \right| = \left|\mathcal{F}_{qp^n t_0 } \left(  f_{\omega' |_{\ttau_ h (\omega')}} \nu \right) \right| =\left|\mathcal{F}_{qs} \left(  f_{\omega' |_{\tau_ h (\omega')}} \nu \right) \right|.$$
\end{proof}
The next Claim, which is the final ingredient in the proof that Theorem \ref{Claim sufficient} implies Theorem \ref{Lemma sufficient}, is about writing measures of the form $f_{\omega|_{\ttau_h(\omega)}} \nu$ as a certain average of measures of the form 
$$f_{\xi|_{ \tau_k (\xi)+\ttau_{h'} ( \sigma^{\tau_k (\xi)} (\xi))}} \circ f_{\rho_\xi} \nu, \quad  \text{ where } \xi \in \mathcal{A}^h (\omega),\, \text{ and } k+h'>h$$
where $\rho_\xi$ is a  word of uniformly bounded length. This is crucially important for our argument, since the local limit Theorem \ref{Theorem lin equid}  applies for the random variable $\tau_k$, but not necessarily for the stopping time $\ttau_k$. 

\begin{Claim} (Relating stopping time with cocycle) \label{Claim stopping}
There is some $P>0$ such that: \newline
 For every $\omega \in \lbrace 1,..., n \rbrace^\mathbb{N}$, $h,h',k >0$  with $h<k+h'$
$$f_{\omega|_{\ttau_h(\omega)}} \nu = \mathbb{E}_{\xi \in \mathcal{A}^h (\omega)} \left( f_{\xi|_{ \tau_k (\xi)+\ttau_{h'} ( \sigma^{\tau_k (\xi)} (\xi))}} \circ f_{\rho_\xi} \nu \right)$$ 
where $\rho_\xi = \rho_{\xi, k,h,h',Q}$ is  the unique random word satisfying
$$\xi|_{\tau_k (\xi)+\ttau_{h'} ( \sigma^{\tau_k (\xi)} (\xi))}*\rho_\xi= \xi|_{\ttau_h(\xi) + \ttau_{k+h'-h+Q } \left( \sigma^{ {\ttau_h(\xi)} } \xi\right)}$$
for some global $Q>0$ that only depends on the IFS, so  that $|\rho_\xi| \leq P$ for all $\xi \in \mathcal{A}^\mathbb{N}$.
\end{Claim}
\begin{proof}
Since $\nu$ is self conformal and $\ttau$ is a stopping time, for any fixed $Q>0$ (to be chosen later)
$$\nu = \mathbb{E}( f_{\eta|_{\ttau_{k+h'-h+Q}}} \nu)$$
so
\begin{equation} \label{Use of stopping time}
f_{\omega|_{\ttau_h(\omega)}} \nu  = \mathbb{E}_{\eta} ( f_{\omega|_{\ttau_h(\omega)}} \circ f_{\eta|_{\ttau_{k+h'-h+Q}}} \nu) = \mathbb{E}_{\xi \in \mathcal{A}^h (\omega)} ( f_{\xi|_{\ttau_h(\xi)} + \ttau_{k+h'-h+Q } \left( \sigma^{ {\ttau_h(\xi)} } \xi\right)} \nu).
\end{equation}
Now, by bounded distortion (Theorem \ref{Bounded distortiion theorem}), there is global $C>0$ such that for every $\xi$ 
$$f ' _{\xi|_{\ttau_h(\xi) + \ttau_{k+h'-h+Q } \left( \sigma^{ {\ttau_h(\xi)} } \xi\right)}} = \Theta_C (e^{ -(h+k+h'-h+Q)\chi}) = \Theta_C (e^{ -(k+h'+Q)\chi})$$ 
on the other hand,  for every $\xi$
$$f ' _{\xi|_{ \tau_k (\xi)+\ttau_{h'} ( \sigma^{\tau_k (\xi)} (\xi))}} = \Theta_C (e^{ -(k+h')\chi}).$$
It follows that we can choose $Q$ based only on $C$ such that for every $\xi$, 
$$\tau_k (\xi)+\ttau_{h'} ( \sigma^{\tau_k (\xi)} (\xi)) \leq \ttau_h(\xi) + \ttau_{k+h'-h+Q } \left( \sigma^{ {\ttau_h(\xi)} } \xi\right) \leq \tau_k (\xi)+\ttau_{h'} ( \sigma^{\tau_k (\xi)} (\xi)) +P$$
where $P$ has uniformly finite length. Therefore, there is a  word $\rho_\xi$ of length $\leq P$ with
$$ \xi|_{\tau_k (\xi)+\ttau_{h'} ( \sigma^{\tau_k (\xi)} (\xi))}*\rho_\xi= \xi|_{\ttau_h(\xi) + \ttau_{k+h'-h+Q } \left( \sigma^{ {\ttau_h(\xi)} } \xi\right)}.$$
Plugging this equality into \eqref{Use of stopping time}, the Claim is proved.
\end{proof}

\subsubsection{Proof that Theorem \ref{Claim sufficient} implies Theorem \ref{Lemma sufficient}} Let $\epsilon>0$. Suppose that for every integer $q$ with $|q|\geq q^*(\epsilon,\bfp)$ there are values $h(q,\epsilon,\bfp)$, $k(q,\epsilon,\bfp)$ and $h'(q,\epsilon,\bfp)$ satisfying the conclusion of Theorem \ref{Claim sufficient}. Then, for $\mathbb{P}$ almost every $\omega$, by applying successively Theorem \ref{Theorem martingale} and Claim \ref{First linearization} with this $h$, Claim \ref{Claim stopping} with these $h,h',k$ and $\omega'$, Jensen's inequality, and finally Theorem \ref{Claim sufficient}, we get 
\begin{eqnarray*}
\left| \mathcal{F}_q ( \frac1N\sum_{n=1}^N  \delta_{T_p^n x_\omega})  \right|&\leq & \frac1N\sum_{n=1}^N \left|\mathcal{F}_q \left( T_p ^n \circ f_{\omega|_{\beta_{n,h} (\omega)}} \nu \right)  \right| +o_N(1) \\
&=& \frac1N\sum_{n=1}^N \sqrt{ \left|\mathcal{F}_q \left( T_p ^n \circ f_{\omega|_{\beta_{n,h} (\omega)}} \nu \right)  \right|^2} +o_N(1) \\
&\leq& \frac{\sqrt{2}}{N}\sum_{n=1}^N \sqrt{  \left|\mathcal{F}_{qs} \left( f_{\omega' |_{\tau_ h (\omega')}} \nu \right) \right|^2  +|q|e^{-(1+\gamma')h\chi} } +o_N(1) \\
&=& \frac{\sqrt{2}}{N}\sum_{n=1}^N \sqrt{  \left|\mathcal{F}_{qs} \left(  \mathbb{E}_{\xi \in \mathcal{A}^h (\omega')} \left( f_{\xi|_{ \tau_k (\xi)+\ttau_{h'} ( \sigma^{\tau_k (\xi)} (\xi))}} \circ f_{\rho_\xi} \nu \right) \right) \right|^2  +|q|e^{-(1+\gamma')h\chi} }  \\
&+& o_N(1) \\
&\leq& \frac{\sqrt{2}}{N}\sum_{n=1}^N \sqrt{ \mathbb{E}_{\xi \in \mathcal{A}^h (\omega')}\big|\mathcal F_{qs}(f_{\xi|_{ \tau_k (\xi)+\ttau_{h'} ( \sigma^{\tau_k (\xi)} (\xi))}} \circ f_{\rho_\xi} \nu)\big|^2  +|q|e^{-(1+\gamma')h\chi} }  \\
&+& o_N(1) \\
&\leq& \sqrt{2\epsilon} +o_N(1). \\
\end{eqnarray*}
Taking $N\rightarrow \infty$, Theorem \ref{Lemma sufficient} is proved.

\subsection{Proof of Theorem \ref{Claim sufficient}}
The proof has two stages. First, for a fixed $\omega \in \mathcal{A}^\mathbb{N}$ and $n\in \mathbb{N}$, we let $q\in \mathbb{Z}$ and $h,h',k >0$ be arbitrary, and use them to bound
$$\mathbb{E}_{\xi \in \mathcal{A}^h (\omega')}\big|\mathcal F_{qs}(f_{\xi|_{ \tau_k (\xi)+\ttau_{h'} ( \sigma^{\tau_k (\xi)} (\xi))}} \circ f_{\rho_\xi} \nu)\big|^2$$
with the notation $\omega'=\omega'(n,\omega)$ and $s(\omega,n,h)$ as in Theorem \ref{Claim sufficient}. The resulting bound will be a sum of several error terms, depending variously on $|q|,k,h,h'$ up to universal constants. These errors are produced by running a similar argument to the one proving Theorem \ref{Theorem main tech} part (1) as in Section \ref{Section proof 1}.  In the second stage of the proof, we let $\epsilon>0$ be small and show that we may choose specific parameters $h,h',k$ such that all of these error terms can be made arbitrarily small simultaneously, as long as $|q|$ is large in a manner that only depends on $\epsilon$. This will give Theorem \ref{Claim sufficient}.
\subsubsection{Collecting error terms} \label{Section error terms}
Fix $\omega$ and $n$ and let $\eta$ be such that $\mathcal{A}^h (\omega') = A_\eta$, and let $s$ be as in Theorem \ref{Claim sufficient}. From this point forward, we can forget about $n,\omega$ and just work with the cylinder  $A_\eta$ and the frequency $qs$. We do recall that $s,s^{-1}=O(1)$ uniformly in $n$ and $\omega$, and this will be used implicitly throughout the proof. Let $k,h,h' \geq 0$ be any parameters, and let $q\in \mathbb{Z}$. For notational convenience, we assume $q\cdot s,q>0$ - otherwise, whenever they appear inside a bound, an absolute value should be applied. By Theorem \ref{Theorem lin equid},  there exists a subset $\overline A_{k,\eta}^{h,h'}\subseteq A_\eta$ such that 
$$\mathbb P(\overline A_{k,\eta}^{h,h'})\geq \mathbb P(A_\eta)\cdot (1-o_{\bfp}^{\min(h,k-h)\to\infty}(1))$$
and (ii)-(iii) of Theorem \ref{Theorem lin equid} hold for it. In particular,
\begin{eqnarray*}
\mathbb{E}_{\xi \in \mathcal{A}^h (\omega')}\big|\mathcal F_{qs}(f_{\xi|_{ \tau_k (\xi)+\ttau_{h'} ( \sigma^{\tau_k (\xi)} (\xi))}} \circ f_{\rho_\xi} \nu)\big|^2 & =& \mathbb{E}_{\xi \in A_\eta}\big|\mathcal F_{qs}(f_{\xi|_{ \tau_k (\xi)+\ttau_{h'} ( \sigma^{\tau_k (\xi)} (\xi))}} \circ f_{\rho_\xi} \nu)\big|^2 \\
\end{eqnarray*}
\begin{eqnarray} \label{First error}
&\leq & \int_{ \xi \in \overline A_{k,\eta}^{h,h'}}  \mathbb{E}_{\cA_k^{h,h'}(\xi)  }\big|\mathcal F_{qs}(f_{\omega|_{ \tau_k (\omega)+\ttau_{h'} ( \sigma^{\tau_k (\omega)} (\omega))}} \circ f_{\rho_\omega} \nu)\big|^2 \, d\mathbb{P}(\xi) +o_{\bfp}^{\min(h,k-h)\to\infty}(1).
\end{eqnarray}
Here we made use of the fact that $|\mathcal{F}_q (\cdot )| \leq 1$. So, our first error term is $o_{\bfp}^{\min(h,k-h)\to\infty}(1)$.

Our next step is to linearize once more, in order to set the stage for the application of the upgraded local limit Theorem \ref{Theorem lin equid}:
\begin{Claim} (Second linerization) \label{Claim linear2} We have
$$  \int_{ \xi \in \overline A_{k,\eta}^{h,h'}}  \mathbb{E}_{\cA_k^{h,h'}(\xi)  }\big|\mathcal F_{qs}(f_{\omega|_{ \tau_k (\omega)+\ttau_{h'} ( \sigma^{\tau_k (\omega)} (\omega))}} \circ f_{\rho_\omega} \nu)\big|^2 \, d\mathbb{P}(\xi) $$
$$\leq \int_{ \xi \in \overline A_{k,\eta}^{h,h'}} \mathbb{E}_{\cA_k^{h,h'}(\xi)  }  \left| \mathcal{F}_{qs} \left(   M_{e^{-S_{\tau_k(\omega)} (\omega)}} \circ M_{ \sign\left( f'_{\omega|_{\tau_k(x)}} (x_{\sigma^{\tau_k(\omega)}(\omega)})\right)}  \circ f_{\eta'} \circ f_{\rho_{\omega}}  \nu \right) \right|^2 d \mathbb{P}( \xi) $$
$$+ O(q\cdot e^{-(k+h') \chi} e^{-\gamma' \cdot h' \chi})$$
where for every $\xi \in \overline A_{k,\eta}^{h,h'}$, recalling that $\cA_k^{h,h'}(\xi) = A_{k,\eta,\eta'}$, $\eta'$ is defined as  $\eta'=\eta'(\xi)$. 
\end{Claim}

\begin{proof}
This is, up to minor changes, Claim \ref{Claim linear}.
\end{proof}

So, our next error term is $q\cdot e^{-(k+h') \chi} e^{-\gamma' \cdot h' \chi}$ (up to multiplying by a global constant that we omit from notation, and recalling our assumption that $q>0$). With the help of the following Corollary, we can remove the randomness of the word $\rho_\xi$. 

\begin{Corollary} \label{Corollary}
 Let $P>0$ be as in Claim \ref{Claim stopping}. Then
$$ \int_{ \xi \in \overline A_{k,\eta}^{h,h'}} \mathbb{E}_{\cA_k^{h,h'}(\xi)  }  \left| \mathcal{F}_{qs} \left(   M_{e^{-S_{\tau_k(\omega)} (\omega)}} \circ M_{ \sign\left( f'_{\omega|_{\tau_k(x)}} (x_{\sigma^{\tau_k(\omega)}(\omega)})\right)}  \circ f_{\eta'} \circ f_{\rho_{\omega}}  \nu \right) \right|^2 d \mathbb{P}( \xi) $$
$$ \leq \sum_{|\rho|\leq P} \int_{ \xi \in \overline A_{k,\eta}^{h,h'}} \mathbb{E}_{\cA_k^{h,h'}(\xi)  }  \left| \mathcal{F}_{qs} \left(   M_{e^{-S_{\tau_k(\omega)} (\omega)}}   \circ f_{\eta'} \circ f_{\rho}  \nu \right) \right|^2 d \mathbb{P}( \xi).$$
\end{Corollary}
Recalling that $|\rho_\omega|\leq P$ for all $\omega$, this is analogues to Corollary \ref{Coro finite sum} and Remark \ref{Remark} via  an assumption (without the loss of generality) that all maps are orientation preserving. We are now ready to apply  Theorem \ref{Theorem lin equid}:
\begin{Claim} \label{Claim LLT} (Application of local limit Theorem)
For every $\rho$ as in the sum in Corollary \ref{Corollary},
$$ \int_{ \xi \in \overline A_{k,\eta}^{h,h'}} \mathbb{E}_{\cA_k^{h,h'}(\xi)  }  \left| \mathcal{F}_{qs} \left(   M_{e^{-S_{\tau_k(\omega)} (\omega)}}   \circ f_{\eta'} \circ f_{\rho}  \nu \right) \right|^2 d \mathbb{P}( \omega)$$
$$\leq \int_{ \xi \in \overline A_{k,\eta}^{h,h'}}  \int_{k\chi} ^{k\chi +D'} \left| \mathcal{F}_{qs} \left(M_{e^{-x}} \circ f_{\eta'} \circ f_\rho    \nu \right) \right|^2 d \Gamma_{A_{k,\eta,\eta'}(\xi)} (x)   d \mathbb{P}( \xi)$$
$$ +O\left(\frac{2}{q e^{-(k+h')\chi}}+  (q e^{-(k+h')\chi})^2 o_{\bfp}^{\min(h,k-h)\to\infty}(1) \right). $$
\end{Claim}
\begin{proof}
This follows from Theorem \ref{Theorem lin equid} in a similar manner to the derivation of Theorem \ref{Theorem equid 2} from Theorem \ref{Theorem equid} in Section \ref{Section proof 1}.
\end{proof}
So, the next error term added to our list is $\frac{2}{q e^{-(k+h')\chi}}+  (q e^{-(k+h')\chi})^2 o_{\bfp}^{\min(h,k-h)\to\infty}(1)$ (again, up to a universal multiplicative constant that we ignore).

Next, using the uniform norm on the density of the measure $\Gamma_{A_{k,\eta,\eta'}(\xi)}$ (Lemma \ref{Lemma Abs. contin of gamma}) we get
$$\sum_{|\rho|\leq 2P} \int_{ \xi \in \overline A_{k,\eta}^{h,h'}}  \int_{k\chi} ^{k\chi +D'} \left| \mathcal{F}_{qs} \left(M_{e^{-x}} \circ f_{\eta'} \circ f_\rho    \nu \right) \right|^2 d \Gamma_{A_{k,\eta,\eta'}(\xi)} (x)   d \mathbb{P}( \xi)$$
$$\leq \sum_{|\rho|\leq 2P} \int_{ \xi \in \overline A_{k,\eta}^{h,h'}} \left(   \int_{k\chi} ^{k\chi +D'} \left| \mathcal{F}_{qs} \left(M_{e^{-z}} \circ f_{\eta'} \circ f_\rho \nu \right) \right|^2 \cdot \frac{1}{D} dz   \right) d \mathbb{P}( \xi).  $$
So, we have reduced our problem to a  sum of oscillatory integrals, that has uniformly bounded many terms. This will give rise to the final error term:

\begin{Claim} \label{Claim integral} (Oscillatory integral)
 For every $\rho$ in the sum above, and for every $\delta>0$,
 $$ \int_{ \xi \in \overline A_{k,\eta}^{h,h'}} \left(   \int_{k\chi} ^{k\chi +D'} \left| \mathcal{F}_{qe^s} \left(M_{e^{-z}} \circ f_{\eta'} \circ f_\rho \nu \right) \right|^2 \cdot \frac{1}{D} dz   \right) d \mathbb{P}( \xi)  \leq O \left( \frac{1}{\delta q e^{-(k+h')}}+ \sup_{y} \nu \left(B_\delta (y) \right) \right).$$
\end{Claim}
\begin{proof}
This is analogues to the application of Lemma \ref{Lemma 3.2 } at the end of the proof of Theorem \ref{Theorem main tech} part (1) in Section \ref{Section proof 1}.
\end{proof}
So, our last error term is $\frac{1}{\delta q e^{-(k+h')}}+ \sup_{y} \nu \left(B_\delta (y) \right) $, for any $\delta>0$ (to be chosen later at our convenience). Yet again, we ignore the underlying global multiplicative constant.
\subsubsection{Conclusion of the proof of Theorem \ref{Claim sufficient}}
Let us first recall the error terms we collected in the previous Section:

\noindent{ {\bf List of error terms :}}

Every bound below is in the sense of $\lesssim_\bfp$. Recall that $\gamma' \in (0,\gamma)$.

Local limit Theorem (Claim \ref{Claim LLT}), and \eqref{First error}
$$ \frac{2}{q e^{-(k+h')\chi}}+ (q e^{-(k+h')\chi})^2 o_{\bfp}^{\min(h,k-h)\to\infty}(1);$$

Second linearization (Claim \ref{Claim linear2}): 
$$   qe^{-(k+h')\chi}e^{-\gamma'h'\chi};$$

Oscillatory integral (Claim \ref{Claim integral}): For every $\delta>0$,
$$\frac 1  {\delta q e^{-(k+h')\chi}} + \sup_y\nu(B_\delta(y));$$

First linearization (Claim \ref{First linearization}):
$$ qe^{-(1+\gamma') h\chi}.$$

Our goal is to show that for every $\epsilon>0$ there is some $q^* (\epsilon) \in \mathbb{N}$ such that for all integer $|q| \geq q^*$,  there exists a choice of $h,k,h'$ based only on $\epsilon$ and $q$, so that every term in the list above is at most $\epsilon$. Recall that we are assuming, without the loss of generality, that $q>0$. We will also be minded to take care of the other constraints: $h$ needs to be large in a manner dependent on the prefixed parameter $\gamma'$ (Claim \ref{First linearization}), and $h<h'+k$ (Claim \ref{Claim stopping}).

{\bf Choices of parameters:}
\begin{enumerate}
\item Fix $\delta=\delta(\epsilon,\bfp)$ such that $\sup_y\nu(B_\delta(y))\ll_\bfp \epsilon$. Here we use Lemma \ref{Lemma nu is continuous}.
\item Fix $h'=h'(\epsilon,\bfp)$ such that $e^{-\gamma'h'\chi}=\delta\epsilon^2$.
\item Fix $h^*=h^*(\epsilon,\bfp)$ such that if $\min(h,k-h)\geq h^*$, then the $o_\bfp^{\min(h,k-h)\to\infty}(1)$ term in Theorem \ref{Theorem lin equid} is $\ll \delta^2\epsilon^3$.
\item Fix $h=h(\epsilon,\bfp,q)\geq \frac1{\gamma'}(h^*+h^\triangle)>h^*$  for some properly chosen $h^\triangle=h^\triangle(\epsilon,\bfp)>0$ (arising from Claim \ref{Claim k-h} below) such that $q e^{-(1+\gamma')h\chi}= \epsilon$. This can be done if $q \geq q^*$ for some $q^*=q^*(\epsilon,\bfp)$.
\item Fix $k=k(\epsilon,\bfp,q)$ such that $qe^{-(k+h')\chi}= \delta^{-1}\epsilon^{-1}$.
\end{enumerate}

\begin{Claim}\label{Claim k-h}If $q\geq q^*$ for a sufficiently large $q^*=(\epsilon,\bfp)$, then $k-h\geq h^*$. \end{Claim}
\begin{proof} We have
$$e^{k\chi-(1+\gamma')h\chi}=\delta\epsilon^2 e^{-h'\chi}= \left( \delta\epsilon^2  \right)^{ 1+ \frac{1}{\gamma'}}.$$ 
Thus $k-h= \gamma'h-O_{\epsilon,\bfp}(1)$. Denote this last $O_{\epsilon,\bfp}(1)$ by $h^\triangle$. Then $k-h\geq h^*$ following the choice of $h$ above. \end{proof}

\textbf{Note} To employ Claim \ref{Claim stopping}, we need that $k+h'>h$. But this  clearly follows from Claim \ref{Claim k-h}. Also, to use Claim \ref{First linearization} we need $h$ to be large enough in a manner dependent on $\gamma'$, but this can clearly be arranged in step (4) by potentially making $q$ larger.
$$ $$

With these parameters, all errors are simultaneously small:

Local limit Theorem:
$$\frac{2}{q e^{-(k+h')\chi}}+ ( qe^{-(k+h')\chi})^2o_\bfp^{\min(h,k-h)\to\infty}(1)\ll_\bfp \delta \cdot \epsilon+ \delta^{-2}\epsilon^{-2}\cdot\delta^2\epsilon^3\leq 2\epsilon;$$

Second linearization:
$$   qe^{-(k+h')\chi}e^{-\gamma'h'\chi}=\delta^{-1}\epsilon^{-1}\cdot\delta\epsilon^2=\epsilon;$$

Oscillatory integral:
$$\frac 1  {\delta q e^{-(k+h')\chi}} + \sup_y\nu(B_\delta(y))\ll_\bfp\frac1{\delta\cdot\delta^{-1}\epsilon^{-1}}+\epsilon\ll 2\epsilon;$$

First linearization:
$$ qe^{-(1+\gamma') h\chi}=\epsilon.$$

\noindent{ Thus, recalling that}
$$\mathbb{E}_{\xi \in \mathcal{A}^h (\omega')}\big|\mathcal F_{qe^s}(f_{\xi|_{ \tau_k (\xi)+\tau_{h'} ( \sigma^{\tau_k (\xi)} (\xi))}} \circ f_{\rho_\xi} \nu)\big|^2 $$
was shown to be bounded by the sum of the first three error terms mentioned above,
 $$ \mathbb{E}_{\xi \in \mathcal{A}^h (\omega')}\big|\mathcal F_{qe^s}(f_{\xi|_{ \tau_k (\xi)+\tau_{h'} ( \sigma^{\tau_k (\xi)} (\xi))}} \circ f_{\rho_\xi} \nu)\big|^2 + qe^{(1+\gamma')h\chi}\leq C_0 \epsilon$$ for uniform $C_0$, which is what we want.

Thus, Theorem \ref{Claim sufficient} is proved. We have shown that it implies Theorem \ref{Main theorem}, and since $p$ was arbitrary, this implies Theorem \ref{Theorem main tech} part (2).

\section{Proof of Corollary \ref{Main Corollary}} \label{Section coro}
\subsection{Proof of Corollary \ref{Main Corollary} part (1)}
Let $\Phi$ be a $C^{1+\gamma}$ IFS, and let $c(\cdot ,\cdot )$ be the derivative cocycle. Recall that $\mathcal{A}=\lbrace1,...,n\rbrace$  let  $H^{\kappa}$ denote the space of $\kappa$-H\"older continuous maps $\mathcal{A}^\mathbb{N} \rightarrow \mathbb{C}$. Recall that we define
\begin{eqnarray*} 
\Lambda_c &=& \lbrace \theta: \exists \phi_\theta \in H^{\kappa} \text{ with } |\phi_\theta|=1 \text{ and } u_\theta \in S^1 \text{ such that } \forall (a,\omega)\in \mathcal{A} \times \mathcal{A}^\mathbb{N}, \\
&& \quad \quad \quad \phi_\theta(\iota_a(\omega)) = u_\theta \exp(-i \theta \cdot c(a,\omega))\cdot \phi_\theta (\omega) \rbrace. \nonumber
\end{eqnarray*}
Following\footnote{In fact, aperiodic cocycles are defined in \cite[equation (15.8)]{Benoist2016Quint} in a different way, by a certain spectral gap property. However, It is a consequence of \cite[Lemma 15.3]{Benoist2016Quint} that the two definitions are equivalent.} Benoist and Quint \cite{Benoist2016Quint} we say that $c$ is an \textit{aperiodic} cocycle if 
$$\Lambda_c = \lbrace 0 \rbrace.$$

Next, writing $\Phi = \lbrace f_1,...,f_n \rbrace$, we define
$$ F_\Phi = \left\lbrace \log \left| f_i ' \left( y_i \right) \right| : \text{ where } f_i (y_i)=y_i,\quad i\in \mathcal{A} \right\rbrace.$$
Notice that $F_\Phi$ is precisely the set that appears in Corollary \ref{Main Corollary} part (1).

\begin{Lemma} \label{Lemma lattice}
If $c$ is not aperiodic (i.e. it is periodic) then $F_\Phi$ belongs to a translation of a lattice.
\end{Lemma}
\begin{proof}
The assumption that $c$ is not aperiodic means that there exists $0\neq \theta \in \Lambda_c$. So, there exists  $ \phi \in H^{\kappa}$  with  $|\phi|=1$  and  $u \in S^1$  such that  for all $(a,\omega)\in \mathcal{A} \times \mathcal{A}^\mathbb{N}$,
$$\phi(\iota_a(\omega)) = u\cdot \exp(-i \theta \cdot c(a,\omega))\cdot \phi (\omega). $$
Now, fix $1\leq a \leq n$ and let $\omega = (a,a,a....) \in \mathcal{A}^\mathbb{N}$. Plugging these into the equation above, we obtain
$$1= u \exp(-i \theta \cdot( -\log f_a ' (x_\omega))). $$
This equation  implies that $F_\Phi$ belongs to a translation (determined by $u$) of the lattice $\frac{2\pi}{\theta} \mathbb{Z}$.
\end{proof}

\noindent{\textbf{Proof of Corollary \ref{Main Corollary} part (1)}} This is immediate from  Lemma \ref{Lemma lattice} and Theorem \ref{Theorem main tech} part (1).
\subsection{Proof of Corollary \ref{Main Corollary} part (2)}
Let $\Phi = \lbrace f_i(x)=r_i \cdot x+t_i \rbrace_{i\in \mathcal{A}}$ be an aperiodic self-similar IFS. Recall that $\Phi$ is aperiodic if there are $i,j\in \mathcal{A}$ such that $\frac{\log |r_i|}{\log |r_j|}\notin \mathbb{Q}$. Without the loss of generality, we assume $i=1$ and $j=2$.  The following Lemma shows that from $\Phi$ we may construct an IFS $\Psi$ such that $F_\Psi$ does not belong to a translation of a lattice (and so its derivative cocycle is aperiodic), and such that every self similar measure with respect to $\Phi$ is also a self similar measure with respect to $\Psi$.
\begin{Lemma} \label{Lemma aperiodic}
Let $\Psi = \lbrace f_1 \circ f_i \rbrace_{i\in \mathcal{A}} \bigcup \lbrace f_i \rbrace_{i=2,..,n}$. Then:
\begin{enumerate}
\item The set $F_\Psi$ does not belong to a translation of a lattice.

\item Let $\nu$ be a self similar measure with respect to $\Phi$ and the probability vector $\mathbf{p}$. Then there exists a probability vector $\mathbf{q}$ such that $\nu$ is a self similar measure with respect to $\mathbf{q}$ and the IFS $\Psi$.
\end{enumerate}
\end{Lemma}
\begin{proof}
For part (1), since $\Phi$ and $\Psi$ are self similar, if $F_\Psi$ belongs to a translation of a lattice then for every $e_1,e_2,e_3\in \mathbb{R}, e_2 \neq e_3,$ that arise by taking the $\log$ of  contraction ratios of $\Psi$, we have
$$\frac{e_1 - e_2}{e_3-e_2} \in \mathbb{Q}.$$
So, taking $e_1=\log \left| r_1\cdot r_1 \right|, e_2 = \log \left|r_2\cdot r_1\right|, e_3 = \log \left|r_2\right|$ we obtain
$$\frac{\log \left| r_1\cdot r_1 \right| - \log \left|r_2\cdot r_1\right|}{ \log \left|r_2\right|-\log \left|r_2\cdot r_1\right|} \in \mathbb{Q} $$
which implies that
$$\frac{\log |r_2|}{\log |r_1|}-1  =\frac{\log |r_1| - \log |r_2|}{-\log |r_1|} \in \mathbb{Q}.$$
This contradicts our assumption that $r_1 \not \sim r_2$, and concludes the proof of part (1).

For part (2), one may verify that $\nu$ is a self similar measure with respect to $\Psi$ and the probability vector
$$\mathbf{q}:=\left( p_1 ^2, p_1 \cdot p_2,..., p_1\cdot p_n, p_2, p_3,...,p_n \right)$$
which is strictly positive since $\mathbf{p}$ is strictly positive. 
\end{proof}

We need one more standard Lemma:
\begin{Lemma} \label{Lemma conjugate}
Let $\Phi$ be an aperiodic self similar IFS on the interval $I$, and let $\Psi$ be the induced IFS as in Lemma \ref{Lemma aperiodic}. Let $g:I\rightarrow g(I)$ be a $C^{1+\gamma}(I)$ map with non vanishing derivative. Then every self conformal measure with respect to the conjugated IFS $g\circ \Phi \circ g^{-1} $ is also a self conformal measure with respect to the IFS $\Theta = g\circ \Psi \circ g^{-1} $. Furthermore, the derivative cocycle of the  IFS $\Theta $ is  aperiodic. 
\end{Lemma}
\begin{proof}
It is elementary that under these assumptions $F_\Psi = F_\Theta$. So, $F_\Theta$ equals $F_\Psi$ that does not lie on a translation of a lattice by Lemma \ref{Lemma aperiodic} part (1), and so $F_\Theta$ does not lie on a translation of a lattice. By Lemma \ref{Lemma lattice}  this means that the derivative cocycle of $\Theta$ is aperiodic. The Claim about the self conformal measures is an immediate consequence of Lemma \ref{Lemma aperiodic} part (2), since every self conformal measure $\nu$ with respect to $g\circ \Phi \circ g^{-1}$ is equal to $g \mu$, where $\mu$ is a self similar measure with respect to $\Phi$ with the same weights. 
\end{proof}

Finally, let $\Phi$ be an aperiodic self similar IFS and let $g:I\rightarrow g(I)$ be a $C^{1+\gamma}(I)$ map with non vanishing derivative. Then by Lemma \ref{Lemma conjugate} any self conformal measure with respect to the conjugated IFS $g\circ \Phi \circ g^{-1} $ is also a self conformal measure with respect to a uniformly contracting $C^{1+\gamma}$ IFS $\Theta$ such that its derivative cocycle is aperiodic. So, applying Theorem \ref{Theorem main tech},  this concludes the proof of the normality and the Rajchman assertions of Corollary \ref{Main Corollary} part (2).

\subsubsection{Proof of the quantitative assertion of Corollary \ref{Main Corollary} part (2)}
Here we assume that $\Phi$ is a Diophantine  self similar IFS (recall the definition from \eqref{Eq Dio condition}). Let $\nu$ be a self similar measure with respect to the probability vector $\mathbf{p}$. Our goal  is to show that there exists some $\alpha=\alpha(\nu)>0$ such that
$$ \left|\mathcal{F}_q (\nu)\right| \leq O\left( \frac{1}{ \left| \log |q| \right| ^\alpha} \right),\, \text{ as } |q|\rightarrow \infty.$$
Let $\lbrace r_1,..., r_n \rbrace$ denote the contraction ratios of $\Phi$, and
let $\mu$ be the distribution $\mathbf{p}$ induces on $\lbrace - \log |r_1|,..., -\log|r_n| \rbrace$. 
Our first observation is that in this case the random walk 
$$S_n(\omega):=-\log|f'_{\omega|_n}(x_{\sigma^n(\omega)})| = -\log|f'_{\omega|_n}(x_0)|,\quad \text{ for any prefixed } x_0\in I$$
as in Section \ref{Section LLT}, is in fact a classical random on $\mathbb{R}$ and its law is given by $\mu^{*n}$, the $n$-fold self convolution of $\mu$. Since for such random walks effective versions of the central and local limit Theorems are available (which is why  the Diophantine condition is imposed), we can strengthen Theorem \ref{Theorem equid} by specifying a rate:
\begin{theorem}  \label{Theorem equid quant}
There exists some $\delta=\delta(\mathbf{p})>0$ such that for every  $k, h'>0$,  $h=0$ and $\eta = $ the empty word: 

 There exists a subset $\mathcal{A}_{k,\eta}^{h,h'}\subseteq \mathcal{A}^\mathbb{N}$   such that, as $k$ tends to $\infty$:\begin{enumerate} [label=(\roman*)]
\item $\mathbb P(\mathcal{A}_{k,\eta} ^{h,h'})\geq 1-O_{\mathbf{p}}(\frac{1}{k^\delta})$.
\item for all $\xi\in \mathcal{A}_{k,\eta}^{h,h'}$, $\mathbb{P}(\mathcal{A}_k^{h,h'}(\xi))>0$.
\item  for all $\xi\in \mathcal{A}_{k,\eta}^{h,h'}$ and for any sub-interval $J\subseteq  [k\chi, k\chi+D']$, 
$$
 \mathbb{P}_{\mathcal{A}_{k}^{h,h'} (\xi)}(S_{\tau_k}\in J)= \Gamma_{\mathcal{A}_{k}^{h,h'}(\xi)} (J)  +O_{\mathbf{p}}(\frac{1}{k^\delta}).
$$
\end{enumerate}
\end{theorem}
The difference between Theorem \ref{Theorem equid quant} and Theorem \ref{Theorem equid} is that  the error term $o^{k\to\infty}_{h_0,\bfp}(1)=O_{\mathbf{p}}(\frac{1}{k^\delta})$ is explicit (here we always take $h_0=0$). We proceed to explain how Theorem \ref{Theorem equid quant} gives us the desired logarithmic decay rate for $\mathcal{F}_q (\nu)$. Afterwards, we explain how to obtain Theorem \ref{Theorem equid quant} by modifying the proof of Theorem \ref{Theorem equid}.

\noindent{\textbf{Proof that Theorem \ref{Theorem equid quant} implies logarithmic decay}} Let $\delta>0$ be as in Theorem \ref{Theorem equid quant}, fix $h=0$ and let $q$ be large. Find some $C>0$ such that for every $q$ large enough there exists  $k>0$ with, letting $h'=\sqrt{k}$,
$$|q| = \Theta_C \left( k^{\frac{\delta}{4}} \cdot e^{(k+h')\chi} \right).$$ 
Notice that asymptotically  $k \approx \log |q|$. Following the argument in Section \ref{Section proof 1}, we bound $\left| \mathcal{F}_q (\nu) \right|$ by the sum of the following terms. As usual, every bound below is in the sense of $\lesssim_\bfp$, $0<\gamma'<1$, $q$ is assumed to be positive, and we ignore global multiplicative constants.

Linearization - Claim \ref{Claim linear} (note that in the self-similar case this step can be easily bypassed, but for consistency we still take this term into account) :
$$   qe^{-(k+h')\chi}e^{-\gamma'h'\chi};$$

Local limit Theorem  - proof of Theorem \ref{Theorem equid 2} and the discussion following Corollary \ref{Coro finite sum}:
$$ \frac{2}{q e^{-(k+h')\chi}}+ (q e^{-(k+h')\chi})^2 \frac{1}{k^\delta};$$
Here in equation \eqref{Term 2} in the proof of Theorem \ref{Theorem equid 2} we use Theorem \ref{Theorem equid quant} instead of  Theorem \ref{Theorem equid}.

Oscillatory integral: For every $r>0$,
$$\frac 1  {r q e^{-(k+h')\chi}} + \sup_y\nu(B_r(y)).$$

\noindent{ {\bf Choice of parameters :}} Recall that $h'=\sqrt{k}$ and fix $r=k^{\frac{-\delta}{8}}$. Then we get:

linearization:
$$   qe^{-(k+h')\chi}e^{-\gamma'h'\chi}= k^{\frac{\delta}{4}}\cdot e^{-\gamma'\sqrt{k}\chi}, \quad  \text{ This decays exponentially fast in } k. $$

Local limit Theorem:
$$ \frac{2}{q e^{-(k+h')\chi}}+ (q e^{-(k+h')\chi})^2 \frac{1}{k^\delta} = \frac{2}{k^{\frac{\delta}{4}}} + k^{\frac{\delta}{2}} \cdot \frac{1}{k^\delta}, \quad  \text{ This decays polynomially fast in } k. $$

Oscillatory integral: There is some $d>0$ such that 
$$\frac 1  {r q e^{-(k+h')\chi}} + \sup_y\nu(B_r(y)) \leq  \frac{k^{\frac{\delta}{8}}}{k^{\frac{\delta}{4}}} +  k^{  \frac{-d \cdot \delta}{8}}, \quad  \text{ This decays polynomially fast in } k.$$
Here we made use of \cite[Proposition 2.2]{feng2009Lau}, where it is shown  that there is some $C>0$ such that for every $r>0$ small enough
$$\sup_y\nu(B_r(y)) \leq Cr^d.$$
Finally, by summing these error terms we see that for some $\alpha=\alpha(\nu)>0$ we have $|\mathcal{F}_q (\nu)|=O(\frac{1}{k^\alpha})$. Since $k\approx \log |q|$ our claim follows.

\noindent{\textbf{Proof of Theorem \ref{Theorem equid quant}}} Recall that here  $S_n \sim \mu^{*n}$. Thus, we  follow the proof of Theorem \ref{Theorem equid} essentailly verbatim, only in the proof Proposition \ref{Prop equid central}  we use the effective Berry-Esseen inequality \cite{Berry1942Esseen} instead of Theorem \ref{B-Q CLT}, and in the proof of Proposition \ref{Prop equid local} we use Breuillard's effective local limit Theorem \cite[Th\'{e}or\`eme  4.2]{Breuillard2005llt} instead of Theorem \ref{B-Q LLT general}. Indeed, recall that the Berry-Esseen inequality yields a rate of $O_\mathbf{p}( \frac{1}{\sqrt{n}})$ in the central limit Theorem  for $S_n$. So, it is straightforward to see that  applying the Berry-Esseen inequality instead of Theorem \ref{B-Q CLT} in the proof of Proposition \ref{Prop equid central} yields that the error term $o_{h_0,\textbf{p}}^{k\to\infty}(1)=o_{0,\textbf{p}}^{k\to\infty}(1)$ decays polynomially in $k$. As for Proposition \ref{Prop equid local}, we require the following Theorem of Breuillard. Recall that we are assuming $\Phi$ is Diophantine in the sense that \eqref{Eq Dio condition} holds true.
\begin{theorem} \cite[Th\'{e}or\`eme  4.2 and Remarque 4.1]{Breuillard2005llt} \label{Bre LLT}
There exists a sequence $\epsilon_n(\mu)$ and some $\delta>0$ such that $\epsilon_n = o(n^{-\delta})$ and the centred distribution $\theta$, where $\theta(A):=\mu(A-\chi)$, satisfies:

 Let $r>0$ denote the variance of the Gaussian on $\mathbb{R}$ associated with $\mu$ in the CLT, and fix $R>0$. For $x\in \mathbb{R}$, $s>0$ and the interval $I=I_s = [-s,s]$  we have, for $n\in \mathbb{N}$,
$$\sup \left\lbrace \left| \frac{1}{G_{\sqrt{n}r} (x)}  \theta^{*n} (I+x) - \lambda(I_s) \right|: \, |x|+s \leq \sqrt{R n \log n}, \, s > n^{-\delta} \right\rbrace  \leq \epsilon_n (\mu)$$
where we recall that $G_t(\cdot)$ stands for the density of the Gaussian law $N(0,t^2)$. 

\end{theorem}
We remark that  the statement of \cite[Th\'{e}or\`eme 4.2]{Breuillard2005llt} does not include  $G_t$, but this version of the Theorem  follows easily by combining it with the arguments of Benoist-Quint as in \cite[Section 16.3] {Benoist2016Quint}, specifically with \cite[Lemma 16.13]{Benoist2016Quint}. Now, fix $\delta$ as in Theorem \ref{Bre LLT}. Instead of using Theorem \ref{B-Q LLT general} in equation  \eqref{EqLLT}, we use Theorem \ref{Bre LLT} and translates $W$ of diameter $n^{\frac{-\delta}{4}}$, so that the error term $o^{n\to\infty}_{\epsilon,h_0,\textbf{p}} (1)$ becomes the polynomially decaying $\epsilon_n (\mu)$. Noting that Lemma \ref{LemGaussian} also holds with a polynomial rate, it is now a straightforward versification that Proposition \ref{Prop equid local} holds with a polynomial rate. Theorem \ref{Theorem equid quant} is proved.

\subsection{Some examples of Diophantine IFS's} \label{Section Moser}
Let $\Phi$ be an orientation preserving self similar IFS  with  contraction ratios $\lbrace r_1,...,r_n \rbrace$. Li and Sahlsten \cite[Theorem 1.3]{li2019trigonometric}  proved that every self similar measure has logarithmic Fourier decay if the following condition holds: 
\begin{equation} \label{Eq Dio li-sa}
\text{ There exist } C>0,l>2 \text{ and } r_i,r_j \text{ such that } |\frac{\log r_i}{\log r_j} - \frac{p}{q}| \geq \frac{C}{q^l} \text{ for all } p\in \mathbb{Z} \text{ and } q\in \mathbb{N}.
\end{equation}
Recall our   Diophantine condition from \eqref{Eq Dio condition}. In this Section we will indicate a family of IFS's that are Diophantine in the sense of \eqref{Eq Dio condition} but not in the sense of \eqref{Eq Dio li-sa}.

To this end, notice that any rational number fails the condition \eqref{Eq Dio li-sa}, and recall that an irrational number is called Liouville if it fails \eqref{Eq Dio li-sa}. So, to produce the desired examples it is clearly sufficient to find sets $\lbrace v_1,...,v_n \rbrace$ of strictly positive real numbers such that:
\begin{enumerate} [label=(\roman*)]
\item For every $i,j$ we have that $\frac{v_i}{v_j}$ is either rational or Liouville.

\item There are $l, C>0$ such that 
$$\inf_{y\in \mathbb{R}} \max_{i\in \lbrace 1,...n\rbrace } d(\,v_i \cdot x+y,\, \mathbb{Z}) \geq \frac{C}{|x|^{l}}, \text{ for all } x\in \mathbb{R} \text{ large enough in absolute value}.$$
\end{enumerate}
We require the following Theorem of Moser \cite{Moser1990Dio}:
\begin{theorem} \cite[Theorem 2]{Moser1990Dio} \label{Theorem Moser}
For every $n\geq 2$ and $\tau > \frac{2}{n-1}$ there exists a set of the cardinality of the continuum of vectors $(\alpha_1,...,\alpha_n)\in \mathbb{R}^n$ such that:
\begin{enumerate} 
\item There exists some $D>0$ such that for every $q\in \mathbb{N}$ we have
$$\max_{i=1,...,n} d(q\cdot \alpha_i, \mathbb{Z}) \geq D\cdot q^{-\tau}.$$

\item For every linearly independent vectors $g,h\in \mathbb{Z}^{n+1}$ 
$$ \text{ The ratio } \frac{g_0 + g_1 \alpha_1 +...+g_n \alpha_n}{h_0 + h_1 \alpha_1 +...+h_n \alpha_n} \text{ is a Liouville number}.$$
\end{enumerate}
\end{theorem}
We proceed to construct $\lbrace v_1,...,v_4 \rbrace$ with properties (i) and (ii) as above. Let $\tau >2$ and find $(\alpha_1,\alpha_2)$ as in Theorem \ref{Theorem Moser}. Assuming without the loss of generality that both $\alpha_1,\alpha_2>0$ we define 
$$v_1 =1,\quad v_2 = 2, \quad v_3 = \alpha_1 +1, \quad v_4 = \alpha_2 +1$$
(if $\alpha_i<0$ for $i=1,2$ we define $v_{i+2}=1-\alpha_i$ and proceed with the same proof). By Theorem \ref{Theorem Moser} part (2) for every $i,j$ the ratio $\frac{v_i}{v_j}$ is either rational or Liouville. Therefore, (i) holds. To verify (ii), let $|x|\gg 1$ and $y\in \mathbb{R}$. We may assume that
$$ d(xv_i +y, \mathbb{Z}) < \frac{1}{|x|^{2\tau}},\quad i=1,2.$$
So for $i=1,2$ there are $n_i \in  \mathbb{Z}$ and $\epsilon_i$ with $|\epsilon_i|<\frac{1}{|x|^{2\tau}}$ such that
$$xv_i + y = n_i + \epsilon_i.$$
Note that this implies $x=\frac{(n_2-n_1)+(\epsilon_2-\epsilon_1)}{v_2-v_1}=(n_2-n_1)+(\epsilon_2-\epsilon_1)$. So, for $i=3,4$ and any $k\in \mathbb{Z}$
\begin{eqnarray*}
|xv_i +y+k| & = & | x(v_i -v_1)+n_1+k+\epsilon_1|\\
&=& \left| (n_2-n_1) \frac{v_i-v_1}{v_2-v_1}+n_1+k + \frac{\epsilon_2-\epsilon_1}{v_2-v_1} \right| \\
&\geq & \left| (n_2-n_1) \frac{v_i-v_1}{v_2-v_1}+n_1+k \right| - \frac{2}{|x|^{2\tau}} \\
&\geq & d(\, (n_2-n_1) \frac{v_i-v_1}{v_2-v_1}, \, \mathbb{Z}) - \frac{2}{|x|^{2\tau}}.
\end{eqnarray*}
Finally since $\frac{v_i-v_1}{v_2-v_1}=\alpha_{i-2}$ for $i=3,4$, and $|x|\geq C_0|n_2-n_1|$ for some global $C_0>0$ as long as $|x|$ is large enough, by part (1) of Theorem \ref{Theorem Moser} we obtain:
\begin{eqnarray*}
\max_{i=3,4} d(\,v_i \cdot x+y,\, \mathbb{Z}) & \geq & \max_{i=3,4} d(\, (n_2-n_1) \frac{v_i-v_1}{v_2-v_1}, \, \mathbb{Z}) - \frac{2}{|x|^{2\tau}} \\
&=& \max_{i=1,2} d(\, (n_2-n_1) \alpha_i, \, \mathbb{Z}) - \frac{2}{|x|^{2\tau}} \\
&\geq & \frac{D}{|n_2-n_1|^{\tau}} - \frac{2}{|x|^{2\tau}} \geq \frac{C'}{|x|^\tau}.
\end{eqnarray*}
For some global constant $C'$, as long as $|x|$ is large enough. This proves (ii) for $\lbrace v_1,...,v_4 \rbrace$ as claimed. Finally, notice that essentially the same proof yields a wide class of further examples of Diophantine IFS's satisfying \eqref{Eq Dio condition} but not \eqref{Eq Dio li-sa} with  arbitrarily many contraction ratios (in particular, more than $4$).

\begin{Remark} \label{Remark Dio}
As pointed out to us by the anonymous referee, if $\Phi$ satisfies that there are $l , C>0$ such that 
\begin{equation} \label{Eq Dio condition 2}
\max_{i\in \lbrace 1,...n\rbrace } d(\,\log |r_i|\cdot x,\, \mathbb{Z}) \geq \frac{C}{|x|^{l}}, \text{ for all } x\in \mathbb{R} \text{ large enough in absolute value}
\end{equation}
then the induced IFS $\Psi$ as in Lemma \ref{Lemma aperiodic} is Diophantine in the sense of \eqref{Eq Dio condition}. In general, however, \eqref{Eq Dio condition 2} does not imply \eqref{Eq Dio condition} (but the converse is obviously true). To see this, let $x,y\in \mathbb{R}$ be such that $y$ is a Liouville number and $x$ is a Diophantine number in the sense of \cite[Equation (1.2)]{li2019trigonometric}. Then the vector $v=(1,x,1+(x-1)y)$ fails \eqref{Eq Dio condition}: Indeed, \eqref{Eq Dio condition} holds for $v=(x-1)\cdot (0,1,y)+(1,1,1)$ if and only if it holds for $(0,1,y)$, and clearly $(0,1,y)$ fails \eqref{Eq Dio condition}. On the other hand, it can be shown that $v$ satisfies \eqref{Eq Dio condition 2} since $x$ is Diophantine.

The main reason why  we work with \eqref{Eq Dio condition} rather than \eqref{Eq Dio condition 2} is that in an upcoming paper we will show that a certain analogue of it for $C^{1+\gamma}$ non-linear IFS's implies effective Fourier decay for self conformal measures. In this non-linear setting it is not clear that if we assume an analogue of \eqref{Eq Dio condition 2} then the induced IFS as in Lemma \ref{Lemma aperiodic} will satisfy the needed condition. We also note that \eqref{Eq Dio condition}  is consistent with the related works \cite{Breuillard2005llt, Dolgopyat1998rapid,  Weak2017Cramer} that we either make use of here, or plan to make use of in future works.
\end{Remark}  

\subsection{Proof of Corollary \ref{Main Corollary} part (3)}
Let $\Phi = \lbrace f_1,...,f_n \rbrace$ be a uniformly contracting $C^{r}$ smooth IFS on the interval $I$, where either $r=2$ or $r=\omega$. Suppose that there exists a self conformal measure $\nu$ that is  not pointwise absolutely normal or not Rajchman.  We can define a derivative cocycle $c'$ directly on $\mathcal{A} \times I$ by
$$c'(i,x)= -\log |f_i '(x)|.$$
Notice that, unlike the symbolic version we have been working with so far, here we need $\Phi$ to be at least $C^2$ so that the cocycle $c'$ has a finite moment in the sense of \cite[Equation (11.15)]{Benoist2016Quint}. This also implies for every fixed $i$ the cocycle $c'$ is $C^1$ and therefore Lipschitz continuous in $x\in I$. Let  $L$ denote the space of Lipschitz continuous maps $I \rightarrow \mathbb{C}$, $K$ be the attractor of $\Phi$,  and define
\begin{eqnarray*} 
\Lambda_{c'} &=& \lbrace \theta: \exists \phi_\theta \in L \text{ with } |\phi_\theta|=1 \text{ and } u_\theta \in S^1 \text{ such that } \forall (a,x)\in \mathcal{A} \times K, \\
&& \quad \quad \quad \phi_\theta(f_a(x)) = u_\theta \exp(-i \theta \cdot c'(a,x))\cdot \phi_\theta (x) \rbrace. \nonumber
\end{eqnarray*}
\begin{Claim}
The assumptions of Corollary \ref{Main Corollary} part (3) imply that   $\Lambda_{c'} \neq \lbrace 0 \rbrace$.
\end{Claim}
\begin{proof}
If $\Lambda_{c'} = \lbrace 0 \rbrace$ then the cocycle $c'$ is aperiodic. Therefore Benoist and Quint's  local limit Theorem for cocycles \cite[Theorem 16.15]{Benoist2016Quint} applies to $c'$. The crucial observation here is that, starting from a given point $x_0 \in I$, the $n$-step random walk driven by the cocycle $c$ and the $n$-step random walk driven by the cocycle $c'$ have exactly the same law, which is the corresponding push-forward of the distribution of $\mathbb{P}$ on the first $n$ digits. It follows that \cite[Theorem 16.15]{Benoist2016Quint} applies to $c$ as well, as long as the target $\varphi=1_{A_\eta}$ for some  cylinder  $A_\eta$ (to see this, just work with the the conjugated IFS $f_\eta \circ \Phi \circ f_\eta ^{-1}$ for which $f_\eta \nu$ is self conformal). Therefore, Theorem \ref{B-Q LLT general}  holds true as stated. Since this is the only use we make of the assumption $\Lambda_c = \lbrace 0 \rbrace$, it follows from the proof of Theorem \ref{Theorem main tech} that $\nu$ is both Rajchman and pointwise absolutely normal. This contradicts our assumptions.
\end{proof}

So, there exists $0 \neq \theta \in \Lambda_{c'}$.  The next Lemma now follows via a standard argument:
\begin{Lemma} \label{Lemma varphi}
If $\Phi$ is $C^2$  then there exists a 
$C^1$ function $\varphi:I \rightarrow \mathbb{R}$ and some $\alpha \in \mathbb{R}$  such that for every $(i,x)\in \mathcal{A}\times K$,
\begin{equation*}
\varphi( f_{i} (x)) = \alpha+ \theta\cdot (-\log |f_i '(x)|)+ \varphi (x) \mod 1.
\end{equation*}
Furthermore, if $\Phi$ is $C^\omega$ then we may assume $\varphi$ is also $C^\omega$.
\end{Lemma}

The next Lemma is where the assumption that $K$ is an interval when $\Phi$ is $C^2$ comes into play:
\begin{Lemma} \label{Lemma cohom on J}
Let $\varphi$ and $\alpha$ be as in Lemma \ref{Lemma varphi}. If $\Phi$ is $C^{2}$ and $K$ is an interval then for every $1\leq i \leq n$ there is some $n_i\in \mathbb{Z}$ such that for every $x\in I$
\begin{equation*}
\varphi( f_{i} (x)) = \alpha+ \theta\cdot (-\log |f_i '(x)|)+ \varphi (x) +n_i.
\end{equation*}
If $\Phi$ is  $C^\omega$ then every $y\in K$ admits a neighbourhood $N_{y}$ in $I$ and $n_{y,i}\in \mathbb{Z}$ such that for every $x\in N_y$ and every $i$
$$\varphi( f_{i} (x)) = \alpha+\theta\cdot (-\log |f_i '(x)|)+ \varphi (x)+n_{y,i}.$$
\end{Lemma}
Notice that one difference between this Lemma and Lemma \ref{Lemma varphi} is the set on which the cohomological equation holds. 
\begin{proof}
First,  by Lemma \ref{Lemma varphi}, for every $1 \leq i \leq n$ and for all $x\in K$
$$ \varphi( f_{i} (x)) - \left( \alpha+ \theta\cdot (-\log |f_i '(x)|)+ \varphi (x) \right) \in \mathbb{Z}.$$
Assuming $K$ is an interval, the function on the left hand side is a continuous function taking values in $\mathbb{Z}$, so it must be constantly $n_i \in \mathbb{Z}$ on $K$.

 If $\Phi$ is $C^\omega$ then so is $\varphi$. So, the function on the left hand side in the last displayed equation is a $C^\omega$ function on $K$ that takes values in $\mathbb{Z}$. Since $K$ is compact and infinite, this Lemma follows.
\end{proof}

\textbf{Proof of Corollary \ref{Main Corollary} part (3)}
 Assume first that $\Phi$ is $C^{2}$ smooth and that $K$ is an interval. Let $\varphi$ be as in Lemma \ref{Lemma cohom on J}, and let $h:I\rightarrow \mathbb{R}$ be a $C^{2}$ smooth function  that is a primitive of $\exp(\frac{ \varphi (x)}{\theta})$ on $I$. Now, for every $i$ define
$$g_i (x) = h \circ f_i \circ h^{-1} : h(I) \rightarrow h(I)$$
and let $\Psi$ be the IFS consisting of the maps $g_i$. We claim that $\Psi$ is a periodic self similar IFS. Indeed,  by Lemma \ref{Lemma cohom on J}, for every $i$ and every $y\in h(I)$
\begin{eqnarray*}
g_i' (y)  &=& \left( h \circ f_i \circ h^{-1} \right)'(y) \\
&=& \frac{h' \left( f_i \circ h^{-1} (y) \right) \cdot  f_i' (h^{-1}(y)) }{h'(h^{-1}(y))}\\
&=& \exp\left( \frac{\varphi \circ f_i (h^{-1}(y))}{\theta} +\log |f_i '(h^{-1}(y))|)-\frac{\varphi(h^{-1}(y))}{\theta} \right) \cdot \sign \left( f_i ' (h^{-1}(y)) \right)  \\
&=& \exp\left( \frac{n_i}{\theta}+\frac{\alpha}{\theta} \right) \cdot \sign \left( f_i ' (h^{-1}(y)) \right). \\
\end{eqnarray*}
Since by uniform contraction $\sign \left( f_i ' (h^{-1}(y)) \right)$ is constant in $y$, $\Psi$ is a  self similar IFS. Finally, $\Psi$ must be periodic, since if it were aperiodic then by Corollary \ref{Main Corollary} part (2) the measure $\nu$ would be both pointwise absolutely normal and Rajchman, contradicting our assumptions.

If the IFS is $C^\omega$ smooth then the same proof shows that $h(K)$ can be covered by finitely many intervals on which every map in $\Psi$ acts as an affine map, with contraction ratios of the form $r^{n_i}$. We leave the verification to the reader.

\section{Proof of Theorem \ref{Theorem 1}} \label{Section proof thm 1}
\subsection{Proof of Theorem \ref{Prop key 1}}
In this Section we prove Theorem \ref{Prop key 1}, which is the key to the proof of Theorem \ref{Theorem 1}.  We follow the same notations as before:  Let  $\Phi = \lbrace f_1,...,f_n \rbrace$ be a self-similar IFS, and  let $\nu$ be a self similar measure. Let $p\geq 2$ be an integer, and define a stopping time 
$$ \beta_n (\omega):= \min \lbrace m: |f_{\omega|_m} ' (0)| < e^{-n\cdot \log p} \rbrace.$$
  The following  is a complete analogue of Theorem \ref{Theorem martingale}:
\begin{theorem} \label{Equi via martingale}
For $\mathbb{P}$ almost every $\omega$, for every integer $q$
\begin{equation*}
\lim_{N\rightarrow\infty} \mathcal{F}_q \left( \frac{1}{N}  \sum_{n=0} ^{N-1} \delta_{T_p ^n (x_\omega)} \right) - \mathcal{F}_q \left( \frac{1}{N} \sum_{n=0} ^{N-1} T_p ^n \circ f_{\omega|_{\beta_{n} (\omega)}} \nu \right)  =0.
\end{equation*}
\end{theorem}

Recall that our aim is to prove the following Theorem:
\begin{theorem} \label{Theorem normality}
If  $\nu$ is a Rajchman measure then it is pointwise normal to base $p$.
\end{theorem}

 We need the following key Proposition.
\begin{Proposition} \label{Prop key}
For every $\epsilon>0$ there is some $q_0=q_0(\epsilon)$ such that for all integer $|q|>q_0$, for $\mathbb{P}$ almost every $\omega$ and every $n$,
$$\left| F_q \left( T_p ^n \circ f_{\omega|_{\beta_n (\omega)}} \nu \right)  \right| < \epsilon.$$
\end{Proposition}
Proposition \ref{Prop key} implies Theorem \ref{Theorem normality}: By appealing first to Theorem \ref{Equi via martingale} and using the relation between $\mathbb{P}$ and $\nu$, this is completely analogous to the implication  Theorem \ref{Lemma sufficient} $\Rightarrow$ Theorem \ref{Main theorem}.

\noindent{\bf Proof of Proposition \ref{Prop key}} Let $\epsilon>0$ and let $q\in \mathbb{Z}$. Fix  $\omega$, a $\mathbb{P}$ typical point. For every $n$ and  $x \in K$, since $\Phi$ is a self similar IFS (and so we may assume all maps are defined on $\mathbb{R}$)
\begin{equation} \label{Eq linear}
T_p ^n \circ  f_{\omega|_{\beta_n(\omega)}} (x) = p^n \left(  f_{\omega|_{\beta_n(\omega)}} ' (0) \cdot x + f_{\omega|_{\beta_n(\omega)}} (0) \right)  - m_{x,n}, \quad \text{ where } m_{x,n}\in \mathbb{Z}.
\end{equation}
Note: There is some global $C_0>0$ such that $\left| p^n \cdot f_{\omega|_{\beta_n(\omega)}} ' (0) \right| \in [C_0,1]$ for all $n$ and $\omega$.
\newline Let $r(\omega,n) = p^n \cdot f_{\omega|_{\beta_n(\omega)}} ' (0)$. Then, by \eqref{Eq linear} and since $q\in \mathbb{Z}$ 
$$\left| F_q \left( T_p ^n \circ f_{\omega|_{\beta_n (\omega)}} \nu \right)  \right| = \left| F_{q\cdot r(\omega,n)}  \left( \nu \right)  \right|.$$
Since $r(\omega,n)$ is now fixed and its norm is in $[C_0,1]$, and since $\nu$ is assumed to be Rajchman, for all $|q|>q_0(\epsilon,C_0)=q_0(\epsilon)$ and every $\omega,n$
$$\left| F_{q\cdot r(\omega,n)}  \left( \nu \right)  \right| < \epsilon.$$
The proof is complete.

\subsection{Reduction to IFS's in integer form}
We begin the proof of Theorem \ref{Theorem 1}. Fix the self similar  IFS $\Phi = \lbrace f_i (x) = a_i x + b_i \rbrace$, where $a_i>0$ for all $i$. Recall the notation $\Phi_1 := \Phi$ and for every integer $m\geq 2$,
$$\Phi_m := \lbrace g: g= \phi_1 \circ ... \circ \phi_\ell, \text{ such that } \phi_i \in \Phi, g'(0)<\frac{1}{m} \text{ and }  \left( \phi_1 \circ ... \circ \phi_{\ell-1} \right)' (0) \geq \frac{1}{m} \rbrace$$
and that for every $m\in \mathbb{N}$, $K_{\Phi_m} = K_\Phi$, i.e. all these IFS's have the same attractor as $\Phi$. Let $K$ denote this common attractor. We require the following Claim:
\begin{Claim} \label{Claim var princ}
$\dim K = \sup \lbrace \dim \mu:\, \mu \text{ is a self similar measure w.r.t } \Phi_m, \, m\in \mathbb{N} \rbrace$.
\end{Claim}
Recall that in this paper self similar measures always correspond to a \textit{strictly} positive probability vector on the underlying IFS. While we have not been able to find the statement of Claim \ref{Claim var princ} in the literature, it can nonetheless be inferred from a combination of existing results. We thus postpone its proof to Section \ref{Section proof of var}.

It will be convenient to introduce the notation
$$ \mathbb{A} = \lbrace x:\, x \text{ is  absolutely normal } \rbrace.$$
So, if
$$\dim K\cap \mathbb{A}< \dim K $$
then by Claim \ref{Claim var princ} there is some IFS $\Phi_m$ and a self similar measure $\mu$ on it that is not pointwise absolutely normal. Without the loss of generality, we assume that $m=1$. That is, $\Phi$ already admits such a self similar measure. Notice that if there is a self similar measure $\mu$ with $\dim \mu = \dim K$ then $\mu$ cannot be pointwise absolutely normal, and we can work with this measure.

So, let $\mu$ be a non pointwise absolutely normal self similar measure. We first claim that it is impossible that there exist contraction ratios $a_i, a_j$ of maps in $\Phi$ such that $\frac{\log a_i}{\log a_j}\notin \mathbb{Q}$. Indeed, if this is  the case then it follows from \cite[Theorem 1.2]{li2019trigonometric} that $\mu$ is a Rajchman measure. So, by Theorem \ref{Prop key 1} $\mu$ is pointwise absolutely normal, a contradiction.

Therefore, we may assume that every $a_i$ is a rational power of some $r>0$. Since $\mu$ cannot be a Rajchman measure, following the work of Br\'{e}mont \cite[Theorem 2.3]{bremont2019rajchman}   the IFS $\Phi$ is affinely conjugated via a map $h$ to an IFS $\Psi = \lbrace g_i (x) = r_i x + t_i \rbrace $ that is in so-called Pisot form \cite[Definition 2.2]{bremont2019rajchman}. For the time being, we note that this means that for every $g_i\in \Psi$, $g_i(x)=r^{n_i} x + t_i$ where $r^{-1}$ is a Pisot number, and $(n_i)$ are relatively prime (the $t_i$ also have an explicit form, but we postpone discussion about this to the next section).

We next claim that $r^{-1}$ is an integer. Again, we argue via contradiction: Otherwise, $r^{-1}$ is a non-integer Pisot number, which in particular implies that $r^{-1}$ is independent of all integers $p\geq 2$. We require the following well known Lemma, which is adapted from e.g. the work of Shmerkin-Peres \cite[Proposition 6]{Peres2009Shmerkin} and Shmerkin \cite[Lemma 4.2]{shmerkin2015projections}:
\begin{Lemma} \label{Approx Lemma}
For every  $\epsilon>0$ there exists an IFS $\Lambda$ satisfying the following properties:
\begin{enumerate}
\item Every $f\in \Lambda$ is a composition of maps from $\Phi$. 

\item Every $f\in \Lambda$ has the same contraction ratio, which must be of the form $r^{-k}$ for some $k\in \mathbb{N}$ by the structure of $\Phi$ and (1).

\item $\Lambda$ satisfies the strong separation condition: The union
$$ K_\Lambda = \bigcup_{f\in \Lambda} f(K_\Lambda)$$
is disjoint. In particular, the IFS $\Lambda$ is regular.

\item $K_\Lambda \subseteq K$ and we have $\dim K \leq  \dim K_\Lambda + \epsilon$. 
\end{enumerate}
\end{Lemma}
Now, let $\epsilon>0$ and produce an IFS $\Lambda$ satisfying the conditions of Lemma \ref{Approx Lemma}. It is well known that for such homogeneous IFS's with separation  there exists a self similar measure $\nu$  on $K_\Lambda$ such that $\dim \nu = \dim K_\Lambda$. By \cite[Theorem 1.4]{hochmanshmerkin2015}, since the IFS $\Lambda$ is regular and $r^{-1}$ is independent of all integer $p\geq 2$, $\nu$ a.e. $x$ is normal to all integer bases, that is, $\nu(\mathbb{A})=1$. Since  $K_\Lambda \subseteq K$, we find that
$$\dim \mathbb{A}\cap K \geq \dim \mathbb{A}\cap K_\Lambda \geq \dim \nu = \dim K_\Lambda \geq \dim K -\epsilon.$$
Taking $\epsilon \rightarrow 0$ we find that $\dim \mathbb{A}\cap K  = \dim K$. This is a contradiction. We conclude that $r^{-1}$ has to be an integer.

\subsection{Structure of the conjugating map and conclusion of proof} \label{Section structure}
Let us recall what we have shown so far:  The IFS $\Phi$ is conjugated via an affine  map $h(x)=cx+d$ to an IFS $\Psi = \lbrace g_i (x) = r_i x + t_i \rbrace $ that is in so-called Pisot form \cite[Definition 2.2]{bremont2019rajchman} for some \textit{integer} Pisot number. This means that for every $g_i\in \Psi$, $g_i(x)=r^{n_i} x + t_i$ where $r^{-1}=n$ is an integer, the $(n_i)$ are relatively prime, and, since $r^{-1}$ is an integer, the translations $t_i$ have the form $t_i = \frac{z_i}{n ^{s_i}}$ where $z_i \in \mathbb{Z}$ and $s_i \in \mathbb{N}\cup \lbrace 0 \rbrace$.

In this section we complete the proof of Theorem \ref{Theorem 1} by showing that we must have $c\in \mathbb{Q}$ and that $d$ is not an $n$-normal number. As we will see, this implies that every translation in the original IFS is not $n$-normal. Afterwards, we show that under some extra assumptions on $\Phi$ each such translation must be rational. So,  let $X$ denote the attractor of $\Psi$. Since our IFS $\Phi$ is conjugated by the affine map $h(x)=cx+d$ to $\Psi$, we have
$$h(X)=K=K_\Phi.$$
Let $\epsilon>0$ and produce an IFS $\Lambda$ as in Lemma \ref{Approx Lemma} but for the IFS $\Psi$. In particular, 
$$\dim X_\Lambda \geq \dim X - \epsilon$$
 and all the maps in $\Lambda$ have the same contraction $n^{-k}$ for some $k\in \mathbb{N}$. Furthermore, since all the maps in $\Lambda$ are compositions of maps in $\Psi$,  the  translations of the maps in $\Lambda$ retain the structure of those in $\Psi$. We also note that $h(X_\Lambda)\subseteq K$. 
 
 Consider the conjugated IFS $\Theta = \lbrace h \circ l_i \circ h^{-1} \rbrace_{l_i \in \Lambda}$:  Every affine map $f_i \in \Theta$ has  the same contraction ratio $n^{-k}$, and its translation $f_i(0)=b_i$ satisfies, by the known structure of the maps in $\Lambda$,
$$b_i = \frac{c\cdot z_i}{n^{s_i}} + d-\frac{d}{n^k}, \quad z_i \in \mathbb{Z},s_i \in \mathbb{N}\cup \lbrace 0 \rbrace.$$
Now, since $\Phi$ is non-trivial we may assume that so is $\Psi$, and consequently $\Lambda$ is non-trivial (if $\epsilon$ is small enough). So, there exist maps $f_i,f_j\in \Theta$ with $b_i \neq b_j$. This implies that
$$0 \neq n^{k} \cdot b_i - n^{k} \cdot b_j = n^k \cdot c\cdot \left( \frac{z_i}{n^{s_i}} - \frac{z_j}{n^{s_j}} \right)  ,\quad z_i,z_j\in \mathbb{Z},s_i,s_j \in \mathbb{N}.$$

By the last displayed equation, if $c\notin \mathbb{Q}$ then 
\begin{equation} \label{Eq dayan}
\lbrace n^{k} b_i - n^{k} b_j:i,j\in \Phi \rbrace \text{ is not contained in a proper closed subgroup of } \mathbb{T}.
\end{equation}
Now, Let $\mu$ be a self similar measure on $h(X_\Lambda)$ with respect to the conjugated IFS $\Theta$, such that $\dim \mu = \dim X_\Lambda$. By \cite[Theorem 4]{dayan2020random}, which we may apply via \eqref{Eq dayan},  $\mu$  is pointwise $n$-normal. It is well known that this implies that $\mu$ is pointwise normal to all integer bases $p$ such that $p\sim n$ (see e.g. \cite{Schmidt1960normal}). We also note that $\Theta$ satisfies the open set condition: Since $\Lambda$ has the strong separation condition, $X_\Lambda$ has positive Hausdorff measure in its dimension. Therefore, the attractor of $\Theta$, $h(X_\Lambda)$, also has positive Hausdorff measure in its dimension. Therefore, by \cite[Theorem 1.1]{Peres2001Simon}, $\Theta$ satisfies the open set condition. So, since all maps in the regular IFS $\Theta$ have contraction $\frac{1}{n^k}$, we may appeal to \cite[Theorem 1.4]{hochmanshmerkin2015} to see that $\mu$ is pointwise normal to all  bases $p\geq 2$ such that $p\not \sim n$. Therefore, $\mu(\mathbb{A})=1$, and we conclude that
$$\dim K\cap \mathbb{A}= \dim h(X)\cap  \mathbb{A} \geq  \dim h(X_\Lambda)\cap \mathbb{A} \geq \dim \mu = \dim X_\Lambda \geq \dim X-\epsilon=\dim K - \epsilon.$$ 
Taking $\epsilon \rightarrow 0$, we obtain a contradiction.

So far, we have established the structure of the conjugated IFS $\Psi$, and found that the conjugating map $h(x)= cx+d$ must have $c\in \mathbb{Q}$. We next show  that $d$ cannot be $n$-normal. If $d$ is $n$-normal, then every translation $b_i$ in the homogeneous IFS $\Theta$ as above must be $n$-normal: Indeed, we have seen that
$$b_i = \frac{c\cdot z_i}{n^{s_i}} + d(1- \frac{1}{n^k}) , \quad z_i \in \mathbb{Z},s_i \in \mathbb{N}\cup \lbrace 0 \rbrace$$
and so $b_i= s\cdot d + t$ where $s,t\in \mathbb{Q}$ and $s\neq 0$. This implies that $b_i$ is $n$-normal as proved by Wall in his Ph.D. thesis \cite{Wall1950thesis}.  We can now run the same argument as above, with the only difference being that since  $b_i - b_j \in \mathbb{Q}$ and $b_i$ is $n$-normal for all $i,j$,  we can use \cite[Theorem 7]{dayan2020random} instead of \cite[Theorem 4]{dayan2020random}  to conclude that $\mu$ is pointwise absolutely normal. We have just shown that this leads to a contradiction. So, $d$ cannot be $n$-normal.

Finally, we have shown that $\Phi = \lbrace h \circ g_i \circ h^{-1} \rbrace_{g_i \in \Psi}$ where $\Psi = \lbrace g_i (x) = r_i x + t_i \rbrace $ is such that for every $g_i\in \Psi$, $g_i(x)=\frac{x}{n^{k_i}} + t_i$ where  $(k_i)$ are relatively prime, and  $t_i = \frac{z_i}{n ^{s_i}}$ where $z_i \in \mathbb{Z}$ and $s_i \in \mathbb{N}\cup \lbrace 0 \rbrace$. We have also shown that $h$ is an affine map such that $h'(0)=c\in \mathbb{Q}$ and $h(0)=d$ is not $n$-normal. So, every map $h \circ g_i \circ h^{-1}$ in $\Phi$ is of the form
$$x\mapsto \frac{x}{n^{k_i}} + \frac{c\cdot z_i}{n ^{s_i}} +d - \frac{d}{n^{k_i}}.$$
Appealing to Wall's thesis \cite{Wall1950thesis} once more, since $d$ is not $n$ normal and $c\in \mathbb{Q}$, $\frac{c\cdot z_i}{n ^{s_i}} +d - \frac{d}{n^{k_i}}$ is also not $n$-normal.

If furthermore there are two maps in $\Phi$ with different contraction ratios $\frac{1}{n^{k_i}} \neq \frac{1}{n^{k_j}}$ then their translations $t_i,t_j$ satisfy that, since $c\in \mathbb{Q}$
$$n^{k_i} t_i - n^{k_j} t_j = q_{i,j} + \left( n^{k_i} - n^{k_j} \right) \cdot d, \text{ where } q_{i,j}\in \mathbb{Q}.$$
So, if $d\notin \mathbb{Q}$ then $n^{k_i} t_i - n^{k_j} t_j \notin \mathbb{Q}$ and thus
\begin{equation} \label{Eq dayan 2}
\lbrace n^{k_i} t_i - n^{k_j} t_j:i,j\in \Phi \rbrace \text{ is not contained in a proper closed subgroup of } \mathbb{T}.
\end{equation}
Now, assuming  that $\Phi$ is regular, let $\mu$ be a self similar measure such that $\dim \mu = \dim K$. Then $\mu$ a.e. $x$ is $n$-normal by \cite[Theorem 4]{dayan2020random} which applies via \eqref{Eq dayan 2}. Also, $\mu$ a.e. $x$ is  $p$-normal for all $p\not \sim n$ via \cite[Theorem 1.4]{hochmanshmerkin2015}, since $\Phi$ is regular and all maps in $\Phi$ have contractions that are independent of $p$. It follows that $\mu(\mathbb{A})=1$ so $\dim K\cap \mathbb{A}=\dim K$, a contradiction. We conclude that in this situation $d\in \mathbb{Q}$, and with this the proof of Theorem \ref{Theorem 1} is done.

\subsection{Proof of Claim \ref{Claim var princ}} \label{Section proof of var}
We now prove Claim \ref{Claim var princ}. We work with the same notations introduced before Claim \ref{Claim var princ}. Recall that for $\mathbf{p}$, a probability vector on the IFS $\Phi$ (which is a finite set), $\mu_{\mathbf{p}}$ is the corresponding self similar measure. By combining the results of Peres-Solomyak \cite{Solomyak2000Peres} and of Feng-Hu \cite{feng2009dimension} (or by \cite{hochman2009local}), we see that the map $\mathbf{p} \mapsto \dim \mu_{\mathbf{p}}$ is lower semi-continuous. Notice that this  holds true even if some of the entries in $\mathbf{p}$ are zero. Therefore, if $\mathbf{p}_k \rightarrow \mathbf{p}$ then 
$$\liminf  \dim \mu_{\mathbf{p}_k} \geq \dim \mu_{\mathbf{p}}.$$

Now, let $\epsilon>0$. It follows from  \cite[Lemma 4.2]{shmerkin2015projections} that there exists an IFS $\Lambda$ with strong separation such that every $f\in \Lambda$ is a composition of maps from $\Phi$, and $K_\Lambda \subseteq K$ with $\dim K_\Lambda \geq \dim K -\epsilon$. Notice that unlike Lemma \ref{Approx Lemma}, here we do not require all the maps in $\Lambda$ to have the same contraction ratio.   It is a  consequence of the proof of  \cite[Lemma 4.2]{shmerkin2015projections}  that we can choose $\Lambda$ by taking $m$ large enough and choosing a subset of the maps that make up the IFS $\Phi_m$. That is, $\Lambda \subseteq \Phi_m$ for some $m$.

Since $\Lambda$ has strong separation, we can always find a self similar measure $\mu_{\mathbf{p}}$ such that $\dim K_\Lambda = \dim \mu_{\mathbf{p}}$. Since $\Lambda \subseteq \Phi_m$, we can find a sequence $\mathbf{p}_k \rightarrow \mathbf{p}$ such that each $\mathbf{p}_k$ is a strictly positive probability vector on $\Phi_m$. By the lower semi-continuity alluded to in the first paragraph, for every $k$ large enough we thus have
$$\dim \mu_{\mathbf{p}_k}  \geq\dim \mu_{\mathbf{p}} - \epsilon = \dim K_\Lambda - \epsilon \geq \dim K - 2 \epsilon$$
which implies Claim \ref{Claim var princ}.

\bibliography{bib}{}
\bibliographystyle{plain}

\end{document}